\documentclass[final]{siamltex}

\usepackage[T1]{fontenc}
\usepackage[latin1]{inputenc}
\usepackage[english]{babel}

\usepackage[pdftex]{graphicx}
\usepackage[caption=false]{subfig}
\usepackage{cite}
\usepackage{url}

\usepackage{amsfonts, amsbsy, latexsym, color}
\usepackage{amssymb,amsmath} %,amsthm}
\newtheorem{example}{Example}

\newcommand\cC{{\mathcal{C}}}

\newcommand\cE{{\mathcal{E}}}

\newcommand\cP{{\mathcal{P}}}

\newcommand\cH{{\mathcal{H}}}

\newcommand\cQ{{\mathcal{Q}}}
\newcommand\cS{{\mathcal{S}}}

\newcommand\SH{SH}%{{\mathcal{S \!\! H}}}
\newcommand{\ip}[2]{\langle#1,#2\rangle}
\DeclareMathOperator{\intt}{int}

\DeclareMathOperator{\volume}{vol} % 
\DeclareMathOperator*{\esssup}{ess\,sup} %
\DeclareMathOperator*{\essinf}{ess\,inf} %

\newcommand{\Frame}{S}

\newcommand{\expo}[1]{\mathrm{e}^{#1}}

\newcommand{\inv}[2][]{\ensuremath{#2^{#1-1}}}
\newcommand{\vol}[1]{\volume{(#1)}}

\newcommand{\card}[1]{\# \abs{#1}}

\newcommand{\cardsmall}[1]{\# \abssmall{#1}}
\newcommand{\charfct}[1]{\ensuremath{\chi_{#1}}}

\newcommand{\eps}{\ensuremath{\varepsilon}}

\newcommand*{\numbersys}[1]{\ensuremath{\mathbb{#1}}} 
\newcommand*{\C}{\numbersys{C}}
\newcommand*{\R}{\numbersys{R}}
\newcommand*{\Rn}{\numbersys{R}^n}

\newcommand*{\Z}{\numbersys{Z}}

\newcommand*{\N}{\numbersys{N}}

\newcommand{\itvoc}[2]{\ensuremath{\left({#1},{#2}\right]}} % InTerValOpenClosed
\newcommand{\itvoo}[2]{\ensuremath{\left({#1},{#2}\right)}} % 
\newcommand{\itvcc}[2]{\ensuremath{\left[{#1},{#2}\right]}} % 
\newcommand{\itvco}[2]{\ensuremath{\left[{#1},{#2}\right)}} % 
\newcommand{\abs}[1]{\ensuremath{\left\lvert#1\right\rvert}}
\newcommand{\abssmall}[1]{\ensuremath{\lvert#1\rvert}}
\newcommand{\absbig}[1]{\ensuremath{\bigl\lvert#1\bigr\rvert}}
\newcommand{\absBig}[1]{\ensuremath{\Bigl\lvert#1\Bigr\rvert}}
\newcommand{\norm}[2][]{\ensuremath{\left\lVert#2\right\rVert_{#1}}}

\newcommand{\normBig}[2][]{\ensuremath{\Bigl\lVert#2\Bigr\rVert_{#1}}} 
\newcommand{\normsmall}[2][]{\ensuremath{\lVert#2\rVert_{#1}}} 
\newcommand{\innerprod}[3][]{\ensuremath{\left\langle #2,#3\right\rangle_{\! #1}}}

\newcommand{\innerprodbig}[3][]{\ensuremath{\bigl \langle #2,#3\bigr\rangle_{\!\!#1}}}
\newcommand{\innerprods}[2]{\ensuremath{\langle #1,#2\rangle}}
\newcommand{\set}[1]{\ensuremath{\left\lbrace{#1}\right\rbrace}}
\newcommand{\setprop}[2]{\ensuremath{\left\lbrace{#1} : {#2}\right\rbrace}}
\newcommand{\setbig}[1]{\ensuremath{\bigl\lbrace{#1}\bigr\rbrace}}
\newcommand{\setsmall}[1]{\ensuremath{\lbrace{#1}\rbrace}}

\newcommand{\setpropsmall}[2]{\ensuremath{\lbrace{#1} : {#2}\rbrace}}
\newcommand{\ceil}[1]{\left\lceil #1 \right\rceil}
\newcommand{\ceilsmall}[1]{\lceil #1 \rceil}
\newcommand{\floor}[1]{\left\lfloor #1 \right\rfloor}

\newcommand{\scfac}{2^{j\frac{\alpha+2}{4}}}
\newcommand{\pa}{\partial}
\newcommand{\scp}{\ensuremath\alpha}
\newcommand{\vap}{\ensuremath\delta}
\newcommand{\dep}{\ensuremath\gamma}
\newcommand{\cp}{\ensuremath\nu}

\usepackage{xspace}
\newcommand{\ie}{i.e.,\xspace} %{\textit{i.e.,\xspace}}
\newcommand{\eg}{e.g.,\xspace} %{\textit{e.g.\@\xspace}}

\usepackage[colorlinks=true, linkcolor=black, citecolor=black,filecolor=black, urlcolor=black, bookmarks=true,bookmarksopen=true, bookmarksopenlevel=3, plainpages=false,pdfpagelabels=true]{hyperref}
\hypersetup{
pdfview={FitH},
pdfstartview={FitH},
pdfauthor = {Gitta Kutyniok, Jakob Lemvig, and Wang-Q Lim}
pdftitle = {Optimally sparse approximations of 3D functions by compactly supported shearlet frames},
pdfsubject = {MSC Primary: 42C40, Secondary: 42C15, 41A30, 94A08},
pdfkeywords = {anisotropic features, multi-dimensional data, shearlets, cartoon-like images, non-linear
  approximations, sparse approximations},
pdfcreator = {LaTeX with hyperref package},
pdfproducer = {pdftex}}

\title{Optimally sparse approximations of 3D functions by compactly supported shearlet frames}

\author{Gitta Kutyniok\thanks{Technische Universit\"at Berlin, Institut f\"ur Mathematik,  10623 Berlin, Germany, E-mail: \url{kutyniok@math.tu-berlin.de}} \and Jakob
  Lemvig\thanks{Technical University of Denmark, Department of Mathematics, Matematiktorvet 303, 2800 Kgs. Lyngby, Denmark,
    E-mail: \url{J.Lemvig@mat.dtu.dk}} \and Wang-Q Lim\thanks{Technische Universit\"at Berlin, Institut f\"ur Mathematik,  10623 Berlin, Germany, E-mail: \url{lim@math.tu-berlin.de}} }

\begin{document}

\maketitle

\begin{abstract}
  We study efficient and reliable methods of capturing and sparsely representing an\-iso\-tro\-pic
  structures in 3D data. As a model class for multidimensional data with anisotropic features, we
  introduce generalized three-dimensional cartoon-like images. This function class will have two
  smoothness parameters: one parameter $\beta$ controlling classical smoothness and one parameter
  $\alpha$ controlling anisotropic smoothness. The class then consists of piecewise $C^\beta$-smooth
  functions with discontinuities on a piecewise $C^\alpha$-smooth surface. We introduce a
  pyramid-adapted, hybrid shearlet system for the three-dimensional setting and construct frames for
  $L^2(\R^3)$ with this particular shearlet structure. For the smoothness range $1<\alpha\le
  \beta\le 2$ we show that pyramid-adapted shearlet systems provide a nearly optimally sparse
  approximation rate within the generalized cartoon-like image model class measured by means of
  non-linear $N$-term approximations.
\end{abstract}

\begin{keywords}
  anisotropic features, multi-dimensional data, shearlets, cartoon-like images, non-linear
  approximations, sparse approximations
\end{keywords}

\begin{AMS}
  Primary: 42C40, Secondary: 42C15, 41A30, 94A08
\end{AMS}

\pagestyle{myheadings} \thispagestyle{plain} \markboth{G. KUTYNIOK, J. LEMVIG, AND W.-Q
  LIM}{APPROXIMATIONS OF 3D FUNCTIONS BY COMPACTLY SUPPORTED SHEARLETS}

%*******************************************************************************************

\section{Introduction}
\label{sec:introduction}
Recent advances in modern technology have created a new world of huge, multi-dimensional data. In
biomedical imaging, seismic imaging, astronomical imaging, computer vision, and video processing,
the capabilities of modern computers and high-precision measuring devices have generated 2D, 3D and
even higher dimensional data sets of sizes that were infeasible just a few years ago. The need to
efficiently handle such diverse types and huge amounts of data has initiated an intense study in
developing efficient multivariate encoding methodologies in the applied harmonic analysis research
community. In neuro-imaging, \eg fluorescence microscopy scans of living cells, the discontinuity
curves and surfaces of the data are important specific features since one often wants to distinguish
between the image ``objects'' and the ``background'', \eg to distinguish actin filaments in eukaryotic
cells; that is, it is important to precisely capture the edges of these 1D and 2D structures. This
specific application is an illustration that important classes of multivariate problems are governed
by \emph{anisotropic features}. The anisotropic structures can be distinguished by location
\emph{and} orientation or direction which indicates that our way of analyzing and representing the
data should capture not only location, but also directional information. This is exactly the idea
behind so-called directional representation systems which by now are well developed and understood
for the 2D setting. Since much of the data acquired in, \eg neuro-imaging, are truly three-dimensional,
analyzing such data should be performed by three-dimensional directional representation systems.
Hence, in this paper, we therefore aim for the 3D setting.

In applied harmonic analysis the data is typically modeled in a continuum setting as
square-integrable functions or distributions. In dimension two, to analyze the ability of
representation systems to reliably capture and sparsely represent anisotropic structures, Cand\'es
and Donoho~\cite{CD04} introduced the model situation of so-called cartoon-like images, \ie
two-dimensional functions which are piecewise $C^2$-smooth apart from a piecewise $C^2$
discontinuity curve. Within this model class there is an optimal sparse approximation rate one can
obtain for a large class of non-adaptive and adaptive representation systems. Intuitively, one
should think \emph{adaptive} systems would be far superior in this task, but it has
been shown in recent years that non-adaptive methods using curvelets, contourlets, and shearlets all
have the ability to essentially optimal sparsely approximate cartoon-like images in 2D measured by
the $L^2$-error of the best $N$-term approximation \cite{CD04, GL07, DV05, KL10}.

\subsection{Dimension three}
\label{sec:dimension-three}
In the present paper we will consider sparse approximations of cartoon-like images using shearlets
in dimension \emph{three}. The step from the one-dimensional setting to the two-dimensional setting
is necessary for the appearance of anisotropic features at all. When further passing from the
two-dimensional setting to the three-dimensional setting, the complexity of anisotropic structures
changes significantly. In 2D one ``only'' has to handle one type of anisotropic features, namely
curves, whereas in 3D one has to handle \emph{two} geometrically very different anisotropic
structures: Curves as one-dimensional features and surfaces as two-dimensional anisotropic features.
Moreover, the analysis of sparse approximations in dimension two depends heavily on reducing the
analysis to affine subspaces of $\R^2$. Clearly, these subspaces always have dimension and
co-dimension one in 2D. In dimension three, however, we have subspaces of co-dimension one and two,
and one therefore needs to perform the analysis on subspaces of the ``correct'' co-dimension.
Therefore, the 3D analysis requires fundamental new ideas. 

Finally, we remark that even though the present paper only deals with the construction of shearlet
frames for $L^2(\R^3)$ and sparse approximations of such, it also illustrates how many of the
problems that arises when passing to higher dimensions can be handled. Hence, once it is known how
to handle anisotropic features of different dimensions in 3D, the step from 3D to 4D can be dealt
with in a similar way as also the extension to even higher dimensions. Therefore the extension of
the presented result in $L^2(\R^3)$ to higher dimensions $L^2(\R^n)$ should be, if not
straightforward, then at least be achievable by the methodologies developed.

\subsection{Modelling anisotropic features}
\label{sec:modell-anis-feat}
The class of 2D cartoon-like images consists, as mentioned above, of piecewise $C^2$-smooth
functions with discontinuities on a piecewise $C^2$-smooth curve, and this class has been
investigated in a number of recent publications. The obvious extension to the 3D setting is to
consider functions of three variables being piecewise $C^2$-smooth function with discontinuities on
a piecewise $C^2$-smooth surface. In some applications the $C^2$-smoothness requirement is too
strict, and we will, therefore, go one step further and consider a larger class of images also
containing less regular images. The generalized class of cartoon-like images in 3D considered in
this paper consists of three-dimensional piecewise $C^\beta$-smooth functions with discontinuities
on a piecewise $C^\alpha$ surface for $\alpha \in \itvoc{1}{2}$. Clearly, this model provides us
with two new smoothness parameters: $\beta$ being a classical smoothness parameter and $\alpha$
being an anisotropic smoothness parameter, see Figure~\ref{fig:cartoon-piecewise} for an illustration. 
\begin{figure}[ht]
\centering
\includegraphics{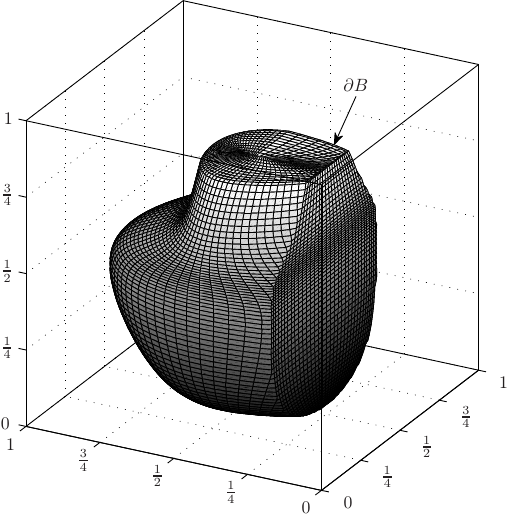}
\caption{The support of a 3D cartoon-like image $f=f_0 \chi_B$, where
  $f_0$ is $C^\beta$ smooth with $\supp f_0 = \R^3$ and the
  discontinuity surface $\pa B$ is piecewise $C^\alpha$ smooth.}
 \label{fig:cartoon-piecewise}
\end{figure}
This image class is unfortunately not a linear space as
traditional smoothness spaces, \eg H\"older, Besov, or Sobolev spaces, but it allows one to study
the quality of the performance of representation systems with respect to capturing anisotropic
features, something that is not possible with traditional smoothness spaces.

Finally, we mention that allowing \emph{piecewise} $C^\alpha$-smoothness and not everywhere $C^\alpha$-smoothness is an essential way to model singularities along surfaces \emph{as well as}
along curves which we already described as the two fundamental types of anisotropic phenomena in 3D.

\subsection{Measure for Sparse Approximation and Optimality}
\label{sec:meas-sparse-appr}

The quality of the performance of a representation system with respect to cartoon-like images is
typically measured by taking a non-linear approximation viewpoint. More precisely, given a
cartoon-like image and a representation system, the chosen measure is the asymptotic behavior of the
$L^2$ error of $N$-term (non-linear) approximations in the number of terms $N$. When the anisotropic
smoothness $\alpha$ is bounded by the classical smoothness as $\alpha \le \frac{4}{3} \beta$, the
anisotropic smoothness of the cartoon-like images will be the determining factor for the optimal
approximation error rate one can obtain. To be more precise, as we will show in
Section~\ref{sec:optimal-sparsity}, the optimal approximation rate for the generalized 3D
cartoon-like images models $f$ which can be achieved for a large class of adaptive and non-adaptive
representation systems for $1 < \alpha \le \beta \le 2$ is
\[
\norm[L^2]{f-f_N}^2 \le C \cdot N^{-\alpha/2} \qquad \text{as } N \to \infty,
\]
for some constant $C>0$, where $f_N$ is an $N$-term approximation of $f$. For cartoon-like images,
wavelet and Fourier methods will typically have an $N$-term approximation error rate decaying as
$N^{-1/2}$ and $N^{-1/3}$ as $N \to \infty$, respectively, see \cite{KLL11}. Hence, as the
anisotropic smoothness parameter $\alpha$ grows, the approximation quality of traditional tools
becomes increasingly inferior as they will deliver approximation error rates that are \emph{far
  from} the optimal rate $N^{-\alpha/2}$. Therefore, it is desirable and necessary to search for new
representation systems that can provide us with representations with a more optimal rate. This is where
pyramid-adapted, hybrid shearlet systems enter the scene. As we will see in
Section~\ref{sec:optimal-sparsity-3d}, this type of representation
system 
provides nearly optimally sparse approximations:
\begin{equation*}
  \norm[L^2]{f-f_N}^2 \le
  \left\{\begin{aligned} 
      C \cdot N^{-\scp/2+\tau}, &\qquad    \text{if }\beta \in \itvco{\alpha}{2},\\
      C \cdot N^{-1}(\log{N})^2 , &\qquad  \text{if }  \beta=\alpha=2,
    \end{aligned}
  \right\} \quad 
  \text{ as $N \to \infty$,} 
\end{equation*}
where $f_N$ is the $N$-term approximation obtained by keeping the $N$ largest shearlet coefficients,
and $\tau = \tau(\alpha)$ with $0\le \tau<0.04$ and $\tau \to 0$ for $\alpha \to 1^+$ and for $\alpha
\to 2^-$. Clearly, the obtained sparse approximations for these shearlet systems are not truly
optimal owing to the polynomial factor $\tau$ for $\alpha<2$ and the polylog factor for $\alpha=2$.
On the other hand, it still shows that non-adaptive schemes such as the hybrid shearlet
system can provide rates that are nearly optimal within a large class of adaptive and
non-adaptive methods.

\subsection{Construction of 3D hybrid shearlets}
\label{sec:constr-3d-shearl}

Shearlet theory has become a central tool in analyzing and representing 2D data with anisotropic
features. Shearlet systems are systems of functions generated by one single generator with parabolic
scaling, shearing, and translation operators applied to it, in much the same way wavelet systems are
dyadic scalings and translations of a single function, but including a directionality characteristic
owing to the additional shearing operation and the anisotropic scaling. Of the many directional
representation systems proposed in the last decade, \eg steerable pyramid
transform~\cite{steer1992}, directional filter banks~\cite{directional1992}, 2D directional
wavelets~\cite{twoDwavelet1993}, curvelets~\cite{CD99}, contourlets~\cite{DV05},
bandelets~\cite{bandelets}, the shearlet system~\cite{LLKW05} is among the most versatile and
successful. The reason for this being an extensive list of desirable properties: Shearlet
systems can be generated by one function, they precisely resolve wavefront sets, they
allow compactly supported analyzing elements, they are associated with fast decomposition
algorithms, and they provide a unified treatment of the continuum and the digital realm. 
We refer to \cite{KLL10} for a detailed review of the advantages and disadvantages of
shearlet systems as opposed to other directional representation systems.

Several constructions of discrete band-limited and compactly supported 2D shearlet frames are already known,
see \cite{GKL06,KL07,DKST09,Lim09,KKL10a,DST11}; for construction of 3D shearlet frames less is
known. Dahlke, Steidl, and Teschke~\cite{DST09} recently generalized the shearlet group and the
associated continuous shearlet transform to higher dimensions $\R^n$. Furthermore, in \cite{DST09}
they showed that, for certain band-limited generators, the continuous shearlet transform is able to
identify hyperplane and tetrahedron singularities. Since this transform originates from a unitary
group representation, it is not able to capture all directions, in particular, it will not capture
the delta distribution on the $x_1$-axis (and more generally, any singularity with
``$x_1$-directions''). We will use a different tiling of the frequency space, namely systems adapted
to pyramids in frequency space, to avoid this non-uniformity of directions. We call these systems 
pyramid-adapted shearlet system\cite{KLL10}. In \cite{GL10}, the continuous version of the
pyramid-adapted shearlet system was introduced, and it was shown that the location and the local
orientation of the boundary set of certain three-dimensional solid regions can be precisely
identified by this continuous shearlet transform. Finally, we will also need to use a different
scaling than the one from \cite{DST09} in order to achieve shearlet systems that provide almost
optimally sparse approximations.

Since spatial localization of the analyzing elements of the encoding system is very important both
for a precise detection of geometric features as well as for a fast decomposition algorithm, we will
mainly follow the sufficient conditions for and construction of compactly supported cone-adapted 2D
shearlets by Kittipoom and two of the authors \cite{KKL10a} and extend these result to the 3D
setting (Section~\ref{sec:shearl-high-dimens}). These results provide us with a large class of
separable, compactly supported shearlet systems with ``good'' frame bounds, optimally sparse
approximation properties, and associated numerically stable algorithms. One important new aspect is
that dilation will depend on the smoothness parameter $\alpha$. This will provide us with
\emph{hybrid} shearlet systems ranging from classical parabolic based shearlet systems ($\alpha =2$)
to almost classical wavelet systems ($\alpha \approx 1$). In other words, we obtain a parametrized
family of shearlets with a smooth transition from (nearly) wavelets to shearlets. This will allow us
to adjust our shearlet system according to the anisotropic smoothness of the data at hand. For
rational values of $\alpha$ we can associate this hybrid system with a fast decomposition algorithm
using the fast Fourier transform with multiplication and periodization in the frequency space (in
place of convolution and down-sampling).

Our compactly supported 3D hybrid shearlet elements (introduced in
Section~\ref{sec:shearl-high-dimens}) will in the spatial domain be of size $2^{-j\scp/2}$ times
$2^{-j/2}$ times $2^{-j/2}$ for some fixed anisotropy parameter $1<\scp\le 2$. When $\scp \approx 1$
this corresponds to ``cube-like'' (or ``wavelet-like'') elements. As $\scp$ approaches $2$ the scaling
becomes less and less $isotropic$ yielding ``plate-like'' elements as $j \to \infty$. This indicates
that these anisotropic 3D shearlet systems have been designed to efficiently capture two-dimensional
anisotropic structures, but neglecting one-dimensional structures. Nonetheless, these
3D shearlet systems still perform optimally when representing and analyzing cartoon-like functions
that have discontinuities on \emph{piecewise} $C^\alpha$-smooth surfaces -- as mentioned such
functions model 3D data that contain both point, curve, and surface singularities.

Let us end this subsection with a general thought on the construction of band-limited tight shearlet
frames versus compactly supported shearlet frames. There seem to be a trade-off between
\emph{compact support} of the shearlet generators, \emph{tightness} of the associated frame, and
\emph{separability} of the shearlet generators. The known constructions of tight shearlet frames,
even in 2D, do not use separable generators, and these constructions can be shown to \emph{not} be
applicable to compactly supported generators. Moreover, these tight frames use a modified version of
the pyramid-adapted shearlet system in which not all elements are dilates, shears, and translations
of a single function. Tightness is difficult to obtain while allowing for compactly supported
generators, but we can gain separability as in Theorem~\ref{thm:compact-for-pyramid} hence fast
algorithmic realizations. On the other hand, when allowing non-compactly supported generators,
tightness is possible, but separability seems to be out of reach, which makes fast algorithmic
realizations very difficult.

\subsection{Other approaches for 3D data}
\label{sec:other-approaches-3d}

Other directional representation systems have been considered for the 3D setting. We mention
curvelets~\cite{CDDY06,borup-nielsen}, surflets~\cite{surflets09}, and surfacelets~\cite{Lu-Do}.
This line of research is mostly concerned with constructions of such systems and not their sparse
approximation properties with respect to cartoon-like images. 
In \cite{surflets09}, however, the authors consider adaptive approximations of Horizon class function using
surflet dictionaries which generalizes the wedgelet dictionary for 2D signals to higher dimensions.

During the final stages of this project, we realized that a similar almost optimal sparsity result
for the 3D setting (for the model case $\alpha=\beta=2$) was reported by Guo and
Labate~\cite{GL10_3d} using \emph{band-limited} shearlet tight frames. They provide a proof for the
case where the discontinuity surface is (non-piecewise) $C^2$-smooth using the X-ray transform.

\subsection{Outline}
\label{sec:outline}
We give the precise definition of generalized cartoon-like image model class in
Section~\ref{sec:cartoon}, and the optimal rate of approximation within this model is then derived in
Section~\ref{sec:optimal-sparsity}. In Section~\ref{sec:shearl-high-dimens} and
Section~\ref{sec:constr-comp-supp} we construct the so-called pyramid-adapted shearlet frames with
compactly supported generators. In Sections~\ref{sec:optimal-sparsity-3d} to \ref{sec:proof-theorem-1} we then prove that
such shearlet systems indeed deliver nearly optimal sparse approximations of three-dimensional
cartoon-like images. We extend this result to the situation of discontinuity surfaces which are
\emph{piecewise} $C^\alpha$-smooth except for zero- and one-dimensional singularities and again derive essential optimal
sparsity of the constructed shearlet frames in Section~\ref{sec:proof-theorem-2}. We end the paper
by discussion various possible extensions in Section~\ref{sec:extensions}.

\subsection{Notation}
\label{sec:notation}
We end this introduction by reviewing some basic definitions. The following definitions will mostly
be used for the case $n=3$, but they will however be defined for general $n\in \N$. For $x\in \Rn$
we denote the $p$-norm on $\R^n$ of $x$ by $\norm[p]{x}$.
The Lebesgue measure on
$\Rn$ is denoted by $\abs{\cdot}$ and the counting measure by $\card{\cdot}$. Sets in $\Rn$ are
either considered equal if they are equal up to sets of measure zero or if they are element-wise
equal; it will always be clear from the context which definition is used. The $L^p$-norm of $f \in
L^p(\Rn)$ is denoted by $\norm[L^p]{f}$. For $f \in L^1(\Rn)$, the Fourier transform is defined by
\[
\hat f(\xi) = \int_{\Rn} f(x)\,\expo{-2 \pi i
  \innerprod{\xi}{x}} \,\mathrm{d}x
  \]
  with the usual extension to $L^2(\Rn)$. The Sobolev space and norm are defined as
\[    H^{s} (\R^{n}) = \setprop{f \colon \R^{n} \to \C}{\norm[H^{s}]{f}^2 := \int_{\R^{n}} \big( 1 + | \xi |^{2 }
\big)^{s} \big| \hat{f} (\xi) \big|^{2} \, \mathrm{d} \xi < + \infty}. 
\]
For functions $f : \R^n \to \C$ the homogeneous H\"older seminorm is given by
 \begin{equation*} %\label{eq:holder-norm}
\norm[\dot{C}^\beta]{f}:=\max_{\abs{\gamma} = \floor{\beta}} \;
\sup_{x,x' \in \R^n} \frac{\abs{\pa^{\gamma}
      f(x)-\pa^{\gamma}
      f(x')}}{\norm[2]{x-x'}^{\{\beta\}}}, 
\end{equation*}
where $\{\beta\} = \beta-\floor{\beta}$ is the fractional part of $\beta$ and $\abs{\gamma}$ is the usual length of a
multi-index $\gamma=(\gamma_1,\gamma_2,\dots,\gamma_n)$. Further, we let
\[\norm[{C}^\beta]{f}:= \max_{\gamma \le \floor{\beta}}
\sup_{}{\abs{\pa^\gamma f}} + \norm[\dot{C}^\beta]{f},
\]
and we denote by $C^\beta(\R^n)$ the space of H\"older functions, \ie functions $f : \R^n \to \C$, whose $C^\beta$-norm is
bounded.

\section{Generalized 3D cartoon-like image model class}
\label{sec:cartoon}
The first complete model of 2D cartoon-like images was introduced in \cite{CD04}, the basic idea
being that a closed $C^2$-curve separates two $C^2$-smooth functions. For 3D cartoon-like images we
consider square integrable functions of three variables that are piecewise $C^\beta$-smooth with
discontinuities on a piecewise $C^\scp$-smooth surface. 

Fix $\alpha>0$ and $\beta >0$, and let $\rho: \itvco{0}{2\pi} \times \itvcc{0}{\pi} \to
\itvco{0}{\infty}$ be continuous and define the set $B$ in $\R^3$ by
\[
B = \{x \in \R^3 : \norm[2]{x} \le \rho(\theta_1,\theta_2), x =
(\norm[2]{x},\theta_1,\theta_2) \text{ in spherical coordinates}\}.
\]
We require that the boundary $\partial B$ of $B$ is a closed surface parametrized by
\begin{equation}\label{eq:curve}
  b(\theta_1,\theta_2) = \begin{pmatrix} \rho(\theta_1,\theta_2)\cos(\theta_1)\sin(\theta_2) \\
    \rho(\theta_1,\theta_2)\sin(\theta_1)\sin(\theta_2) \\
    \rho(\theta_1,\theta_2)\cos(\theta_2)\end{pmatrix},
  \quad \theta = (\theta_1,\theta_2) \in \itvco{0}{2\pi} \times \itvcc{0}{\pi}. 
\end{equation}
Furthermore, the radius function $\rho$ must be H\" older continuous with coefficient $\nu$, \ie
\begin{equation}\label{eq:curvebound}
\norm[\dot{C}^\alpha]{\rho}=\max_{\abs{\gamma} = \floor{\alpha}} \; \sup_{\theta,\theta'} \frac{\abs{\pa^{\gamma}
      \rho(\theta)-\pa^{\gamma}
      \rho(\theta')}}{\norm[2]{\theta-\theta'}^{\{\alpha\}}}
 \leq \cp, \quad \rho=\rho(\theta_1,\theta_2), \quad \rho \leq \rho_0 < 1.
\end{equation}

For $\cp > 0$, the set $\mathit{STAR}^\alpha(\cp)$ is defined to be the set of all $B \subset
\itvcc{0}{1}^3$ such that $B$ is a translate of a set obeying \eqref{eq:curve} and
\eqref{eq:curvebound}. The boundary of the surfaces in $\mathit{STAR}^\alpha(\cp)$ will be the
discontinuity sets of our cartoon-like images. We remark that any starshaped sets in
$\itvcc{0}{1}^3$ with bounded principal curvatures will belong to $\mathit{STAR}^2(\cp)$ for some
$\cp$. Actually, the property that the sets in $\mathit{STAR}^\alpha(\cp)$ are parametrized by
spherical angles, which implies that the sets are starshaped, is not important to us. For $\alpha=2$
we could, \eg extend $\mathit{STAR}^2(\cp)$ to be all bounded subset of $\itvcc{0}{1}^3$, whose
boundary is a closed $C^2$ surface with principal curvatures bounded by $\cp$.

To allow more general discontinuities surfaces, we extend $\mathit{STAR}^\alpha(\cp)$ to a class of
sets $B$ with \emph{piecewise} $C^\alpha$ boundaries $\partial B$. We denote this class
$\mathit{STAR}^\alpha(\cp,L)$, where $L \in \N$ is the number of $C^\alpha$ pieces and $\cp > 0$ be
an upper bound for the ``curvature'' on each piece. In other words, we say that $B \in
\mathit{STAR}^\alpha(\cp, L)$ if $B$ is a bounded subset of $\itvcc{0}{1}^3$ whose boundary
$\partial B$ is a union of finitely many pieces $\pa B_1, \dots, \pa B_L$ which do not overlap
except at their boundaries, and each patch $\partial B_i$ can be represented in parametric form
$\rho_l = \rho_l(\theta_1,\theta_2)$ by a $C^\alpha$-smooth radius function with
$\norm[\dot{C}^\alpha]{\rho_l} \le \nu$. We remark that we put no restrictions on how the patches
$\pa B_l$ meet, in particular, $B \in \mathit{STAR}^\alpha(\cp,L)$ can have arbitrarily sharp
edges % or cusp-like?
joining the pieces $\pa B_l$. Also note that
$\mathit{STAR}^\alpha(\cp)=\mathit{STAR}^\alpha(\cp,1)$.

The actual objects of interest to us are, as mentioned, not these starshaped sets, but functions
that have the boundary $\partial B$ as discontinuity surface.

\begin{definition} \label{def:cartoon-3d} %-piecewise}
  Let $\cp, \mu > 0$, $ \scp, \beta \in \itvoc{1}{2}$, and $L \in \N$. Then $\cE^{\beta}_{\alpha,L}(\R^3)$ denotes the
  set of functions $f: \R^3\to \C$ of the form
  \[
  f = f_0 + f_1 \chi_{B},
  \]
  where $B \in \mathit{STAR}^\alpha(\cp,L)$ and $f_i \in C^\beta(\R^3)$ % \cap
                                % H^\beta(\R^3)$ 
  with $\supp f_0 \subset \itvcc{0}{1}^3$ and $\norm[C^\beta]{f_i} \leq \mu$ for each
%  $\sum_{|\gamma| \leq 2} \|\pa^{\gamma}f_i\|_{\infty} \leq 1$ for each
  $i=0,1$. We let $\cE^{\beta}_{\alpha}(\R^3):=
  \cE^{\beta}_{\alpha,1}(\R^3)$.
\end{definition}

We speak of $\cE^{\beta}_{\alpha,L}(\R^3)$ as consisting of \emph{cartoon-like 3D images} having
$C^\beta$-smo\-oth\-ness apart from a piecewise $C^\alpha$ discontinuity surface. We stress that
$\cE_{\alpha,L}^\beta(\R^3)$ is not a linear space of functions and that
$\cE^\beta_{\alpha,L}(\R^3)$ depends on the constants $\nu$ and $\mu$ even though we suppress this
in the notation.
Finally, we let $\cE_{\scp,L}^{\text{bin}}(\R^3)$ denote binary cartoon-like images, that is,
functions $f=f_0 + f_1 \chi_B \in \cE_{\alpha,L}^{\beta}(\R^3)$, where $f_0=0$ and $f_1=1$.

\section{Optimality bound for sparse approximations} 
\label{sec:optimal-sparsity}
After having clarified the model situation $\cE^{\beta}_{\alpha,L}(\R^3)$, 
we will now discuss which measure for the accuracy of approximation by
representation systems we choose, and what optimality means in this
case.
We will later in Section~\ref{sec:optimal-sparsity-3d} restrict the parameter range in our model
class $\cE^{\beta}_{\alpha,L}(\R^3)$ to $1<\alpha\le\beta \le2$. In this section, however, we will
find the theoretical optimal approximation error rate within $\cE^{\beta}_{\alpha,L}(\R^3)$ for the
full range $1<\alpha\le 2$ and $\beta \ge 0$. Before we state and prove the main optimal sparsity
result of this section, Theorem~\ref{thm:copy-lp-in-E}, we discuss the notions of $N$-term
approximations and frames.

\subsection{$N$-term approximations}
\label{sec:n-term-appr-1}

Let $\Phi=\set{\phi_i}_{i \in I}$ be a dictionary with the index set $I$ not necessarily being
countable. We seek to approximate each single element of $\cE^{\beta}_{\alpha,L}(\R^3)$ with
elements from $\Phi$ by $N$ terms of this system. For this, let $f \in \cE^{\beta}_{\alpha,L}(\R^3)$ be arbitrarily chosen.
Letting now $N \in \N$, we consider $N$-term approximations of $f$, \ie 
\[
\sum_{i \in {I}_N} {c}_i \phi_i \quad \text{with } {I}_N \subset I, \, \card{{I}_N} = N.
\]
The \emph{best $N$-term approximation} to $f$ is an $N$-term approximation
\[
f_N = \sum_{i \in {I}_N} {c}_i \phi_i,
\]
which satisfies that, for all $I_N \subset I$, $\card{I_N} = N$, and
for all scalars $(c_i)_{i \in I}$, 
\[
\norm[L^2]{f-f_N} \le \normBig[L^2]{f - \sum_{i \in I_N} c_i \phi_i}.
\]

\subsection{Frames}
\label{sec:frames}

A \emph{frame} for a separable Hilbert space $\mathcal{H}$ is a countable collection of vectors
$\{f_j\}_{j \in \mathbb{J}}$ for which there are constants $0 < A \leq B < \infty$ such that
\[ 
A \norm{f}^2 \leq \sum_{j \in \mathbb{J}} \abs{\innerprod{f}{f_j}}^2 \leq
B \norm{f}^2 \qquad\text{for all }f\in \mathcal{H}.
\]
If the upper bound in this inequality holds, then $\{f_j\}_{j \in \mathbb{J}}$ is said to be a
\emph{Bessel sequence} with Bessel constant $B$. For a Bessel sequence $\{f_j\}_{j \in \mathbb{J}}$,
we define the frame operator of $\{f_j\}_{j \in \mathbb{J}}$ by
\[
\Frame\colon \mathcal{H} \to \mathcal{H}, \qquad \Frame f = \sum_{j \in
  \mathbb{J}} \innerprod{f}{f_j} f_j. 
\] 
If $\{f_j\}_{j \in \mathbb{J}}$ is a frame, this operator is bounded, invertible, and positive. A
frame $\{f_j\}_{j \in \mathbb{J}}$ is said to be \emph{tight} if we can choose $A = B$. If
furthermore $A= B=1$, the sequence $\{f_j\}_{j \in \mathbb{J}}$ is said to be a \emph{Parseval
  frame}. Two Bessel sequences $\{f_j\}_{j \in \mathbb{J}}$ and $\{g_j\}_{j \in \mathbb{J}}$ are
said to be \emph{dual frames} if
\[ f = \sum_{j \in \mathbb{J}} \innerprod{f}{g_j}f_j \qquad \text{for all } f
\in \mathcal{H}. 
\] 
It can be shown that, in this case, both Bessel sequences are even frames, and we shall say that
the frame $\{g_j\}_{j \in \mathbb{J}}$ is \emph{dual} to $\{f_j\}_{j \in \mathbb{J}}$, and vice
versa. At least one dual always exists; it is given by $\{\inv\Frame f_j \}_{j \in \mathbb{J}}$ and
called the \emph{canonical dual}.

Now, suppose the dictionary $\Phi$ forms a frame for $L^2(\R^3)$ with frame bounds $A$ and $B$, and
let $\setsmall{\tilde\phi_i}_{i \in I}$ denote the canonical dual frame. We then consider the expansion
of $f$ in terms of this dual frame, \ie
\begin{equation*} %\label{eq:first} 
f = \sum_{i \in I} \ip{f}{\phi_i}
  \tilde{\phi}_i.
\end{equation*} 
For any $f \in L^2(\R^2)$ we have $(\innerprod{f}{\phi_i})_{i \in I} \in \ell^2(I)$ by definition.
Since we only consider expansions of functions $f$ belonging to a subset
$\cE^{\beta}_{\alpha,L}(\R^3)$ of $L^2(\R^3)$, this can, at least, potentially improve the decay
rate of the coefficients so that they belong to $\ell^p(I)$ for some $p<2$. This is exactly what is
understood by {\em sparse approximation} (also called {\em compressible approximations}). We hence
aim to analyze shearlets with respect to this behavior, \ie the decay rate of shearlet coefficients.

For frames, tight and non-tight, it is not possible to derive a usable, explicit form for the
best $N$-term approximation. We therefore crudely approximate the best $N$-term approximation
by choosing the $N$-term approximation provided by the indices ${I}_N$ associated with the $N$
largest coefficients $\ip{f}{\phi_i}$ in magnitude with these coefficients, \ie
\[
f_N = \sum_{i \in {I}_N} \ip{f}{\phi_i} \tilde \phi_i.
\]
However, even with this rather crude greedy selection procedure, we obtain very strong
results for the approximation rate of shearlets as we will see in
Section~\ref{sec:optimal-sparsity-3d}.

The following well-known result shows how the $N$-term approximation error can be bounded by the
tail of the square of the coefficients $c_i=\innerprod{f}{\phi_i}$. We refer to \cite{KLL11} for a
proof.
\begin{lemma}\label{lemma:n-term-frame-approx}
  Let $\setsmall{\phi_i}_{i \in I}$ be a frame for $H$ with frame bounds $A$ and $B$, and let
  $\setsmall{\tilde\phi_i}_{i \in I}$ be the canonical dual frame. Let $I_N \subset I$ with $\card{I_N} = N$,
  and let $f_N$ be the $N$-term approximation $f_N = \sum_{i \in I_N} \innerprod{f}{\phi_i} \tilde
  \phi_i$. Then
\begin{equation*}
\norm{f-f_N}^2 \le \frac1A \sum_{i \notin I_N} \abs{\innerprod{f}{\phi_i}}^2 %\label{eq:n-term-frame-approx-bound}
\end{equation*}
for any $f \in L^2(\R^3)$.
\end{lemma}

Let $c^\ast$ denote the non-increasing (in modulus) rearrangement of
$c=(c_i)_{i \in I}=(\innerprod{f}{\phi_i})_{i \in I}$, \eg
$c^\ast_{\,n}$ denotes the $n$th largest coefficient of $c$ in
modulus. This rearrangement corresponds to a bijection $\pi:\N \to I$
that satisfies 
\begin{equation*}
\pi:\N \to I, \quad c_{\pi(n)}=c^\ast_{\,n} \text{ for all $n \in
  \N$}. %\label{eq:rearrangement-bijection}
 \end{equation*}
Since $c \in \ell^2(I)$, also $c^\ast \in \ell^2(\N)$. Let $f$ be a cartoon-like image, and suppose 
that $\abs{c^\ast_{n}}$, in this case, even decays as
\begin{equation}
\abs{c^\ast_{n}} \lesssim n^{-(\alpha+2)/4} \qquad \text{for} \quad n \to \infty\label{eq:sought-weak-decay}
\end{equation}
for some $\alpha>0$, where the notation $h(n)\lesssim g(n)$ means that
there exists a $C>0$ such that $h(n) \le C g(n)$, \ie $h(n)= O(g(n))$.
Clearly, we then have $c^\ast \in \ell^p(\N)$ for $p\ge \tfrac{4}{\alpha+2}$.
By Lemma~\ref{lemma:n-term-frame-approx}, the $N$-term
approximation error will therefore decay as
\begin{equation}
\norm{f-f_N}^2 \le \frac1A \sum_{n>N} \abs{c^\ast_{n}}^2 \lesssim \sum_{n>N} n^{-\alpha/2+1} \asymp N^{-\alpha/2},
\label{eq:from-coeff-decay-to-error-decay}
\end{equation}
where $f_N$ is the $N$-term approximation of $f$ by keeping the $N$ largest coefficients, that is,
\begin{equation}
f_N = \sum_{n=1}^N c^*_{\,n} \, \tilde{\phi}_{\pi(n)}.
\label{eq:frame-n-term-largest}
\end{equation}
The notation $h(n)\asymp g(n)$, sometimes also written as $h(n)=\Theta(g(n))$, used above means that $h$ is
bounded both above and below by $g$ asymptotically as $n\to \infty$, that is, $h(n)=O(g(n))$
\emph{and} $g(n)=O(h(n))$. The approximation error rate $N^{-\alpha/2}$ obtained
in~(\ref{eq:from-coeff-decay-to-error-decay}) is exactly the sought optimal rate mentioned in the
introduction. This illustrates that the fraction $\frac{\alpha+2}{4}$ introduced in the decay of the
sequence $c^*$ will play a major role in the following. In particular, we are searching for a
representation system $\Phi$ which forms a frame and delivers decay of
$c=(\innerprod{f}{\phi_i})_{i\in I}$ as in~\eqref{eq:sought-weak-decay} for any cartoon-like image.

\subsection{Optimal sparsity}
\label{sec:sparsity}

In this subsection we will state and prove the main result of this section,
Theorem~\ref{thm:copy-lp-in-E}, but let us first discuss some of its implications for sparse
approximations of cartoon-like images. 

From the $\Phi=\setsmall{\phi_i}_{i \in I}$ dictionary with the index set $I$ not necessarily being
countable, we consider  expansions of the form
\begin{equation}
f = \sum_{i \in I_f } c_i \, \phi_i,
\label{eq:dict-expansion}
\end{equation}
where $I_f \subset I$ is a countable selection from $I$ that may depend on $f$. Moreover, we can assume that $\phi_i$
are normalized by $\norm[L^2]{\phi_i} =1$. The selection of the $i$th term is obtained according to a selection rule
$\sigma(i,f)$ which may \emph{adaptively} depend on $f$.  Actually, the $i$th element may also be modified adaptively
and depend on the first $(i-1)$th chosen elements \cite{Don01}. We assume that how deep or how far down in the indexed
dictionary $\Phi$ we are allowed to search for the next element $\phi_i$ in the approximation is limited by a polynomial
$\pi$.  Without such a depth search limit, one could choose $\Phi$ to be a countable, dense subset of $L^2(\R^3)$ which
would yield arbitrarily good sparse approximations, but also infeasible approximations in practise. We shall denote any
sequence of coefficients $c_i$ chosen according to these restrictions by $c(f)=(c(f)_i)_i$.

We are now ready to state the main result of this section. Following Donoho~\cite{Don01} we say that
a function class $\mathcal{F}$ contains an embedded orthogonal hypercube of dimension $m$ and side
$\delta$ if there exists $f_0 \in \mathcal{F}$, and orthogonal functions $\psi_{i,m,\delta}$, $i
=1,\dots,m$, with $\norm[L^2]{\psi_{i,m,\delta}}=\delta$, such that the collection of hypercube
vertices
\begin{equation*}
  %\label{eq:hypercube}
  \mathcal{H}(m;f_0,\{\psi_i\}):= \setprop{f_0+\sum_{i=1}^m\xi_i
    \psi_{i,m,\delta}}{\xi_i \in \{0,1\}}
\end{equation*}
is contained in $\mathcal{F}$. The sought bound on the optimal sparsity within the set of cartoon-like images will be obtained by showing that the cartoon-like image class contains sufficiently high-dimensional hypercubes with sufficiently large sidelength; intuitively, we will see that a certain high complexity of the set of cartoon-like images limits the possible sparsity level. The meaning of ``sufficiently'' is made precise by the following definition. 
 We say that a function class $\mathcal{F}$ contains a
copy of $\ell^p_0$ if $\mathcal{F}$ contains embedded orthogonal hypercubes of dimension $m(\delta)$
and side $\delta$, and if, for some sequence $\delta_k \to 0$, and some constant $C>0$:
\begin{equation}
 m(\delta_k) \ge C \, \delta_k^{-p}, \qquad k=k_0,k_0+1,\dots\label{eq:dimension-growth}
\end{equation}
The first part of the following result is an extension from the 2D to the 3D setting of \cite[Thm.
3]{Don01}. 

\begin{theorem}\label{thm:copy-lp-in-E}
  \begin{romannum}
  \item The class of binary cartoon-like images $\cE_\alpha^{\text{bin}}(\R^3)$
    contains a copy of $\ell^p_0$ for $p = 4/(\alpha + 2)$.
  \item The space of H\"older functions $C^\beta(\R^3)$ with compact
    support in $\itvcc{0}{1}^3$ contains a copy of $\ell^p_0$ for $p =
    6/(2\beta + 3)$.
  \end{romannum}
\end{theorem}

Before providing a proof of the theorem, let us discuss some of its implications for sparse
approximations of cartoon-like images. Theorem~\ref{thm:copy-lp-in-E}(i) implies, by
\cite[Theorem~2]{Don01}, that for every $p<4/(\alpha+2)$ and every method of atomic decomposition
based on polynomial $\pi$ depth search from any countable dictionary $\Phi$, we have for $f \in
\cE_\alpha^{\text{bin}}(\R^3)$:
\begin{equation}
\min_{\sigma(n,f)\le \pi(n)} \max_{f\in \cE_{\alpha,L}^{\beta}(\R^3)}
\norm[w\ell^p]{c(f)} = +\infty, \label{eq:sparsity-of-coefficients}
\end{equation}
where the weak-$\ell_p$ ``norm''\footnote{Note that neither
  $\norm[w\ell_p]{\cdot}$ nor $\norm[\ell_p]{\cdot}$ (for $p<1$) is a norm since
  they do not satisfy the triangle inequality. Note also that the
  weak-$\ell_p$ norm is a special case of the Lorentz quasinorm.}
is defined as $\norm[w\ell^p]{c(f)} =
\sup_{n>0}n^{1/p}\abs{c^*_{\,n}}$. \emph{Sparse} approximations are
approximations of the form $\sum_i c(f)_i \, \phi_i$ with coefficients
$c(f)_n^\ast$ decaying at certain, hopefully high, rate.
Equation~(\ref{eq:sparsity-of-coefficients}) is a precise statement of
the optimal achievable sparsity level. No representation system (up to
the restrictions described above) can deliver
expansions~(\ref{eq:dict-expansion}) for
$\cE^{\text{bin}}_{\alpha}(\R^3)$ with coefficients satisfying $c(f)
\in w\ell_p$ for $p < 4/(\alpha + 2)$.
As we will see in Theorems~\ref{thm:3d-opt-sparse} and~\ref{thm:3d-opt-sparse-piecewise},
pyramid-adapted shearlet frames deliver $(\innerprod{f}{\psi_{\lambda}})_{\lambda} \in w\ell^p$ for
$p = 4/(\alpha+2-2\tau)$, where $0\le \tau<0.04$.

Assume for a moment that we have an ``optimal'' dictionary $\Phi$ at hand that delivers $c(f) \in
w\ell^{4/(\alpha+2)}$, and assume further that it is also a frame. As we saw in the
Section~\ref{sec:frames}, this implies that
\begin{equation*}
\norm[L^2]{f-f_N}^2 \lesssim  N^{-\alpha/2} \qquad \text{as } N \to \infty,
\end{equation*}
where $f_N$ is the $N$-term approximation of $f$ by keeping the $N$
largest coefficients. % On the
Therefore, no frame representation system can deliver at better approximation error rate than
$O(N^{-\alpha/2})$ under the chosen approximation procedure within the image model class
$\cE_\alpha^{\text{bin}}(\R^3)$. If $\Phi$ is actually an orthonormal basis, then this is truly the
optimal rate since \emph{best} $N$-term approximations, in this case, are obtained by keeping the
$N$ largest coefficients.

Similarly, Theorem~\ref{thm:copy-lp-in-E}(ii) tells us that the optimal approximation error rate
within the H\"older function class is $O(N^{-2\beta/3})$. Combining the two estimates we see that the
optimal approximation error rate within the full cartoon-like image class $\cE^\beta_\alpha(\R^3)$
cannot exceed $O(N^{-\min\{\alpha/2,2\beta/3\}})$ convergence. For the parameter range $1<\alpha \le
\beta \le 2$, this rate reduces to $O(N^{-\alpha/2})$. For $\alpha= \beta = 2$, as will show in
Section~\ref{sec:optimal-sparsity-3d}, shearlet systems actually deliver this rate except from an
additional polylog factor, namely $O(N^{-\alpha/2}(\log N)^2)=O(N^{-1}(\log N)^2)$. For $1<\alpha
\le \beta \le 2$ and $\alpha<2$, the $\log$-factor is replaced by a small polynomial factor
$N^{\tau(\alpha)}$, where $\tau(\alpha)<0.04$ and $\tau(\alpha)\to 0$ for $\alpha \to 1^+$ or
$\alpha \to 2^-$.

It is striking that one is able to obtain such a near optimal approximation error rate
since the shearlet system as well as the approximation procedure will be non-adaptive; in
particular, since traditional, non-adaptive representation systems such as Fourier series and wavelet
systems are far from providing an almost optimal approximation rate. This is illustrated in the
following example.

\begin{example}
\label{example:fourier-wavelets}
  Let $B=B(x,\rho)$ be the ball in $\itvcc{0}{1}^3$ with center $x$ and
  radius $r$. Define $f = \chi_B$. Clearly, $f \in \cE^2_2(\R^3)$ if $B \subset \itvcc{0}{1}^3$. 
Suppose $\Phi=\setsmall{\mathrm{e}^{2\pi ikx}}_{k \in \Z^d}$. The best $N$-term Fourier sum $f_N$ yields
\[ \norm[L^2]{f-f_N}^2 \asymp N^{-1/3} \qquad \text{for } N\to \infty,\] which is far from the
optimal rate $N^{-1}$. For the wavelet case the situation is only slightly better. Suppose 
$\Phi$ is any compactly supported wavelet basis. Then
 \[ \norm[L^2]{f-f_N}^2 \asymp N^{-1/2} \qquad \text{for } N\to
  \infty,\]
where $f_N$ is the best $N$-term approximation from $\Phi$. The calculations leading to these
estimates are not difficult, and we refer to \cite{KLL11} for the details. We will later see that
shearlet frames yield $\norm[L^2]{f-f_N}^2 \lesssim N^{-1}(\log N)^2$, where  $f_N$ is the best
$N$-term approximation. 
\end{example}

We mention that the rates obtained in Example~\ref{example:fourier-wavelets} are \emph{typical} in
the sense that most cartoon-like images will yield the exact same (and far from optimal) rates.

Finally, we end the subsection with a proof of Theorem~\ref{thm:copy-lp-in-E}.
%\bigskip
\begin{proof}[Proof of Theorem~\ref{thm:copy-lp-in-E}]
  The idea behind the proofs is to construct a collection of functions in
  $\cE_\alpha^{\text{bin}}(\R^3)$ and $C^\beta(\R^3)$, respectively, such that the collection of
  functions will be vertices of a hypercube with dimension satisfying (\ref{eq:dimension-growth}).

  (i): Let $\varphi_1$ and $\varphi_2$ be smooth $C^\infty$ functions with compact support $\supp
  \varphi_1 \subset \itvcc{0}{2\pi}$ and $\supp \varphi_2 \subset \itvcc{0}{\pi}$. For $A>0$ and $m
  \in \N$ we define:
  \[\varphi_{i,m}(t) =\varphi_{i_1,i_2,m}(t) = A m^{-\alpha}
  \varphi_1(mt_1-2\pi i_1)\varphi_2(mt_2-\pi i_2),
 \] for $i_1,i_2 \in \{0, \dots , m -
    1\}$, where $i=(i_1,i_2)$ and $t=(t_1,t_2)$. We let further
  $\varphi(t):=\varphi_1(t_1)\varphi_2(t_2)$.
  It is easy to see that $\norm[L^1]{\varphi_{i,m}} = m^{-\alpha +2} A \norm[L^1]{\varphi}$.
  Moreover, it can also be shown that $\norm[\dot{C}^\alpha]{\varphi_{i,m}} = A
  \norm[\dot{C}^\alpha]{\varphi}$, where $\norm[\dot{C}^\alpha]{\cdot}$ denotes the homogeneous
  H\"older norm introduced in (\ref{eq:curvebound}).

  Without loss of generality, we can consider the cartoon-like images
  $\cE_\alpha^{\text{bin}}(\R^3)$ translated by $-(\frac12,\frac12,\frac12)$ so that their support lies in
  $\itvcc{-1/2}{1/2}^3$. Alternatively, we can fix an origin at $(1/2,1/2,1/2)$, and use spherical
  coordinates $(\rho,\theta_1,\theta_2)$ relative to this choice of origin. We set $\rho_0=1/4$ and define
\[ \psi_{i,m}= \chi_{\{\rho_0 < \rho \le \rho_0 + \varphi_{i,m}\}} \quad \text{for
    $i_1,i_2 \in \{0, \dots , m - 1\}$}. \]
  The radius functions $\rho_\gamma$ for $\gamma=(\gamma_{i_1,i_2})_{i_1,i_2\in \setsmall{0,\dots, m-1}}$ with
  $\gamma_{i_1,i_2} \in \{0, 1\}$ defined by
  \begin{equation}
    \rho_\gamma(\theta_1,\theta_2) = \rho_0 + \sum^m_{i_1= 1} \sum^m_{i_2= 1} 
    \gamma_{i_1,i_2}  \,\varphi_{i,m}(\theta_1,\theta_2)  %\quad \text{where 
 ,\label{eq:radius-functions}
  \end{equation}
 determines the discontinuity surfaces of the functions of the form: 
  \begin{equation*}
    f_\gamma = \chi_{\{\rho \le \rho_0\}} + \sum^m_{i_1= 1} \sum^m_{i_2= 1} 
    \gamma_{i_1,i_2}  \psi_{i,m}  \quad \text{for $\gamma_{i_1,i_2} \in 
      \{0, 1\}$}. %\label{eq:corres-images}
  \end{equation*}
  For a fixed $m$ the functions $\psi_{i ,m}$ are disjointly supported and therefore mutually
  orthogonal. Hence, $\cH(m^2,\chi_{\{\rho \le \rho_0\}},\{\psi_{i,m}\})$ is a collection of hypercube
  vertices.  Moreover,
  \begin{align*}
    \norm[L^2]{\psi_{i,m}}^2 &=
    \lambda(\setprop{(\rho,\theta_1,\theta_2)}{\rho_0 \le \rho \le
      \rho_0+\varphi_{i,m}(\theta_1,\theta_2)})\\ &\le \int_0^{2\pi}
    \int_0^\pi \int_{\rho_0}^{\rho_0+\varphi_{i,m}(\theta_1,\theta_2)} \rho^2
    \sin \theta_2 \; d\rho \, d\theta_2 \, d\theta_1 \\ %&= ... \\
    & \le C_0\, m^{-\alpha-2} \norm[L^1]{\varphi},
  \end{align*}
  where the constant $C_0$ only depends on $A$. % [WRITE EXACT EXP?].
  Any radius function $\rho=\rho(\theta_1,\theta_2)$ of the form~(\ref{eq:radius-functions}) satisfies
  \[ 
  \norm[\dot{C}^\alpha]{\rho_\gamma} \le \norm[\dot{C}^\alpha]{\varphi_{i,m}} =A
  \norm[\dot{C}^\alpha]{\varphi}. 
  \] 
  Therefore, $\norm[\dot{C}^\alpha]{\rho} \le \nu$ whenever $A \le \nu/\norm[\dot{C}^\alpha]{\varphi}$. This shows that we
  have the hypercube embedding
  \[
   \cH(m^2,\chi_{\{\rho \le \rho_0\}},\{\psi_{i,m}\}) \subset
  \cE_\alpha^{\text{bin}}(\R^3). 
  \]
  The side length $\delta = \norm[L^2]{\psi_{i,m}}$ of the hypercube satisfies
  \[
  \delta^2 \le C_0\, m^{-\alpha-2} \norm[L^1]{\varphi} \le \nu \,
  \frac{\norm[L^1]{\varphi}}{\norm[\dot{C}^\alpha]{\varphi}}\,
  m^{-\alpha-2},
  \]
  whenever $C_0 \le \nu/\norm[\dot{C}^\alpha]{\varphi}$. Now, we finally choose $m$ and $A$ as
  \[
  m(\delta) = \floor{\left(\frac{\delta^2}{\nu}
      \frac{\norm[\dot{C}^\alpha]{\varphi}}{\norm[L^1]{\varphi}}
    \right)^{-1/(\alpha+2)}} \quad \text{and} \quad A(\delta,\nu)=
  \delta^2 m^{\alpha+2}/ \norm[L^1]{\varphi}.
  \]
  By this choice, we have $C_0 \le \nu/\norm[\dot{C}^\alpha]{\varphi} $ for sufficiently small $\delta$.
  Hence, $\cH$ is a hypercube of side length $\delta$ and dimension $d=m(\delta)^2$ embedded in
  $\cE_\alpha^{\text{bin}}(\R^3)$. We obviously have $m(\delta) \ge C_1 \nu^{\frac{1}{\alpha+2}}
  \delta^{-\frac{2}{\alpha+2}}$, thus the dimension $d$ of the hypercube obeys
  \[ d \ge C_2 \, \delta^{-\frac{4}{\alpha+2}}\] for all sufficiently small $\delta>0$.

  \smallskip

  \noindent (ii): Let $\varphi \in C_0^\infty(\R)$ with compact support $\supp \varphi \subset
  \itvcc{0}{1}$. For $m \in \N$ to be determined, we define for $i_1,i_2,i_3
    \in \{0, \dots , m - 1\}$:
  \[
  \psi_{i,m}(t) =\psi_{i_1,i_2,i_3,m}(t) = m^{-\beta} \varphi(mt_1-
  i_1)\varphi(mt_2- i_2)\varphi(mt_3-i_3) ,
  \]
  where $i=(i_1,i_2,i_3)$ and $t=(t_1,t_2,t_3)$. We let $\psi(t):=\varphi(t_1)\varphi(t_2)
  \varphi(t_3)$. It is easy to see that $\norm[L^2]{\psi_{i,m}}^2 = m^{-2\beta-3}
  \norm[L^2]{\psi}^2$. We note that the functions $\psi_{i ,m}$ are disjointly supported (for a
  fixed $m$) and therefore mutually orthogonal. Thus we have the hypercube embedding
  \[
  \cH(m^3,0,\{\psi_{i,m}\}) \subset C^\beta(\R^3),
  \]
  where the side length of the hypercube is $\delta = \norm[L^2]{\psi_{i,m}} = m^{-\beta-3/2}
  \norm[L^2]{\psi}$. Now, chose $m$ as
  \[
  m(\delta) = \floor{\left(\frac{\delta}{ \norm[L^2]{\psi}}
    \right)^{-1/(\beta+3/2)}}.
  \]
  Hence, $\cH$ is a hypercube of side length $\delta$ and dimension $d=m(\delta)^3$ embedded in
  $C^\beta(\R^3)$. The dimension $d$ of the hypercube obeys
  \[ d \ge C \, \delta^{-3 \frac{1}{\beta+3/2}} = C \, \delta^{-
    \frac{6}{2\beta+3}}, \] for all sufficiently small $\delta>0$.
\end{proof}

\subsection{Higher dimensions}
\label{sec:higher-dimensions}

Our main focus is, as mentioned above, the three-dimensional setting, but let us briefly sketch how
the optimal sparsity result extends to higher dimensions. The $d$-dimensional cartoon-like image
class $\cE^\beta_\alpha(\R^d)$ consists of functions having $C^\beta$-smoothness apart from a
$(d-1)$-dimensional $C^\alpha$-smooth discontinuity surface. The $d$-dimensional analogue of
Theorem~\ref{thm:copy-lp-in-E} is then straightforward to prove.
\begin{theorem}\label{thm:copy-lp-in-E-higher-dim}
  \begin{romannum}
  \item The class of $d$-dimensional binary cartoon-like images $\cE_\alpha^{\text{bin}}(\R^d)$ contains a copy of $\ell^p_0$
    for $p = 2(d-1)/(\alpha + d -1)$.
  \item The space of H\"older functions $C^\beta(\R^d)$ contains a copy of $\ell^p_0$ for $p = \frac{2d}{2\beta + d}$.
  \end{romannum}
\end{theorem}
It is then intriguing to analyze the behavior of $p = 2(d-1)/(\alpha + d -1)$ and $p = 2d/(2\beta +
d)$. from Theorem~\ref{thm:copy-lp-in-E-higher-dim}. In fact, as $d \to \infty$, we observe that $p
\to 2$ in both cases. Thus, the decay of any ${c(f)}$ for cartoon-like images becomes slower as $d$
grows and approaches $\ell^2$ which is actually the rate guaranteed for \emph{all} $f \in
L^2(\R^d)$.

Moreover, by Theorem~\ref{thm:copy-lp-in-E-higher-dim} we see that the optimal approximation error
rate for $N$-term approximations $f_N$ within the class of $d$-dimensional cartoon-like images
$\cE^\beta_\alpha(\R^d)$ is %bounded from below by
 $N^{-\min\{\alpha/(d-1),2\beta/d\}}$.
In this paper we will however restrict ourselves to the case $d=3$ since we, as mentioned
in the introduction, can see this dimension as a critical one.

%*******************************************************************************************
%*******************************************************************************************
\section{Hybrid shearlets in 3D}
\label{sec:shearl-high-dimens}

After we have set our benchmark for directional representation systems
in the sense of stating an optimality criteria for sparse
approximations of the cartoon-like image class $\cE^\beta_{\alpha,L}(\R^3)$, we
next introduce the class of shearlet systems we claim behave optimally.

\subsection{Pyramid-adapted shearlet systems}
\label{sec:pyram-adapt-shearl}

Fix $\scp \in \itvoc{1}{2}$. We scale according to \emph{scaling matrices} $A_{2^j}$,
$\tilde{A}_{2^j}$ or $\breve{A}_{2^j}$, $j \in \Z$, and represent directionality by the \emph{shear
  matrices} $S_k$, $\tilde{S}_k$, or $\breve{S}_k$, $k = (k_1,k_2) \in \Z^2$, defined by
\begin{alignat*}{6} 
  A_{2^j} &=\begin{pmatrix}
    2^{j\scp/2}\! & 0 & 0 \\ 0 & \!\!\!2^{j/2}\! & 0 \\ 0 & 0 & \!\!\!2^{j/2}
  \end{pmatrix}\!, &\;\;&  \tilde{A}_{2^j}&=&\begin{pmatrix}
    2^{j/2} & 0 & 0 \\ 0 & \!\!\!2^{j\scp/2}\! & 0 \\ 0 & 0 & \!\!\!2^{j/2}
  \end{pmatrix}\!, &\;\;& \text{and}&\;\;& \breve{A}_{2^j} &=&\begin{pmatrix}
    2^{j/2}\! & 0 & 0 \\ 0 &\!\!\! 2^{j/2}\! & 0 \\ 0 & 0 & \!\!\!2^{j\scp/2}
  \end{pmatrix},
% \end{align*}
\intertext{and}
%\begin{align*}
S_k &=\begin{pmatrix} 1\; & k_1\; & k_2 \\ 0 & 1 & 0
    \\ 0 & 0 & 1
  \end{pmatrix}, &&  \tilde{S}_k &=&\begin{pmatrix} 1 & 0 & 0 \\ k_1\;
    & 1\; & k_2 \\ 0 & 0 & 1
  \end{pmatrix}, && \text{and}&& \breve{S}_k\; &=&\begin{pmatrix} 1
    & 0 & 0 \\ 0 & 1 & 0 \\ k_1\; & k_2\; & 1
  \end{pmatrix},
\end{alignat*}
% \end{align*}
respectively. The case $\alpha=2$ corresponds to paraboloidal scaling. As $\alpha$ decreases, the
scaling becomes less anisotropic, and allowing $\alpha=1$ would yield isotropic scaling. The action
of isotropic scaling and shearing is illustrated in Figure~\ref{fig:action-scaling-shearing}.   
\begin{figure}[ht]
\centering
\includegraphics{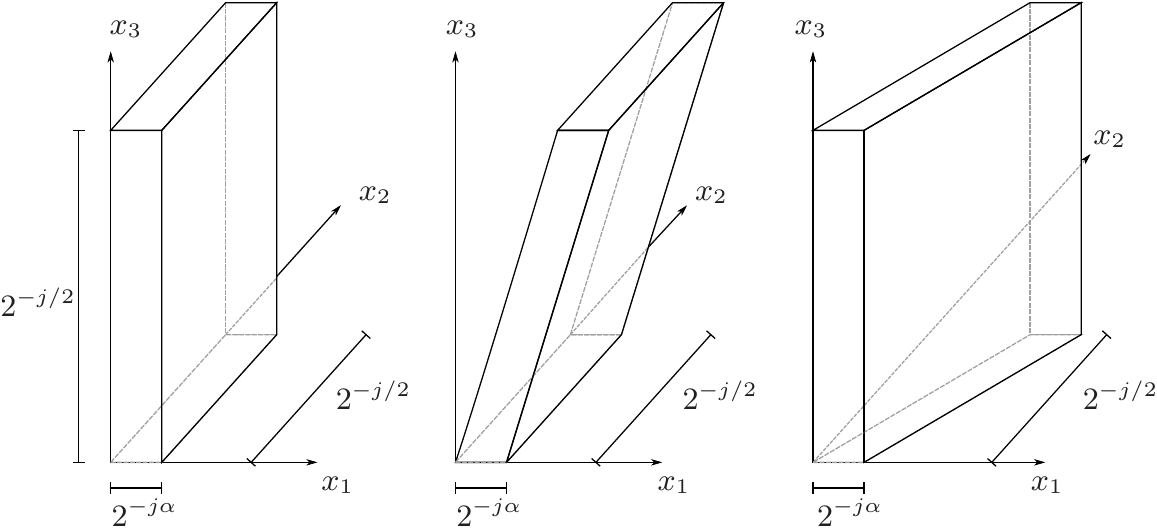}
  \caption{Sketch of the action of scaling ($\alpha \approx 2$) and shearing. For $\psi \in
    L^2(\R^3)$ with $\supp \psi \subset \itvcc{0}{1}^3$ we plot the support of $\psi(S_kA_{j}\cdot)$
    for fixed $j > 0$ and various $k=(k_1,k_2) \in \Z^2$. From left to right: $k_1=k_2=0$, $k_1=0,
    k_2< 0$, and $k_1<0, k_2=0$.}
\label{fig:action-scaling-shearing}
\end{figure}
The translation lattices will be generated by the following matrices: $M_c =
\mathrm{diag}(c_1,c_2,c_2)$, $\tilde{M}_c = \mathrm{diag}(c_2,c_1,c_2)$, and $\breve{M}_c =
\mathrm{diag}(c_2,c_2,c_1)$, where $c_1>0$ and $c_2>0$.

\begin{figure}[ht]
%\vspace*{-1.5em}
%\sidecaption[t]
 \centering
      %\psfragfig{./pyramidrectangle} % [width=5cm]
\includegraphics{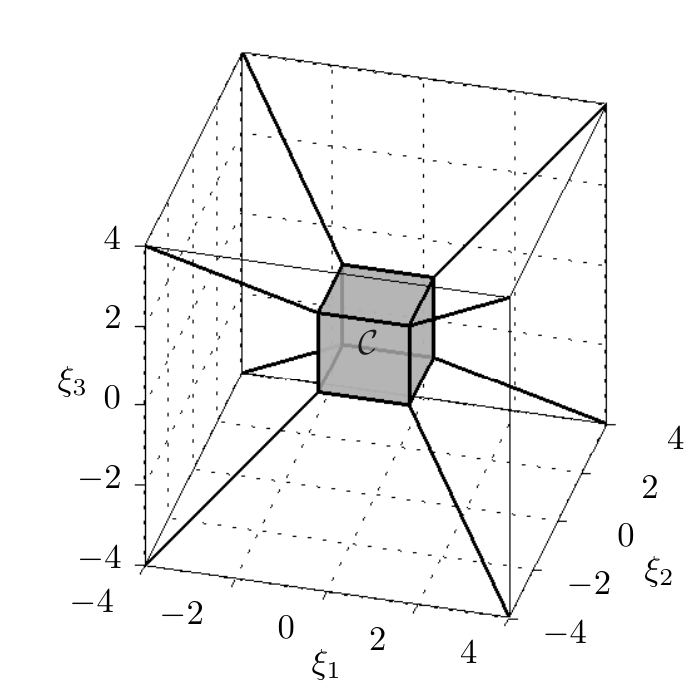}
      \caption{Sketch of the partition of the frequency domain. The centered cube $\cC$ is shown,
        and the arrangement of the six pyramids is indicated by the ``diagonal'' lines. We refer to
        Figure~\ref{fig:pyramids} for a sketch of the pyramids.}
\label{fig:partition}
\end{figure}

We next partition the frequency domain into the following six pyramids:
\begin{align*}
  \cP_\iota = \left\{ \begin{array}{rcl}
      \{(\xi_1,\xi_2,\xi_3) \in \R^3 : \xi_1 \ge 1,\, |\xi_2/\xi_1| \le 1,\, |\xi_3/\xi_1| \le 1\} & : & \iota = 1,\\
      \{(\xi_1,\xi_2,\xi_3) \in \R^3 : \xi_2 \ge 1,\, |\xi_1/\xi_2|
      \le 1,\,
      |\xi_3/\xi_2| \le 1\} & : & \iota = 2,\\
      \{(\xi_1,\xi_2,\xi_3) \in \R^3 : \xi_3 \ge 1,\, |\xi_1/\xi_3|
      \le 1,\,
      |\xi_2/\xi_3| \le 1\} & : & \iota = 3,\\
      \{(\xi_1,\xi_2,\xi_3) \in \R^3 : \xi_1 \le -1,\, |\xi_2/\xi_1| \le 1,\, |\xi_3/\xi_1| \le 1\} & : & \iota = 4,\\
      \{(\xi_1,\xi_2,\xi_3) \in \R^3 : \xi_2 \le -1,\, |\xi_1/\xi_2|
      \le
      1,\, |\xi_3/\xi_2| \le 1\} & : & \iota = 5,\\
      \{(\xi_1,\xi_2,\xi_3) \in \R^3 : \xi_3 \le -1,\, |\xi_1/\xi_3|
      \le 1,\, |\xi_2/\xi_3| \le 1\} & : & \iota = 6,
    \end{array}
  \right.
\end{align*}
and a centered cube
\[
\cC = \{(\xi_1,\xi_2,\xi_3) \in \R^3 : \norm[\infty]{(\xi_1,\xi_2, \xi_3)} < 1\}.
\]

The partition is illustrated in Figures~\ref{fig:partition} and~\ref{fig:pyramids}. This partition
of the frequency space into pyramids allows us to restrict the range of the shear parameters. In
case of the shearlet group systems, one must allow arbitrarily large shear parameters. For the
pyramid-adapted systems, we can, however, restrict the shear parameters to
$\itvcc{-\ceil{2^{j(\scp-1)/2}}}{\ceil{2^{j(\scp-1)/2}}}$. We would like to emphasize that this
approach is important for providing an almost uniform treatment of different directions -- in a
sense of a good approximation to rotation.

\begin{figure}[ht]
 \vspace*{-.8em}
\centering
\subfloat[Pyramids $\cP_1$ and $\cP_4$ and the $\xi_1$ axis.]{%
\includegraphics[height=2.7cm]{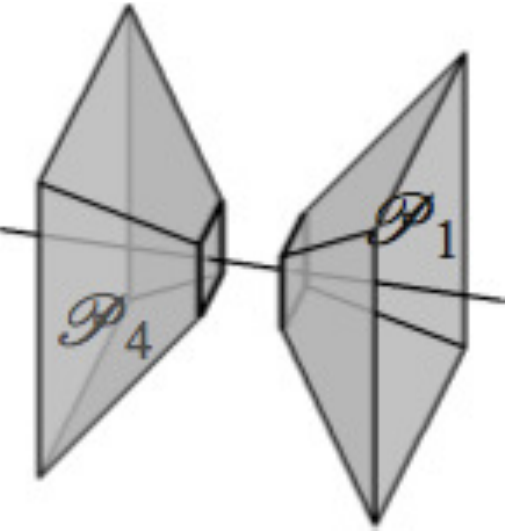}
%\psfragfig{./pyramid14_small}
%\subref{
\label{fig:pyramids14}
}\qquad
\subfloat[Pyramids $\cP_2$ and $\cP_5$ and the $\xi_2$ axis.]{%
\includegraphics[height=2.7cm]{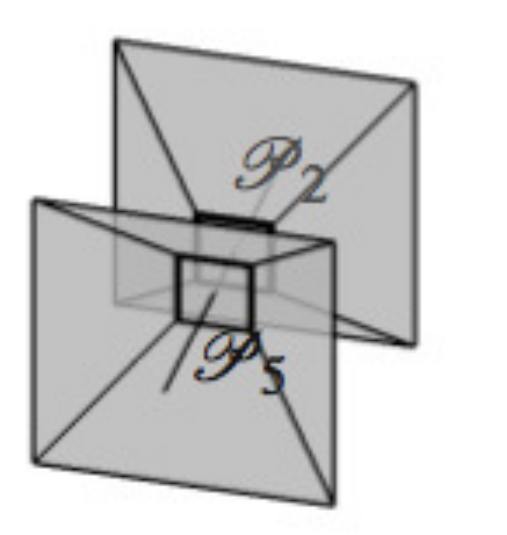}
%\psfragfig{./pyramid25_small}
%\subref{
\label{fig:pyramids25}
}\qquad
\subfloat[Pyramids $\cP_3$ and $\cP_6$ and the $\xi_3$ axis.]{%
\includegraphics[height=2.7cm]{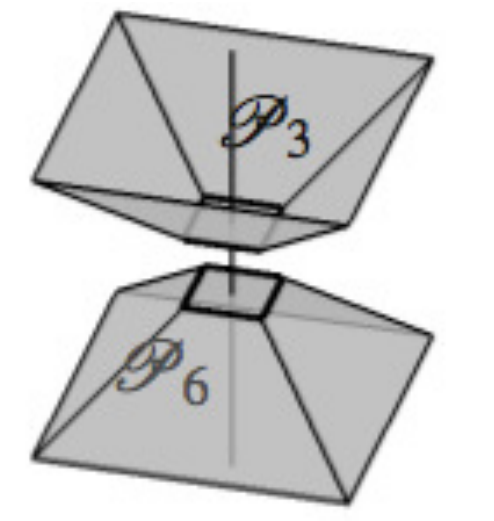}
%\psfragfig{./pyramid36_small}
%\subref{}
\label{fig:pyramids136}
}
\caption{The partition of the
  frequency domain: The ``top'' of the six pyramids.}
\label{fig:pyramids}
\end{figure}

These considerations are made precise in the following definition.

\begin{definition}
  \label{def:discreteshearlets3d}
  For $\scp \in \itvoc{1}{2}$ and $c=(c_1,c_2) \in (\R_+)^2$, the \emph{pyramid-adapted, hybrid
    shearlet system} $\SH(\phi,\psi,\tilde{\psi},\breve{\psi};c, \scp)$ generated by $\phi, \psi,
  \tilde{\psi}, \breve{\psi} \in L^2(\R^3)$ is defined by
  \[
  \SH(\phi,\psi,\tilde{\psi},\breve{\psi};c,\scp) = \Phi(\phi;c_1) \cup \Psi(\psi;c,\scp) \cup \tilde{\Psi}(\tilde{\psi};c,\scp) \cup
  \breve{\Psi}(\breve{\psi};c,\scp),
  \]
  where
  \begin{align*}
    \Phi(\phi;c_1) &= \setprop{\phi_m = \phi(\cdot-m)}{ m \in c_1\Z^3}, \\
    \Psi(\psi;c,\scp) &= \setprop{\psi_{j,k,m} = \scfac {\psi}({S}_{k} {A}_{2^j}\cdot-m) }{ j \ge 0, |k| \le
      \ceilsmall{2^{j(\scp-1)/2}}, m \in
      M_c \Z^3 }, \\
    \tilde{\Psi}(\tilde{\psi};c,\scp) &= \{\tilde{\psi}_{j,k,m} = \scfac \tilde{\psi}(\tilde{S}_{k} \tilde{A}_{2^j}\cdot-m) : j
    \ge 0, |k| \le
    \lceil 2^{j(\scp-1)/2} \rceil, m \in \tilde{M}_c \Z^3 \},\\
    \intertext{and} \breve{\Psi}(\breve{\psi};c,\scp) &= \{\breve{\psi}_{j,k,m} = \scfac \breve{\psi}(\breve{S}_{k}
    \breve{A}_{2^j}\cdot-m) : j \ge 0, |k| \le \lceil 2^{j(\scp-1)/2} \rceil, m \in \breve{M}_c \Z^3 \},
  \end{align*}
  where $j \in \N_0$ and $k \in \Z^2$. %and $m\in \Z^3$.
  Here we have used the vector notation $\abs{k} \le K$ for $k = (k_1,k_2)$ and $K>0$ to denote
  $\abs{k_1} \le K$ \emph{and} $\abs{k_2} \le K$. We will often use $\Psi(\psi)$ as shorthand
  notation for $\Psi(\psi;c,\scp)$. If $\SH(\phi,\psi,\tilde{\psi},\breve{\psi};c,\scp)$ is a frame
  for $L^2(\R^3)$, we refer to $\phi$ as a \emph{scaling function} and $\psi$, $\tilde{\psi}$, and
  $\breve{\psi}$ as \emph{shearlets}. Moreover, we often simply term
  $\SH(\phi,\psi,\tilde{\psi},\breve{\psi};c,\scp)$ \emph{pyramid-adapted shearlet system}.
 \end{definition}

 We let $\cP = \cP_1 \cup \cP_4$, $\tilde\cP = \cP_2 \cup \cP_5$, and $\breve\cP = \cP_3 \cup
 \cP_6$. In the remainder of this paper, we shall mostly consider $\cP$; the analysis for
 $\tilde\cP$ and $\breve\cP$ is similar (simply append $\tilde{\cdot}$ and $\breve{\cdot}$,
 respectively, to suitable symbols).

 We will often assume the shearlets to be compactly supported in spatial domain. If
 \eg $\supp \psi \subset \itvcc{0}{1}^3$, then the shearlet element $\psi_{j,k,m}$ will be supported
 in a parallelepiped with side lengths $2^{-j\alpha/2}$, $2^{-j/2}$, and $2^{-j/2}$, see
 Figure~\ref{fig:action-scaling-shearing}. For $\scp=2$ this shows that the shearlet elements will
 become plate-like as $j \to \infty$. As $\scp$ approaches $1$ the scaling becomes almost isotropic
 giving almost isotropic cube-like elements. The key fact to mind is, however, that our shearlet
 elements always become plate-like as $j \to \infty$ with aspect ratio depending on $\alpha$.

 In general, however, we will have very weak requirements on the shearlet generators $\psi$,
 $\tilde{\psi}$, and $\breve{\psi}$. As a typical minimal requirement in our construction and
 approximation results we will require the shearlet $\psi$ to be \emph{feasible}.
\begin{definition}\label{def:feasible}
  Let $\delta, \gamma >0$. A function $\psi \in L^2(\R^3)$ is called a $(\delta,
  \gamma)$-\emph{feasible shearlet} associated with $\cP$, if there exist $q \ge q' >0$, $q \ge r
  >0$, $q \ge s >0$ such that
  \begin{equation}\label{eq:psi}
    \abssmall{\hat{\psi}(\xi)} \lesssim
    \min\setsmall{1,\abs{q\xi_1}^{\vap}} \;
    \min{\setsmall{1,\abs{q'\xi_1}^{-\dep}}} \;
    \min{\setsmall{1,\abs{r\xi_2}^{-\dep}}} \; \min{\setsmall{1,\abs{s\xi_3}^{-\dep}}},
  \end{equation}
  for all $\xi=(\xi_1,\xi_2,\xi_3)\in \R^3$. For the sake of brevity, we will often simply say that
  $\psi$ is $(\delta, \gamma)$-\emph{feasible}.
 %We will sometimes say that $\psi$ is a $(\delta, \gamma)$-feasible shearlet 
\end{definition}

Let us briefly comment on the decay assumptions in (\ref{eq:psi}). If $\psi$ is compactly
supported, then $\hat \psi$ will be a continuous function satisfying the decay assumptions as
$\abs{\xi} \to \infty$ for sufficiently small $\gamma>0$. The decay condition controlled by
$\delta$ can be seen as a vanishing moment condition in the $x_1$-direction which suggests that 
a $(\delta,\gamma)$-feasible shearlet will behave as a wavelet in the $x_1$-direction. 

\section{Construction of compactly supported shearlets}
In the following subsection we will describe the construction of pyramid-adapted shearlet systems
%$\SH(\phi,\psi,\tilde{\psi},\breve{\psi};c,\scp)$ 
with compactly supported generators. This
construction uses ideas from the classical construction of wavelet frames in \cite[\S 3.3.2]{Dau92};
we also refer to the recent construction of cone-adapted shearlet systems in $L^2(\R^2)$ described
in the paper \cite{KKL10a}.

\label{sec:constr-comp-supp}
\subsection{Covering properties}

We fix $\alpha \in \itvoc{1}{2}$, and let $\psi \in L^2(\R^3)$ be a feasible shearlet associated
with $\cP$. We then define the function $\Phi : \cP \times \R^3 \to \R$
by
\begin{equation}\label{eq:Phi}
\Phi(\xi,\omega) = \sum_{j \ge 0} \sum_{k \le \ceilsmall{2^{j(\scp-1)/2}}}
\abs{\hat{\psi}(S_{-k}^{T} A_{2^{-j}}\xi)} \abs{\hat{\psi}(S_{-k}^{T}A_{2^{-j}}\xi + \omega)}.
\end{equation}
This function measures to which extent the effective part of the supports of the scaled and sheared versions of the shearlet
generator overlaps. Moreover, it is linked to the so-called $t_q$-equations albeit with absolute value of the functions in
the sum~(\ref{eq:Phi}). We also introduce the function $\Gamma : \R^3
\to \R$ defined by
\[
\Gamma(\omega) = \esssup_{\xi \in \cP} \Phi(\xi,\omega),
\]
measuring the maximal extent to which these scaled and sheared versions overlap for a given distance $\omega \in \R^3$. 
The values
\begin{equation}
 \label{eq:defiLinLsup}
L_{\mathit{inf}} = \essinf_{\xi \in \cP} \Phi(\xi,0)
\qquad \text{and} \qquad
L_{\mathit{sup}} = \esssup_{\xi \in \cP} \Phi(\xi,0),
\end{equation}
will relate to the classical discrete Calder\'{o}n condition. Finally, the value
\begin{equation} \label{eq:defR}
R(c) = \sum\limits_{m \in \Z^{3}\setminus\{0\}} \left[ \Gamma\left({M_c^{-1}{m}}\right)
\Gamma\left({-M_c^{-1}{m}}\right)\right]^{1/2}, \quad \text{where} \quad c = (c_1,c_2) \in \R_+^2,
\end{equation}
measures the average of the symmetrized function values $\Gamma(M_c^{-1}m)$ and is again related to
the so-called $t_q$-equations.

We now first turn our attention to the terms $L_{\mathit{sup}}$ and $R(c)$ and provide upper bounds
for those. These estimates will later be used for estimates for frame bounds associated to a
shearlet system; we remark that the to be derived estimates~(\ref{eq:Upp}) and (\ref{eq:upper_R})
also hold when the essential supremum in the definition of $L_{\mathit{sup}}$ and $R(c)$ is taken
over all $\xi \in \R^3$.

To estimate the effect of shearing, we will repeatedly use the following estimates:
\begin{equation}
\sup_{(x,y)\in \R^2}\sum_{k \in \Z}  \min\set{1,\abs{y}} 
\min\set{1,\abs{x+ky}^{-\dep}} 
 \le 3+\frac{2}{\dep-1}=:  C(\dep) \label{eq:sk2}
 \end{equation}
and 
\begin{equation*}
\sup_{(x,y)\in \R^2}\sum_{k \neq 0}  \min\set{1,\abs{y}} 
\min\set{1,\abs{x+ky}^{-\dep}} 
 \le 2+\frac{2}{\dep-1} =C(\dep)-1 %\label{eq:skneq0}
 \end{equation*}
for $\dep>1$.

\begin{proposition} \label{prop:Lsupfinite}
  Suppose $\psi \in L^2(\R^3)$ is a $(\delta,\gamma)$-feasible shearlet with $\vap > 1$ and $\dep >
  1/2$.
Then
\begin{equation}\label{eq:Upp}
L_{\mathit{sup}} \le  \frac{q^2}{rs} \,  C(2\dep)^2 \,
\Bigl( \frac{1}{1-2^{(-\vap+1)\scp}}+ \ceil{\frac{2}{\scp} \log_{2}\Bigl(\frac{q}{q'}\Bigr)}
+ 1 \Bigr) < \infty, %\frac{1}{1-2^{-2\dep}}\Bigr) < \infty,
\end{equation}
where $C(\dep)=3+\frac{2}{\dep-1}$. 
\end{proposition}

\begin{proof}
%Let $\xi = (\xi_{1}, \xi_{2})^{T} \in {\mathbb{R}^{2}}$.
By \eqref{eq:psi}, we immediately have the following bound for $\Phi(\xi,0)$:
\begin{align*}
\Phi(\xi,0) &\leq \sup_{\xi \in \R^3} \sum_{j \ge 0}
\min{\setbig{1,|q2^{-j\scp/2}\xi_1|^{2\vap}}}\;
\min\setbig{1,|q'2^{-j\scp/2}\xi_1|^{-2\dep}} \\ %\label{eq:estimate1-constr} \\ 
& \phantom{\leq \sup \sum \;} \cdot \sum_{k_1 \in
  \Z}\min\setbig{1,|r(2^{-j/2}\xi_2+k_12^{-j\scp/2}\xi_1)|^{-2\dep}} \\ %\nonumber   \\ 
& \phantom{\leq \sup \sum \;} \cdot \sum_{k_2 \in
  \Z}\min\setbig{1,|s(2^{-j/2}\xi_3+k_2 2^{-j\scp/2}\xi_1)|^{-2\dep}}. %\nonumber 
\end{align*}
Letting $\eta_1 = q\xi_1$ and using that $q\ge r$ and $q\ge s$, we obtain
\begin{multline}
  \Phi(\xi,0) \le \sup_{(\eta_1,\xi_2,\xi_3) \in \R^3} \sum_{j \ge 0}
  \min{\setbig{1,\absbig{2^{-j\scp/2}\eta_1}^{2\vap-2}}}
  \min\setbig{1,\abssmall{q'q^{-1}2^{-j\scp/2}\eta_1}^{-2\dep}}  \\
   \phantom{\sup_{(\eta_1,\xi_2,\xi_3) \in \R^3}} 
\cdot \sum_{k_1 \in \Z} \frac qr
  \min\setbig{1,\abssmall{rq^{-1}2^{-j\scp/2}\eta_1}}
  \min\setbig{1,\abssmall{r2^{-j/2}\xi_2+k_1 rq^{-1}2^{-j\scp/2}\eta_1}^{-2\dep}}   \\
   %\phantom{\leq \sup_{(\eta_1,\xi_2,\xi_3) \in \R^3} \;} 
\cdot \sum_{k_2 \in \Z} \frac qs
  \min\setbig{1,\abssmall{sq^{-1}2^{-j\scp/2}\eta_1}} \min\setbig{1,\abssmall{s2^{-j/2}\xi_3+k_2
  sq^{-1}2^{-j\scp/2}\eta_1}^{-2\dep}}. \label{eq:Phi0}
\end{multline}
By \eqref{eq:sk2},  the sum over $k_1 \in \Z$ in
\eqref{eq:Phi0} is bounded by $\frac{q}{r} C(2\dep)$.
Similarly, the sum over $k_2 \in \Z$ in
\eqref{eq:Phi0} is bounded by $\frac{q}{s} C(2\dep)$.
 Hence, we can continue \eqref{eq:Phi0} by
\begin{align*}
\Phi(\xi,0)
&\leq \frac{q^2}{rs} C(2\dep)^2 \sup_{\eta_1 \in \R} \sum_{j \ge 0}
\min{\setbig{1,\absbig{2^{-j\scp/2}\eta_1}^{2\vap-2}}} \min\setbig{1,\absbig{q'q^{-1}2^{-j\scp/2}\eta_1}^{-2\dep}}  \\
&= \frac{q^2}{rs} C(2\dep)^2 \sup_{\eta_1 \in \R} \Bigl(\sum_{j \ge 0}
\absbig{2^{-j\scp/2}\eta_1}^{2\vap-2} \charfct{\itvco{0}{1}}(\abssmall{2^{-j\scp/2}\eta_1}) +
\charfct{\itvco{1}{q/q'}}(\abssmall{2^{-j\scp/2}\eta_1})\\
&  \phantom{\le \frac{q^2}{rs} C(2\dep)^2 \sup_{\eta_1 \in \R} \sum \;\;\,} 
\absbig{q'q^{-1}2^{-j\scp/2}\eta_1}^{-2\dep} \charfct{\itvco{q/q'}{\infty}}(\abssmall{2^{-j\scp/2}\eta_1})\Bigl)\\
&\leq \frac{q^2}{rs} C(2\dep)^2 \sup_{\eta_1 \in \R} \Bigl(
\sum_{\abs{2^{-j\scp/2}\eta_1} \leq 1}\absbig{2^{-j\scp/2}\eta_1}^{2\vap-2}+ \sum_{j\ge 0}
\chi_{[1,\frac{q}{q'})}(|2^{-j\scp/2}\eta_1|) \\  %\hspace*{-0.5cm}
&  \phantom{\le \frac{q^2}{rs} C(2\dep)^2 \sup_{\eta_1 \in \R} \sum \;\;\,} 
+ \sum_{|q'q^{-1}2^{-j\scp/2}\eta_1| \ge 1} |q'q^{-1}2^{-j\scp/2}\eta_1|^{-2\dep}\Bigr).
\end{align*}
The claim \eqref{eq:Upp} now follows from (\ref{eq:quotientsum-finite}), (\ref{eq:quotientsum-inf})
and \eqref{eq:quotientsum-t-big}.
\end{proof}

The next result, Proposition~\ref{prop:estimateR}, exhibits how $R(c)$ depends on the parameters
$c_1$ and $c_2$ from the translation matrix $M_c$. In particular, we see that the size of $R(c)$ can
be controlled by choosing $c_1$ and $c_2$ small. The result can be simplified as follows: For any
$\dep'$ satisfying $1<\dep'<\dep-2$, there exist positive constants $\kappa_1$ and $\kappa_2$
independent on $c_1$ and $c_2$ such that
\begin{equation*} %\label{eq:decayR}
R(c) \le %\kappa_1\left(
         %\frac{2c_1}{q'}\right)^{\dep}+\kappa_2\left(\frac{2qc_2}{q'\min{\{r,s\}}}\right)^{\dep-\dep'}
         %.
\kappa_1\, c_1^{\,\dep} + \kappa_2\, c_2^{\,\dep-\dep'}.
\end{equation*}
The constants  $\kappa_1$ and $\kappa_2$ depends on the parameters $q,q',r,s,\delta$ and $\gamma$, and the result below shows
exactly how this dependence is.  

\begin{proposition}
\label{prop:estimateR}
Let $\psi \in L^2(\R^3)$ be a $(\delta,\gamma)$-feasible shearlet for $\delta > 2\gamma > 6$, and
let the translation lattice parameters $c = (c_1,c_2)$ satisfy $c_1 \ge c_2
>0$. 
Then, for any $\dep'$ satisfying $1<\dep'<\dep-2$, we have
\begin{multline}\label{eq:upper_R}
R(c) \leq T_1 \bigl(8 \zeta(\dep-2) - 4
   \zeta(\dep-1) + 2 \zeta(\dep) \bigr) \\ + 
   3\min{\set{\ceil{\frac{c_1}{c_2}},2}} T_2 \bigl(16 \zeta(\dep-2) - 4
   \zeta(\dep-1)\bigr) + T_3  \bigl(24 \zeta(\dep-2) + 2 \zeta(\dep)\bigr),
\end{multline}
where
\begin{align*}
  T_1&= \frac{q^2}{rs} C(\dep)^2 \left(\frac{2c_1}{q'}
  \right)^\dep \Bigl( \ceil{\log_{2}\Bigl(\frac{q}{q'}\Bigr)}
  +\frac{1}{1-2^{-\vap+2\dep}}+\frac{1}{1-2^{-\dep}}\Bigr) \\
  T_2 &= \frac{q^2}{rs} C(\dep)C(\dep') \left(\frac{2qc_2}{q'\min\{r,s\}}
  \right)^{\dep-\dep'} \Bigl( 2\ceil{ \log_{2}\Bigl(\frac{q}{q'}\Bigr)}
  +\frac{1}{1-2^{-\vap+2\dep}} \\
& \phantom{= \frac{q^2}{rs} C(\dep)C(\dep') \left(\frac{2qc_2}{q'\min\{r,s\}}
  \right)^{\dep-\dep'}+} +\frac{1}{1-2^{-\dep}}+\frac{1}{1-2^{-\vap+\dep+\dep'}}+\frac{1}{1-2^{-\dep'}}\Bigr)
  \\
  T_3 &= \frac{q^2}{rs} C(\dep)^2 \left(\frac{2c_1}{q'}
  \right)^\dep \frac{1}{1-2^{-\dep}},
%      \label{eq:case2a}
\end{align*}
and $\zeta$ is the Riemann zeta function.  
\end{proposition}

\begin{proof}
The proof can be found the Appendix~\ref{sec:proof-proposition-Rc}.
\end{proof}

The tightness of the estimates of $R(c)$ in Proposition~\ref{prop:estimateR} are important for the
construction of shearlet frames in the next section since the estimated frame bounds will depend
heavily on the estimate of $R(c)$. If we allowed a cruder estimate of $R(c)$, the proof of
Proposition~\ref{prop:estimateR} could be considerably simplified; as we do not allow this, the
slightly technical proof is relegated to the appendix.

\subsection{Frame constructions}
\label{sec:frame-constructions}

The results in this section (except Corollary~\ref{proposition:upperbd}) are presented without proofs since these are
straightforward generalizations of results on cone-adapted shearlet frames for $L^2(\R^2)$ from \cite{KKL10a}. We first
formulate a general sufficient condition for the existence of pyramid-adapted shearlet frames.
\begin{theorem}
\label{thm:general-suff-condition}
Let $\psi \in L^2(\R^3)$ be a $(\delta,\gamma)$-feasible shearlet (associated with $\cP$) for $\delta > 2\gamma > 6$, and let the translation lattice
parameters $c = (c_1,c_2)$ satisfy $c_1 \ge c_2 >0$. %, and let $R(c)$ be defined as in \eqref{eq:defR}.
If $R(c) < L_{\mathit{inf}}$,
then $\Psi(\psi)$ is a frame for $\check L^2(\cP):=\setpropsmall{f\in L^2(\R^3)}{\supp \hat f \subset \cP}$ with frame bounds
$A$ and $B$ satisfying
\[
\frac{1}{|\det M_c|} [{L}_{\mathit{inf}} - R(c)] \le A \le B \le \frac{1}{|\det M_c|} [{L}_{\mathit{sup}} +
R(c)].
\]
\end{theorem}

Let us comment on the sufficient condition for the existence of shearlet frames in Theorem~\ref{thm:general-suff-condition}.
Firstly, to obtain a lower frame bound $A$, we choose a shearlet generator $\psi$ such that
\begin{equation}\label{eq:covering}
\cP \subset \bigcup_{j \ge 0} \bigcup_{k \in \Z^2} A_{2^j}S^T_{k} \Omega, 
\end{equation}
where 
\[
\Omega = \{ \xi \in \R^3 : \abssmall{\hat \psi(\xi)} > \rho \}, \,\, \text{for some} \,\, \rho > 0.
\]
For instance, one can choose $\Omega = [1,2]\times[-1/2,1/2]\times[-1/2,1/2]$ here. 
From \eqref{eq:covering}, we have ${L}_{\mathit{inf}} > \rho^2$. 
Secondly, note that $R(c) \rightarrow 0$ as $c_1 \to 0^{+}$ and $c_2 \to 0^{+}$  by Proposition~\ref{prop:estimateR} (see $T_1,T_2$, and $T_3$ in (5.7)). 
In particular, for a given ${L}_{\mathit{inf}}>0$, one can make $R(c)$ sufficiently small for some translation lattice parameter $c = (c_1,c_2)$ so that 
${L}_{\mathit{inf}} - R(c) > 0$. Finally, Proposition~\ref{prop:Lsupfinite} and \ref{prop:estimateR} imply the existence of an upper frame bound $B$. We refer to \cite{KLL11} for concrete examples with frame bound estimates.

By the following result we then have an explicitly given family of shearlets satisfying the
assumptions of Theorem~\ref{thm:general-suff-condition} at disposal.
\begin{theorem} 
\label{thm:compact-for-pyramid}
Let $K, L \in \N$ be such that $L \ge 10$ and $\frac{3L}{2} \le K \le 3L-2$, and define a shearlet $\psi \in L^2(\R^2)$ by
\[
\hat{\psi}(\xi) =
m_1(4\xi_1)\hat{\phi}(\xi_1)\hat{\phi}(2\xi_2)\hat\phi(2\xi_3), \quad
\xi = (\xi_1,\xi_2,\xi_3) \in \R^3,
\]
where $m_0$ is the low pass filter satisfying
\[
|m_0(\xi_1)|^2 = (\cos(\pi\xi_1))^{2K}\sum_{n=0}^{L-1} \genfrac(){0pt}{0}{K-1+n}{n} (\sin(\pi\xi_1))^{2n}, \quad \xi_1 \in \R,
\]
$m_1$ is the associated bandpass filter defined by
\[
|m_1(\xi_1)|^2 = |m_0(\xi_1+1/2)|^2, \quad \xi_1 \in \R,
\]
and $\phi$ is the scaling function given by
\[
\hat{\phi}(\xi_1) = \prod_{j=0}^{\infty} m_0(2^{-j}\xi_1), \quad \xi_1 \in \R.
\]
Then there exists a sampling constant $\hat c_1>0$ such that the shearlet system $\Psi(\psi)$ forms a frame for $\check
L^2(\cP)$ for any sampling matrix $M_c$ with $c=(c_1,c_2) \in (\R_+)^2$ and $c_2 \le c_1 \le \hat{c}_1$. Furthermore, the
corresponding frame bounds $A$ and $B$ satisfy
\[
\frac{1}{|\det(M_c)|} [{L}_{\mathit{inf}} - R(c)] \le A \le B \le \frac{1}{|\det(M_c)|} [{L}_{\mathit{sup}} +
R(c)],
\]
where $R(c) < {L}_{\mathit{inf}}$.
\end{theorem}

Theorem~\ref{thm:compact-for-pyramid} provides us with a family of compactly supported shearlet
frames for $\check L^2(\cP)$. For these shearlet systems there is a bias towards the $x_1$ axis,
especially at coarse scales, since they are defined for $\check L^2(\cP)$, and hence, the frequency
support of the shearlet elements overlaps more significantly along the $x_1$ axis. In order to
control the upper frame bound, it is therefore desirable to have a denser translation lattice in the
direction of the $x_1$ axis than in the other axis directions, \ie $c_1 \ge c_2$.

In the next result we extend the construction from Theorem~\ref{thm:compact-for-pyramid} for $\check
L^2(\cP)$ to all of $L^2(\R^3)$. We remark that this type of extension result differs from the
similar extension for band-limited (tight) shearlet frames since in the latter extension procedure
one needs to introduce artificial projections of the frame elements onto the pyramids in the Fourier
domain.
\begin{theorem} \label{thm:compact-completeframe} Let $\psi \in L^2(\R^3)$ be the shearlet with
  associated scaling function $\phi \in L^2(\R)$ introduced in
  Theorem~\ref{thm:compact-for-pyramid}, and set $\phi(x_1,x_2,x_3)=\phi(x_1)\phi(x_2)\phi(x_3)$,
  $\tilde{\psi}(x_1,x_2,x_3) = \psi(x_2,x_1,x_3)$, and $\breve{\psi}(x_1,x_2,x_3) =
  \psi(x_3,x_2,x_1)$. Then the corresponding shearlet system
  $\SH(\phi,\psi,\tilde{\psi},\breve{\psi};c,\scp)$ forms a frame for $L^2(\R^3)$ for the sampling
  matrices $M_c$, $\tilde{M}_c$, and $\breve{M}_c$ with $c=(c_1,c_2) \in (\R_+)^2$ and $c_2 \le c_1
  \le \hat{c}_1$.
\end{theorem}

For the pyramid $\cP$, we allow for a denser translation lattice $M_c \Z^3$ along the $x_1$ axis,
\ie $c_2 \le c_1$, precisely as in Theorem~\ref{thm:compact-for-pyramid}. For the other pyramids
$\tilde{\cP}$ and $\breve{\cP}$, we analogously allow for a denser translation lattice along the
$x_2$ and $x_3$ axes, respectively; since the position of $c_1$ and $c_2$ in $\tilde{M}_c$ and
$\breve{M}_c$ are changed accordingly, this still corresponds to $c_2 \le c_1$.

The final result of this section generalizes Theorem~\ref{thm:compact-completeframe} in the sense that it shows that not only the shearlet 
introduced in Theorem~\ref{thm:compact-for-pyramid}, but also any $(\delta,\gamma)$-feasible shearlet $\psi$
satisfying \eqref{eq:covering} generates a shearlet frame for $L^2(\R^3)$ provided that $\delta > 2\gamma > 6$.  
For this, we change the definition of $R(c)$, $L_{\mathit{inf}}$ and $L_{\mathit{sup}}$ in (\ref{eq:defiLinLsup}) and
  (\ref{eq:defR}) so that the essential infimum and supremum are taken over all of $\R^3$ and not only over the
  pyramid $\cP$, and we denote these new constants again by $R(c)$, $L_{\mathit{inf}}$ and $L_{\mathit{sup}}$.

\begin{corollary}\label{proposition:upperbd}
  Let $\psi \in L^2(\R^3)$ be a $(\delta,\gamma)$-feasible shearlet for $\delta > 2\gamma > 6$.
  Also, define $\tilde{\psi}$ and $\breve{\psi}$ as in Theorem~\ref{thm:compact-completeframe} and choose 
  $\phi \in L^2(\R^3)$ such that $|\hat \phi(\xi)| \lesssim (1+|\xi|)^{-\gamma}$.
  Suppose that $L_{\mathit{inf}} > 0$.
  Then $\SH(\phi,\psi,\tilde{\psi},\breve{\psi};c,\scp)$ forms a frame for $L^2(\R^3)$ for the sampling
  matrices $M_c$, $\tilde{M}_c$, and $\breve{M}_c$ for some translation lattice parameter $c = (c_1,c_2)$.
\end{corollary}
\begin{proof}
  The proofs of Proposition~\ref{prop:Lsupfinite} and \ref{prop:estimateR} show that the same estimate
  as in (\ref{eq:Upp}) and (\ref{eq:upper_R}) holds for our new $R(c)$ and $L_{\mathit{sup}}$; this
  is easily seen since the very first estimate in both these proofs extends the supremum from $\cP$
  to $ \R^3$. Furthermore, by Proposition~\ref{prop:estimateR}, one can choose $c = (c_1,c_2)$ such that  
  $L_{\mathit{inf}}-R(c) > 0$. Now, we have that $L_{\mathit{sup}}+R(c)$ is bounded and $L_{\mathit{inf}}-R(c) > 0$. Since
  $R(c)$ and $L_{\mathit{sup}}$ are associated to the $t_q$-terms and a discrete Calder\'on
  condition, respectively, following arguments as in \cite[\S 3.3.2]{Dau92} or \cite{KKL10a} show that frame bounds $A$ and $B$ exist and that
\[0 < (R(c)-L_{\mathit{inf}})/\det M_c \leq A \leq B \leq (R(c)+L_{\mathit{sup}})/\det M_c < \infty.\]
\end{proof}

\section{Optimal sparsity of 3D shearlets}
\label{sec:optimal-sparsity-3d}
Having 3D shearlet frames with compactly supported generators at hand
by Theorem~\ref{thm:compact-completeframe}, we turn to
sparse approximation of cartoon-like images by these shearlet systems.

\subsection{Sparse approximations of 3D Data}
\label{sec:sparse-appr-3d}

Suppose $\SH(\phi,\psi,\tilde{\psi},\breve{\psi};c,\scp)$ forms a frame for $L^2(\R^3)$ with frame
bounds $A$ and $B$. Since the shearlet system is a countable set of functions, we can denote it by
$\SH(\phi,\psi,\tilde{\psi},\breve{\psi};c,\scp) = \set{\sigma_i}_{i \in I}$ for some countable
index set $I$. We let $\setsmall{\tilde{\sigma}_i}_{i \in I}$ be the canonical dual frame of
$\set{\sigma_i}_{i \in I}$. As our $N$-term approximation $f_N$ of a cartoon-like image $f \in
\cE^\beta_\alpha(\R^3)$ by the frame $\SH(\phi,\psi,\tilde{\psi};c)$, we then take, as in
Equation~(\ref{eq:frame-n-term-largest}),
\[
f_N = \sum_{i \in I_N} c_i\, \tilde{\sigma}_i, \qquad c_i=\innerprod{f}{\sigma_i},
\]
where $(\innerprod{f}{\sigma_i})_{i \in I_N}$ are the $N$ largest coefficients
$\innerprod{f}{\sigma_i}$ in magnitude.

The benchmark for optimal sparse approximations that we are aiming for is, as we showed in
Section~\ref{sec:optimal-sparsity}, for all $f= f_0 + \chi_B f_1 \in \cE^\beta_\alpha(\R^3)$,
\[
  \norm[L^2]{f-f_N}^2 \lesssim N^{-\scp/2} \qquad \text{as $N \to \infty$,} 
\]
and
\[
\abs{c^*_n} \lesssim  n^{-\frac{\alpha+2}{4}}, \qquad \text{as $n \to \infty$,}
\]
where $c^\ast=(c^\ast_n)_{n\in \N}$ is a decreasing (in modulus) rearrangement of $c=(c_i)_{i \in
  I}$. The following result shows that compactly supported pyramid-adapted, hybrid shearlets almost deliver
this approximation rate for all $1<\alpha\le\beta \le 2$. We remind the reader that the parameters
$\cp$ and $\mu$, suppressed in our notation $\cE^\beta_\alpha(\R^3)$, are 
bounds of the homogeneous H\"older $\dot{C}^\alpha$ norm of the radius function for the discontinuity
surface $\pa B$ and of the $C^\beta$ norms of $f_0$ and $f_1$, respectively.

\begin{theorem}
\label{thm:3d-opt-sparse}
Let $\alpha \in \itvoc{1}{2}$, $c \in (\R_+)^2$, and let $\phi, \psi,
\tilde{\psi}, \breve{\psi} \in L^2(\R^3)$ be compactly supported.
Suppose that, for all $\xi = (\xi_1,\xi_2,\xi_3) \in \R^3$, the
function $\psi$ satisfies:
\begin{romannum}
\item $|\hat\psi(\xi)| \le C \cdot \min\lbrace 1,|\xi_1|^{\vap}\rbrace \cdot
  \min\lbrace 1,|\xi_1|^{-\dep}\rbrace \cdot \min \lbrace 1,|\xi_2|^{-\dep} \rbrace \cdot
  \min\lbrace 1,|\xi_3|^{-\dep}\rbrace $,
\item $\left|\frac{\partial}{\partial \xi_i}\hat \psi(\xi)\right| \le
  |h(\xi_1)| \cdot \left(1+\frac{|\xi_2|}{|\xi_1|}\right)^{-\dep}
  \left(1+\frac{|\xi_3|}{|\xi_1|}\right)^{-\dep}, \qquad i=2,3,$
  % \item $\left|\frac{\partial}{\partial \xi_3}\hat \psi(\xi)\right|
  %   \le |h(\xi_1)| \cdot
  %   \left(1+\frac{|\xi_3|}{|\xi_1|}\right)^{-\dep}$,
\end{romannum}
where $\vap > 8$, $\dep \ge 4$, $h \in L^1(\R)$, and $C$ a constant,
and suppose that $\tilde{\psi}$ and $\breve{\psi}$ satisfy analogous
conditions with the obvious change of coordinates. Further, suppose
that the shearlet system
$\SH(\phi,\psi,\tilde{\psi},\breve{\psi};c,\scp)$ forms a frame for
$L^2(\R^3)$.

Let $\tau=\tau(\alpha)$ be given by
\begin{align}\label{eq:eps-of-alpha}
  \tau(\alpha)=\frac{3 (2-\alpha)(\alpha-1)(\alpha+2)}{2 (9
    \alpha^2+17 \alpha-10)},
\end{align}
and let $\beta \in \itvcc{\alpha}{2}$. Then, for any $\cp,\mu > 0$, the
shearlet frame $\SH(\phi,\psi,\tilde{\psi},\breve{\psi};c,\scp)$
provides nearly optimally sparse approximations of functions $f \in
\cE_\alpha^\beta(\R^3)$ in the sense that 
\begin{equation}\label{eq:3d-approx-rate}
  \norm[L^2]{f-f_N}^2 \lesssim
  \left\{\begin{aligned} 
      N^{-\scp/2+\tau}, &\qquad    \text{if }\beta \in \itvco{\alpha}{2},\\
      N^{-1}\,(\log{N})^2, &\qquad  \text{if }  \beta=\alpha=2,
    \end{aligned}
  \right\} 
  \text{ as $N \to \infty$,} 
\end{equation}
where $f_N$ is the N-term approximation obtained by choosing
the N largest shearlet coefficients of $f$, and
\begin{equation}
  \label{eq:3d-coeff-decay-rate}
%  \sup_{f \in \cE_\alpha^\beta(\R^3)} 
%\abs{\innerprod{f}{\psi_\lambda}}_{(N)}
\abs{c^*_n} \lesssim 
  \left\{\begin{aligned} 
      n^{-\frac{\alpha+2}{4}+\frac{\tau}{2}}, &\qquad    \text{if }\beta \in \itvco{\alpha}{2},\\
      n^{-1} \, \log n, &\qquad  \text{if }  \beta=\alpha=2,
    \end{aligned}
  \right\} 
  \text{ as $n \to \infty$,}
\end{equation}
 where $c=\setpropsmall{\innerprods{f}{\mathring{\psi}_\lambda}}{\lambda \in
     \varLambda, \mathring{\psi}=\psi, \mathring{\psi}=
     \tilde\psi  \text{ or } \breve\psi}$ and $c^\ast=(c^\ast_n)_{n\in \N}$ is a decreasing
   (in modulus) rearrangement of $c$.
\end{theorem}

 We postpone the proof of Theorem~\ref{thm:3d-opt-sparse} until Section~\ref{sec:proof-theorem-1}. The sought optimal approximation error rate in~(\ref{eq:3d-approx-rate}) was $N^{-\alpha/2}$, hence
for $\alpha =2$ the obtained rate (\ref{eq:3d-approx-rate}) is almost optimal in the sense that it
is only a polylog factor $(\log N)^2$ away from the optimal rate. However, for $\alpha \in
\itvoo{1}{2}$ we are a power of $N$ with exponent $\tau$ away from the optimal rate. The exponent
$\tau$ is close to negligible; in particular, we have that $0<\tau(\alpha)<0.04$ for $\alpha \in
\itvoo{1}{2}$ and that $\tau(\alpha)\to 0$ for $\alpha \to 1+$ or $\alpha \to 2-$, see also
Figure~\ref{fig:optimality-gap}. The approximation error rate (\ref{eq:3d-approx-rate}) obtained for
$\alpha <2$ can also be expressed as
\[ 
\norm[L^2]{f-f_N}^2 =O(N^{-\frac{6 \alpha^3+7 \alpha^2-11 \alpha+6}{9
    \alpha^2+17 \alpha-10}}),
\]
which, of course, still is an $\tau=\tau(\alpha)$ exponent away from being optimal. Let us mention that a
slightly better estimate $\tau(\alpha)$ can be obtained satisfying $\tau(\alpha)<0.037$ for $\alpha
\in \itvoo{1}{2}$, but the expression becomes overly complicated; we can, however, with the current
proof of Theorem~\ref{thm:3d-opt-sparse} not make $\tau(\alpha)$ arbitrarily small. As $\alpha \to 2+$ we see that the exponent $-\alpha/2+\tau \to -1$, however, for $\alpha=\beta=2$ an additional log factor appears in the approximation error rate.
This jump in the error rate is a consequence of our proof technique, and it might be that a truly optimal decay rate depends continuously on the model parameters.

If the smoothness of the discontinuity surface $C^\alpha$ of a 3D cartoon-like image approaches $C^1$ smoothness, we loose so much directional information that we do not gain anything by using a directional representation system, and we might as well use a standard wavelet system, see Example~\ref{example:fourier-wavelets} and Figure~\ref{fig:optimality-gap}(a). However, as the discontinuity surface becomes smoother, that is, as $\alpha$ approaches $2$, we acquire enough directional information about the singularity for directional representation systems to become a better choice; exactly how one should adapt the directional representation system to the smoothness of the singular is seen from the the definition of our hybrid shearlet system.

The constants in the expressions in (\ref{eq:3d-approx-rate}) depend only on $\nu$ and $\mu$,
where $\nu$ is a bound of the homogeneous H\"older norm for the radius function $\rho \in
\dot{C}^\alpha$ associated with the discontinuity surface $\pa B$ and $\mu$ is the bound of the
H\"older norm of $f_1,f_2 \in C^\beta(\R^3)$ with $f = f_0 + \chi_B f_1$, see also
Definition~\ref{def:cartoon-3d}. We remark that these constants grow with $\nu$ and $\mu$ hence we
cannot allow $f = f_0 + \chi_B f_1$ with only $\norm[C^\beta]{f_i}<\infty$.

\begin{figure}[ht]
% \vspace*{-.5em}
\centering
\subfloat[Graph of $\frac{6 \alpha^3+7 \alpha^2-11 \alpha+6}{9
    \alpha^2+17 \alpha-10}$ and the optimal rate $\alpha/2$ (dashed) as a
function of $\alpha$. ]{%
%\psfragfig{./rate-and-opt-rate} 
\includegraphics{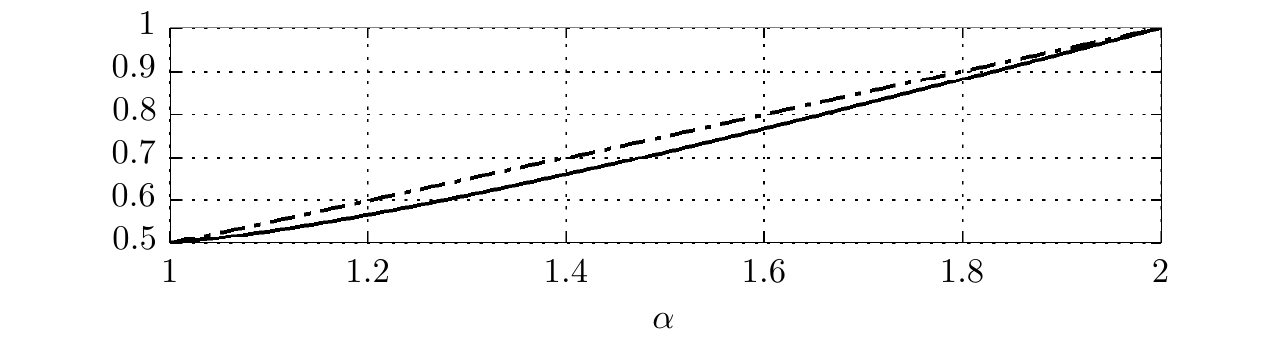}
\label{fig:rate-and-opt-rate}
}\\
\subfloat[Graph of $\tau(\alpha)$ given by \eqref{eq:eps-of-alpha}.]{%
\includegraphics{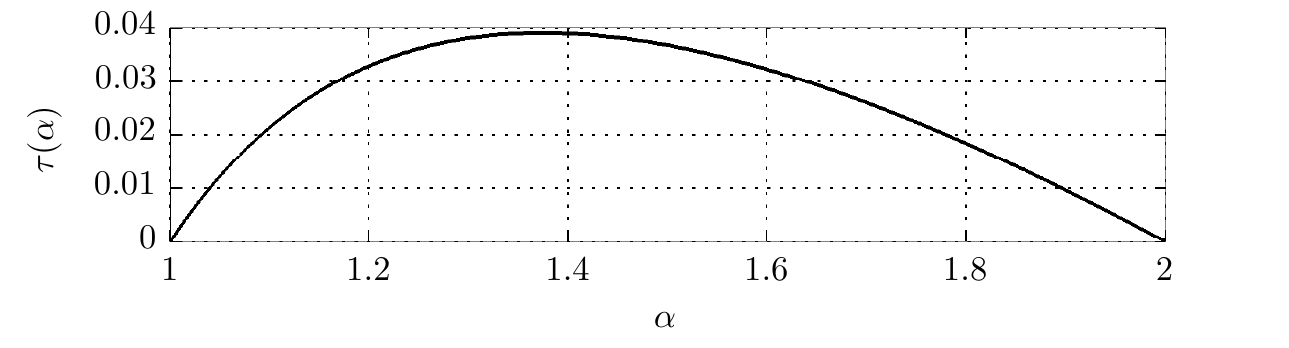}
%\psfragfig{./eps-from-opt} 
\label{fig:eps-from-opt}
}
\caption{The optimality gap for $\beta \in \itvco{\alpha}{2}$:
 Figure~\ref{fig:rate-and-opt-rate} shows the optimal and the obtained
 rate, and Figure~\ref{fig:eps-from-opt} their difference $\tau(\alpha)$.}
\label{fig:optimality-gap} 
\end{figure}

Let us also briefly discuss the two decay assumptions in the frequency domain on the shearlet
generators in Theorem~\ref{thm:3d-opt-sparse}. Condition~(i) says that $\psi$ is
$(\delta,\gamma)$-feasible and can be interpreted as both a condition ensuring almost separable
behavior and controlling the effective support of the shearlets in frequency domain as well
as a moment condition along the $x_1$ axis, hence enforcing directional selectivity. Condition~(ii),
together with (i), is a weak version of a directional vanishing moment condition (see \cite{DV05}
for a precise definition), which is crucial for having fast decay of the shearlet coefficients when
the corresponding shearlet intersects the discontinuity surface. We refer to the exposition
\cite{KLL11} for a detailed explanation of the necessity of conditions~(i) and (ii). Conditions~(i)
and (ii) are rather mild conditions on the generators; in particular, shearlets constructed by
Theorem~\ref{thm:compact-for-pyramid} and~\ref{thm:compact-completeframe}, with extra assumptions on
the parameters $K$ and $L$, will indeed satisfy~(i) and (ii) in Theorem~\ref{thm:3d-opt-sparse}. To
compare with the optimality result for band-limited generators we wish to point out that
conditions~(i) and (ii) are obviously satisfied for band-limited generators.

Theorem~1.3 in \cite{KL10} shows optimal sparse approximation of compactly supported shearlets in
2D. Theorem~\ref{thm:3d-opt-sparse} is similar in spirit to Theorem~1.3 in \cite{KL10}, but for the
three-dimensional setting. However, as opposed to the two-dimensional setting, anisotropic
structures in three-di\-men\-si\-onal data comprise of \emph{two} morphological different types of
structures, namely surfaces \emph{and} curves. It would therefore be desirable to have a similar
optimality result for our extended 3D image class $\cE_{\scp,L}^\beta(\R^3)$ which also allows types
of \emph{curve-like} singularities. Yet, the pyramid-adapted shearlets introduced in
Section~\ref{sec:pyram-adapt-shearl} are plate-like and thus, a priori, not well-suited for
capturing such one-dimensional singularities. However, these plate-like shearlet systems still
deliver the nearly optimal error rate as the following result shows. The proof of the result is postponed to Section~\ref{sec:proof-theorem-2}.

\begin{theorem}
\label{thm:3d-opt-sparse-piecewise}
Let $\scp \in\itvoc{1}{2}$, $c \in (\R_+)^2$, and let $\phi, \psi, \tilde{\psi}, \breve{\psi}
\in L^2(\R^3)$ be compactly supported. For each $\kappa \in [-1,1]$ and $x_3
\in \R$, define $g^0_{\kappa,\, x_3} \in L^2(\R^2)$ by
\[
g^0_{\kappa,\, x_3}(x_1,x_2) = \psi(x_1,x_2,\kappa x_2 + x_3),
\]
and, for each $\kappa \in [-1,1]$ and $x_2 \in \R$, define $g^1_{\kappa,\, x_2} \in L^2(\R^2)$ by
\[
g^1_{\kappa,\, x_2}(x_1,x_3) = \psi(x_1,\kappa x_3 + x_2,x_3).
\]
Suppose that, for all $\xi = (\xi_1,\xi_2,\xi_3) \in \R^3$, $\kappa \in [-1,1]$, and $x_2, x_3 \in \R$, the
function $\psi$ satisfies:
%for all $\xi =
 %  (\xi_1,\xi_2,\xi_3) \in \R^3$, the shearlet $\psi$ satisfies%\\[-1.25ex]
   \begin{romannum}
   \item $|\hat\psi(\xi)| \le C \cdot \min\{1,|\xi_1|^{\delta}\} \cdot
     \min\{1,|\xi_1|^{-\gamma}\} \cdot \min\{1,|\xi_2|^{-\gamma}\} \cdot
     \min\{1,|\xi_3|^{-\gamma}\}$, \vspace{0.2em}
   \item $ \abs{\bigl(\frac{\partial}{\partial
         \xi_2}\bigr)^{\ell}\hat g^0_{\kappa,\, x_3}(\xi_1,\xi_2)}
     \le |h(\xi_1)| \cdot
     \left(1+\frac{|\xi_2|}{|\xi_1|}\right)^{-\gamma} \quad \text{for}
     \,\, \ell =0,1,$  \vspace{0.2em}
   \item $\abs{\bigl(\frac{\partial}{\partial \xi_3}\bigr)^{\ell}\hat g^1_{\kappa,\,x_2}(\xi_1,\xi_3)}
     \le |h(\xi_1)| \cdot
     \left(1+\frac{|\xi_3|}{|\xi_1|}\right)^{-\gamma} \quad \text{for} \,\, \ell =0,1,$
   \end{romannum}
   where $\delta > 8$, $\gamma \ge 4$, $h \in L^1(\R)$, and $C$ a
   constant, and suppose that $\tilde{\psi}$ and $\breve{\psi}$
   satisfy analogous conditions with the obvious change of
   coordinates. Further, suppose that the shearlet system
   $\SH(\phi,\psi,\tilde{\psi},\breve{\psi};c,\scp)$ forms a frame for
   $L^2(\R^3)$.

   Let $\beta \in \itvcc{\alpha}{2}$. Then, for any $\cp > 0$, $L >0$, and $\mu
   >0$, the shearlet frame
   $\SH(\phi,\psi,\tilde{\psi},\breve{\psi};c,\scp)$ provides
   nearly optimally sparse approximations of functions $f \in
   \cE_{\scp,L}^\beta(\R^3)$  
in the sense that 
\begin{equation*}%\label{eq:3d-approx-rate}
  \norm[L^2]{f-f_N}^2 \lesssim
  \left\{\begin{aligned} 
      N^{-\scp/2+\tau}, &\qquad    \text{if }\beta \in \itvco{\alpha}{2},\\
      N^{-1}\,(\log{N})^2, &\qquad  \text{if }  \beta=\alpha=2,
    \end{aligned}
  \right\} 
  \text{ as $N \to \infty$,} 
\end{equation*}
and
\begin{equation*}
\abs{c^*_n} \lesssim 
  \left\{\begin{aligned} 
      n^{-\frac{\alpha+2}{4}+\frac{\tau}{2}}, &\qquad    \text{if }\beta \in \itvco{\alpha}{2},\\
      n^{-1} \, \log n, &\qquad  \text{if }  \beta=\alpha=2,
    \end{aligned}
  \right\} 
  \text{ as $n \to \infty$,}
\end{equation*}
where $\tau=\tau(\alpha)$ is given by \eqref{eq:eps-of-alpha}.
\end{theorem}

We remark that there exist numerous examples of $\psi, \tilde{\psi}$,
and $\breve{\psi}$ satisfying the conditions (i) and (ii) in
Theorem~\ref{thm:3d-opt-sparse} and the conditions (i)-(iii) in
Theorem~\ref{thm:3d-opt-sparse-piecewise}. One large class of examples
are separable generators $\psi, \tilde{\psi}, \breve{\psi} \in
L^2(\R^3)$, \ie
\[\psi(x)=\eta(x_1) \varphi(x_2)\varphi(x_3), \quad
\tilde\psi(x)=\varphi(x_1) \eta(x_2)\varphi(x_3), \quad
\breve\psi(x)=\varphi(x_1) \varphi(x_2)\eta(x_3),\] where $\eta,
\varphi \in L^2(\R)$ are compactly supported functions satisfying:
\begin{romannum}
\item $|\hat\eta(\omega)| \le C_1 \cdot \min\{1,|\omega|^{\delta}\}
  \cdot \min\{1,|\omega|^{-\gamma}\}$, \vspace{0.2em}
\item $\abs{\bigl(\frac{\partial}{\partial \omega}\bigr)^{\ell}\hat
    \varphi(\omega)} \leq C_2 \cdot \min\{1,|\omega|^{-\gamma}\}
  \quad$ for $\ell=0,1$,
\end{romannum}
for $\omega \in \R$, where $\alpha > 8$, $\gamma \ge 4$, and $C_1,
C_2$ are constants. Then it is straightforward to check that the
shearlet $\psi$ satisfies the conditions (i)-(iii) in
Theorem~\ref{thm:3d-opt-sparse-piecewise} and $\tilde{\psi},
\breve{\psi}$ satisfy analogous conditions as required in
Theorem~\ref{thm:3d-opt-sparse-piecewise}. Thus, we have the following
result.

\begin{corollary}
  \label{thm:3d-opt-sparse-piecewise-separable}
  Let $\alpha \in \itvoc{1}{2}$, $c \in (\R_+)^2$, and let $\eta,
  \varphi \in L^2(\R)$ be compactly supported functions satisfying:
  \begin{romannum}
  \item $|\hat\eta(\omega)| \le C_1 \cdot \min\set{1,|\omega|^{\vap}}
    \cdot \min\set{1,|\omega|^{-\dep}}$,
  \item $\abs{\bigl(\frac{\partial}{\partial \omega}\bigr)^{\ell}\hat
      \varphi(\omega)} \leq C_2 \cdot \min\set{1,|\omega|^{-\dep}} \quad$
    for $\ell=0,1$,
  \end{romannum}
  for $\omega \in \R$, where $\vap > 8$, $\dep \ge 4$, and $C_1$ and
  $C_2$ are constants. Let $\phi \in L^2(\R^3)$ be compactly
  supported, and let $\psi, \tilde{\psi}, \breve{\psi} \in L^2(\R^3)$
  be defined by:
  \[
  \psi(x)=\eta(x_1) \varphi(x_2)\varphi(x_3), \quad
  \tilde\psi(x)=\varphi(x_1) \eta(x_2)\varphi(x_3), \quad
  \breve\psi(x)=\varphi(x_1) \varphi(x_2)\eta(x_3).
  \]
  Suppose that the shearlet system
  $\SH(\phi,\psi,\tilde{\psi},\breve{\psi};c,\scp)$ forms a frame for
  $L^2(\R^3)$.

  Let $\beta \in \itvcc{\alpha}{2}$. Then, for any $\cp > 0$, $L >0$,
  and $\mu >0$ , the shearlet frame
  $\SH(\phi,\psi,\tilde{\psi},\breve{\psi};c,\scp)$ provides nearly
  optimally sparse approximations of functions $f \in
  \cE_{\scp,L}^\beta(\R^3)$ in the sense that 
\begin{equation*}
  \norm[L^2]{f-f_N}^2 \lesssim
  \left\{\begin{aligned} 
      N^{-\scp/2+\tau}, &\qquad    \text{if }\beta \in \itvco{\alpha}{2},\\
      N^{-1}\,(\log{N})^2, &\qquad  \text{if }  \beta=\alpha=2,
    \end{aligned}
  \right\} 
  \text{ as $N \to \infty$,} 
\end{equation*}
and 
\begin{equation*}
\abs{c^*_n} \lesssim 
  \left\{\begin{aligned} 
      n^{-\frac{\alpha+2}{4}+\frac{\tau}{2}}, &\qquad    \text{if }\beta \in \itvco{\alpha}{2},\\
      n^{-1} \, \log n, &\qquad  \text{if }  \beta=\alpha=2,
    \end{aligned}
  \right\} 
  \text{ as $n \to \infty$,}
\end{equation*}
  where $\tau=\tau(\alpha)$ is given by \eqref{eq:eps-of-alpha}.
\end{corollary}

In the remaining sections of the paper we will prove
Theorem~\ref{thm:3d-opt-sparse} and
Theorem~\ref{thm:3d-opt-sparse-piecewise}.

\subsection{General Organization of the Proofs of Theorems~\ref{thm:3d-opt-sparse} and \ref{thm:3d-opt-sparse-piecewise}}
\label{sec:general-orga-proofs}
Fix $\alpha \in \itvoc{1}{2}$ and $c \in (\R_+)^2$, and take $B \in \mathit{STAR}^\alpha(\nu)$ and
$f = f_0 + \chi_B f_1 \in \cE_\scp^\beta(\R^3)$. Suppose
$\SH(\phi,\psi,\tilde{\psi},\breve{\psi};c,\scp)$ satisfies the hypotheses of
Theorem~\ref{thm:3d-opt-sparse}. Then by condition~(i) the generators $\psi,\tilde{\psi}$ and
$\breve{\psi}$ are absolute integrable in frequency domain hence continuous in time domain and
therefore of finite max-norm $\norm[L^\infty]{\cdot}$. Let $A$ denote the lower frame bound of
$\SH(\phi,\psi,\tilde{\psi},\breve{\psi};c,\scp)$.

Without loss of generality we can assume the scaling index $j$ to be sufficiently large. To see this
note that $\supp f \subset \itvcc{0}{1}^3$ and all elements in the shearlet frame
$\SH(\phi,\psi,\tilde{\psi},\breve{\psi};c,\scp)$ are compactly supported making the number of
nonzero coefficients below a fixed scale $j_0$ finite. Since we are aiming for an asymptotic
estimate, this finite number of coefficients can be neglected. This, in particular, means that we do
not need to consider frame elements from the low pass system $\Phi(\phi;c)$. Furthermore, it
suffices to consider shearlets $\Psi(\psi)=\{\psi_{j,k,m}\}$ associated with the pyramid $\cP$
since the frame elements $\tilde{\psi}_{j,k,m}$ and $\breve{\psi}_{j,k,m}$ can be handled
analogously.

To simplify notation, we denote our shearlet elements by $\psi_\lambda$, where $\lambda = (j,k,m)$ is
indexing scale, shear, and position. We let $\Lambda_j$ be the indexing sets of shearlets in
$\Psi(\psi)$ at scale $j$, \ie
\[
\Psi(\psi) = \{\psi_{\lambda} : \lambda \in \Lambda_j, j\ge 0\},
\]
and collect these indices cross scales as
\[
\Lambda = \bigcup_{j=0}^\infty \Lambda_j.
\]

Our main concern will be to derive appropriate estimates for the shearlet coefficients
$\setprop{\innerprod{f}{\psi_{\lambda}}}{\lambda \in \Lambda}$ of $f$. Let $c(f)^*_n$ denote the
$n$th largest shearlet coefficient $\innerprod{f}{\psi_{\lambda}}$ in absolute value. As mentioned
in Section~\ref{sec:sparsity}, to obtain the sought estimate on $\norm[L^2]{f-f_N}$ in
(\ref{eq:3d-approx-rate}), it suffices (by Lemma~\ref{lemma:n-term-frame-approx}) to show that the
$n$th largest shearlet coefficient $c(f)^*_{\,n}$ decays as specified by (\ref{eq:3d-coeff-decay-rate}).

To derive the estimate in (\ref{eq:3d-coeff-decay-rate}), we will study two separate cases.
The first case for shearlet elements $\psi_{\lambda}$ that do not interact with the discontinuity
surface, and the second case for those elements that do.
\begin{description}
\item[Case 1.] The compact support of the shearlet
  $\psi_{\lambda}$ does not intersect the boundary of the set $B$,
  \ie $\abs{\supp \psi_{\lambda} \cap \partial B} = 0$.
  % $\intt(\supp(\psi_{\lambda})) \cap \partial B = \emptyset.$
\item[Case 2.] The compact support of the shearlet
  $\psi_{\lambda}$ does intersect the boundary of the set $B$, \ie
  $\abs{\supp \psi_{\lambda} \cap \partial B} \neq 0 $.
  % $\intt(\supp(\psi_{\lambda})) \cap \partial B \ne \emptyset.$
\end{description}

For \emph{Case 1} we will not be concerned with decay estimates of single coefficients
$\innerprod{f}{\psi_\lambda}$, but with the decay of sums of coefficients over several scales and
all shears and translations. The frame property of the shearlet system, the Sobolev smoothness of
$f$ and a crude counting argument of the cardinal of the essential indices $\lambda$ will basically
be enough to provide the needed approximation rate. We refer to Section~\ref{sec:smoothpart} for the
exact procedure.

For \emph{Case 2} we need to estimate each coefficient $\innerprod{f}{\psi_\lambda}$ individually
and, in particular, how $\abs{\innerprod{f}{\psi_\lambda}}$ decays with scale $j$ and shearing $k$.
We assume, in the remainder of this section, that $f_0=0$ whereby $f = \chi_B f_1$. Depending on the orientation
of the discontinuity surface, we will split Case~2 into several subcases. The estimates in each
subcase will, however, follow the same principle: Let
\[M = \supp \psi_\lambda \cap B.
\] 
Further, let $H$ be an affine hyperplane that intersects $M$ and thereby divides $M$ into two sets $M_t$ and
$M_l$. We thereby have that
\[
\innerprod{f}{\psi_\lambda} = \innerprod{\chi_{M_t} f}{\psi_\lambda} + \innerprod{\chi_{M_l} f}{\psi_\lambda}.
\]  
The hyperplane will be chosen in such way that $\vol{ M_t}$ is sufficiently small. In particular,
$\vol{ M_t}$ should be small enough so that the following estimate
\[ \abs{\innerprod{\chi_{M_t}f}{\psi_\lambda}} \le \norm[L^\infty]{f}
\norm[L^\infty]{\psi_\lambda} \vol{ M_t} \le \mu \, 2^{j(\alpha+2)/4}
\vol{M_t} 
\] 
does not violate~(\ref{eq:3d-coeff-decay-rate}). We call estimates of this form, where we have
restricted the integration to a small part $M_t$ of $M$, \emph{truncated} estimates (or the truncation
term).

For the other term $\innerprod{\chi_{M_l} f}{\psi_\lambda}$ we will have to integrate over a
possibly much large part $M_l$ of $M$. To handle this we will use that $\psi_\lambda$ only interacts
with the discontinuity of $\chi_{M_l} f$ on a affine hyperplane inside $M$. This part of the
estimate is called the \emph{linearized} estimate (or the linearization term) since the
discontinuity surface in $\innerprod{\chi_{M_l} f}{\psi_\lambda}$ has been reduced to a linear
surface. In $\innerprod{\chi_{M_l} f}{\psi_\lambda}$ we are integrating over three variables, and we
will as the inner integration always choose to integrate along lines parallel to the ``singularity''
hyperplane $H$. The important point here is that along all these line integrals, the function $f$ is
$C^\beta$-smooth without discontinuities on the entire interval of integration. This is exactly the
reason for removing the $M_t$-part from $M$. Using the Fourier slice theorem we will then turn the
line integrations along $H$ in the spatial domain into two-dimensional plane integrations the
frequency domain. The argumentation is as follows: Consider $g:\R^3 \to \C$ compactly supported and
continuous, and let $p: \R^2\to \C$ be a projection of $g$ onto, say, the $x_2$ axis, \ie
$p(x_1,x_3)=\int_\R g(x_1,x_2,x_3)d x_2$. This immediately implies that $\hat p(\xi_1,\xi_3) = \hat
g(\xi_1,0,\xi_3)$ which is a simplified version of the Fourier slice theorem. By an inverse Fourier
transform, we then have
\begin{equation}
\int_\R g(x_1,x_2,x_3)d x_2 = p(x_1,x_3) = \int_{\R^2} \hat
g(\xi_1,0,\xi_3) \expo{2\pi i \innerprod{(x_1,x_3)}{(\xi_1,\xi_3)}}
\mathrm{d}\xi_1 \mathrm{d}\xi_3, \label{eq:fourier-slice-thm}
\end{equation}
and hence
\begin{equation}
 \int_\R  \abs{g(x_1,x_2,x_3)}d x_2 = \int_{\R^2} \abs{\hat
g(\xi_1,0,\xi_3)} \mathrm{d}\xi_1 \mathrm{d}\xi_3. \label{eq:fourier-slice-thm-2}
\end{equation}
The left-hand side of (\ref{eq:fourier-slice-thm-2}) corresponds to line integrations of $g$
parallel to the $x_1 x_3$ plane. By applying shearing to the coordinates $x \in \R^3$, we can
transform $H$ into a plane of the form $\setprop{x \in \R^3}{x_1=C_1, x_3=C_2}$, whereby we can
apply~(\ref{eq:fourier-slice-thm-2}) directly.

Finally, the decay assumptions on $\hat \psi$ in Theorem~\ref{thm:3d-opt-sparse} are then used to
derive decay estimates for $\abs{\innerprod{f}{\psi_\lambda}}$. Careful counting arguments will
enable us to arrive at the sought estimate in~(\ref{eq:3d-coeff-decay-rate}). We refer to
Section~\ref{sec:discontinuity} for a detailed description of Case~2.

With the sought estimates derived in Section~\ref{sec:smoothpart} and \ref{sec:discontinuity}, we
then prove Theorem~\ref{thm:3d-opt-sparse} in Section~\ref{sec:proof-theorem-1}. The proof of
Theorem~\ref{thm:3d-opt-sparse-piecewise} will follow the exact same organization and setup as
Theorem~\ref{thm:3d-opt-sparse}. Since the proofs are almost identical, in the proof of
Theorem~\ref{thm:3d-opt-sparse-piecewise}, we will only focus on issues that need to be handled
differently. The proof of Theorem~\ref{thm:3d-opt-sparse-piecewise} is presented in
Section~\ref{sec:proof-theorem-2}.

We end this section by fixing some notation used in the sequel. Since we are concerned with an
asymptotic estimate, we will often simply use $C$ as a constant although it might differ for each
estimate; sometimes we will simply drop the constant and use $\lesssim $
instead. 
 We will also use the notation $r_j \sim s_j$ for $r_j, s_j
\in \R$, if $C_1 \, r_j \le s_j \le C_2 \, r_j$ with constants
$C_1$ and $C_2$ independent on the scale $j$. 

\section{Analysis of shearlet coefficients away from the discontinuity surface}
\label{sec:smoothpart}

In this section we derive estimates for the decay rate of the shearlet coefficients
$\innerprod{f}{\psi_\lambda}$ for \emph{Case~1} described in the previous section. Hence, we
consider shearlets $\psi_\lambda$ whose support does not intersect the discontinuity surface $\pa
B$. This means that $f$ is $C^\beta$-smooth on the entire support of $\psi_\lambda$, and we can
therefore simply analyze shearlet coefficients $\innerprod{f}{\psi_\lambda}$ of functions $f \in
C^\beta(\R^3)$ with $\supp f \subset \itvcc{0}{1}^3$. The main result of this section,
Proposition~\ref{cor:main-smooth}, shows that $\norm[L^2]{f-f_N}^2=O(N^{-2\beta/3+\eps})$ as $N\to
\infty$ for any $\eps$, where $f_N$ is our $N$-term shearlet approximation. The result follows
easily from Proposition~\ref{prop:main-smooth} which is similar in spirit to
Proposition~\ref{cor:main-smooth}, but for the case where $f \in H^\beta$. The proof builds on
Lemma~\ref{lemma:smoothpart} which shows that the system $\Psi(\psi)$ forms a weighted Bessel-like
sequence with strong weights such as $(2^{\scp\beta j})_{j\ge 0}$ provided that the shearlet $\psi$
satisfies certain decay conditions. Lemma~\ref{lemma:smoothpart} is, in turn, proved by transferring
Sobolev differentiability of the target function to decay properties in the Fourier domain and
applying Lemma~\ref{proposition:upperbd}.
\begin{lemma}\label{lemma:smoothpart}
  Let $g \in H^\beta(\R^3)$ with $\supp g \subset
  \itvcc{0}{1}^3$. Suppose that $\psi \in L^2(\R^3)$ is $(\delta,\gamma)$-feasible for  $\vap >2\dep +
\beta$, $\dep >3$. 
 Then there exists a constant $B>0$ such that
\[
\sum_{j=0}^{\infty}\sum_{\abs{k} \le \lceil 2^{j(\scp-1)/2} \rceil}\sum_{m \in
  \Z^3} 2^{\scp\beta j}\abs{\innerprod{ g}{\psi_{j,k,m}}}^2 \le B
\normsmall[L^2]{\pa^{(\beta,0,0)} g}^2,
\]
where $\pa^{(\beta,0,0)} g$ denotes the $\beta$-fractional partial derivative of $g=g(x_1,x_2,x_3)$ with respect to
$x_1$.
\end{lemma}

\begin{proof}
  Since $\psi\in L^2(\R^3)$ is $(\delta,\gamma)$-feasible, we can choose $\varphi \in L^2(\R^3)$ as
\[
(2 \pi i \xi_1)^\beta \hat \varphi(\xi) = \hat \psi(\xi) \qquad \text{for }
\xi \in \R^3,
\]
hence $\psi$ is the $\pa^{(\beta,0,0)}$-fractional derivative of $\varphi$. This definition is well-defined due to the
decay assumptions on $\hat\psi$. By definition of the fractional derivative, it follows that
\begin{multline*}
  \absbig{\innerprodbig{\pa^{(\beta,0,0)}g}{\varphi_{j,k,m}}}^2 = \abs{\innerprod{(2\pi i
      \xi_1)^\beta \hat g(\xi)}{\widehat{\varphi_{j,k,m}}}}^2 \\ =
  \absbig{\innerprodbig{g}{\pa^{(\beta,0,0)}\varphi_{j,k,m}}}^2 = 2^{\scp \beta j}
  \abs{\innerprod{g}{\psi_{j,k,m}}}^2,
\end{multline*}
where we have used that $\pa^{(\beta,0,0)}f_{j,k,m} = (2^{j\scp /2})^\beta
(\pa^{(\beta,0,0)}f)_{j,k,m}$ for $f\in H^\beta(\R^3)$. A straightforward computation shows that
$\varphi$ satisfies the hypotheses of Lemma~\ref{proposition:upperbd}, and an application of
Lemma~\ref{proposition:upperbd} then yields
%\[
\begin{align*}
\sum_{j=0}^{\infty}\sum_{|k| \le \lceil 2^{j(\scp-1)/2} \rceil}\sum_{m \in \Z^3} 2^{\scp\beta j} |\langle
g,\psi_{j,k,m}\rangle|^2
&=\sum_{j=0}^{\infty}\sum_{|k| \le \lceil 2^{j(\scp-1)/2}
  \rceil}\sum_{m \in \Z^3} 
  \absbig{\innerprodbig{\partial^{(\beta,0,0)} g}{\varphi_{j,k,m}}}^2
\\
&\le B \normsmall[L^2]{ \partial^{(\beta,0,0)} g}^2,
%\]
\end{align*}
which completes the proof.
\end{proof}

We are now ready to prove the following result. % Proposition \ref{prop:main-smooth}. 
\begin{proposition}
\label{prop:main-smooth}
Let $g \in H^\beta(\R^3)$ with $\supp g \subset \itvcc{0}{1}^3$. Suppose that $\psi \in L^2(\R^3)$ is compactly
supported and $(\delta,\gamma)$-feasible for $\vap >2\dep + \beta$ and $\dep >3$.
Then 
\[
\sum_{n>N} \abs{c(g)^*_n}^2 \lesssim  N^{-2\beta/3} \qquad
\text{as} \quad  N \to \infty,
\]
where $c(g)^*_n$ is the $n$th largest coefficient $\innerprod{g}{\psi_\lambda}$ in modulus for $\psi_\lambda
\in \Psi(\psi)$.
\end{proposition}

\begin{proof}
Set
\[
\tilde{\Lambda}_j = \{\lambda \in \Lambda_j : \supp \psi_{\lambda}
\cap \supp g \neq \emptyset \},\quad j > 0,
\]
\ie $\tilde{\Lambda}_j$ is the set of indices in $\Lambda_j$ associated with shearlets whose support intersects the
support of $g$. Then, for each scale $J>0$, we have
\begin{equation}
N_{J} = \Big|\bigcup_{j=0}^{J-1} \tilde{\Lambda}_j \Big| \sim
 \sum_{j=0}^{J-1}(2^{j(\scp-1)/2})^2\,  2^{j\scp /2} \, 2^{j/2}\, 2^{j/2} =2^{(3/2)\scp
J},\label{eq:crude-est-N-smooth-coeff}
\end{equation}
where the term $(2^{j(\scp-1)/2})^2$ is due to the number of shearing $\abs{k} = \abs{(k_1,k_2)} \in
2^{j(\scp-1)/2} $ at scale $j$ and the term $ 2^{j \scp /2} \,2^{j/2}\, 2^{j/2}$ is due to the
number of translation for which $g$ and $\psi_\lambda$ interact; recall that $\psi_\lambda$ has
support in a set of measure $2^{-j\scp /2} \cdot 2^{-j/2} \cdot 2^{-j/2}$.

We observe that there exists some $C > 0$ such that
\begin{align*}
\sum_{j_0 = 1}^{\infty}2^{\scp\beta j_0} \sum_{n>N_{j_0}} |c(g)^*_n|^2
& \leq  C \cdot \sum_{j_0=1}^{\infty}\sum_{j=j_0}^{\infty}\sum_{k,m} 2^{\scp\beta j_0} |\langle
g,\psi_{j,k,m}\rangle|^2\\
& =  C\cdot\sum_{j=1}^{\infty}\sum_{k,m}|\langle g,\psi_{j,k,m}\rangle|^2\biggl( \sum_{j_0=1}^{j}2^{\scp\beta
j_0}\biggr).
\end{align*}
By Lemma \ref{lemma:smoothpart}, this yields
\[
\sum_{j_0 = 1}^{\infty}2^{\scp\beta j_0} \sum_{n>N_{j_0}} |c(g)^*_n|^2
\leq C \cdot \sum_{j=1}^{\infty}\sum_{k,m}2^{\scp\beta j} |\langle g,\psi_{j,k,m}\rangle|^2 < \infty,
\]
and thus, by \eqref{eq:crude-est-N-smooth-coeff}, that
\[
\sum_{n>N_{j_0}} |c(g)^*_n|^2 \leq C \cdot 2^{-\scp \beta j_0}
  = C \cdot (2^{(3/2)\scp j_0})^{-2\beta /3} \leq C \cdot N_{j_0}^{\,-2\beta /3}.
\]
Finally, let $N>0$. Then there exists a positive integer $j_0>0$ such that 
\begin{equation*} %\label{eq:NJ}
N \sim N_{j_0} \sim 2^{(3/2)\scp j_0},
\end{equation*}
which completes the proof.
\end{proof}

We can get rid of the Sobolev space requirement in
Proposition~\ref{prop:main-smooth} if we accept a slightly worse decay
rate.
\begin{proposition}
  \label{cor:main-smooth}
  Let $f \in C^\beta(\R^3)$ with $\supp g \subset
  \itvcc{0}{1}^3$. Suppose that $\psi \in L^2(\R^3)$ is compactly
  supported and $(\delta,\gamma)$-feasible for $\vap >2\dep + \beta$ and $\dep >3$.
  Then 
  \[
  \sum_{n>N} \abs{c(g)^*_n}^2 \lesssim  N^{-2\beta/3+\eps} \qquad
  \text{as} \quad N \to \infty,  \]
  for any $\eps>0$.
\end{proposition}
\begin{proof}
  By the intrinsic characterization of fractional order Sobolev spaces \cite{MR0450957}, we see that
  $C^\beta_0(\R^3) \subset H^{\beta-\eps}_0(\R^3)$ for any $\eps>0$. The result now follows from
  Proposition~\ref{prop:main-smooth}.
\end{proof}

\section[Analysis of coefficients on the discontinuity surface]{Analysis
  of shearlet coefficients associated with the discontinuity surface}
\label{sec:discontinuity}

We now turn our attention to \emph{Case~2}. Here we have to estimate those shearlet coefficients whose support
intersects the discontinuity surface. For any scale $j \ge 0$ and any grid
point $p \in \Z^3$, we let $\cQ_{j,p}$ denote the dyadic cube defined
by
\[
\cQ_{j,p} = [-2^{-j/2},2^{-j/2}]^3+2^{-j/2}2p.
\]
We let $\cQ_j$ be the collection of those dyadic cubes $\cQ_{j,p}$ at scale $j$ whose interior $\intt(\cQ_{j,p})$
intersects $\partial B$, \ie
\[
\cQ_j = \{\cQ_{j,p}: \intt(\cQ_{j,p}) \cap \partial B \ne \emptyset, p \in \Z^3\}.
\]
Of interest to us are not only the dyadic cubes, but also the shearlet indices associated with shearlets
intersecting the discontinuity surface inside some $\cQ_{j,p} \in \cQ_{j}$, \ie for $j\ge0$ and $p \in \Z^3$ with
$\cQ_{j,p} \in \cQ_{j}$, we will consider the index set
\[ {\Lambda_{j,p}} = \{\lambda \in \Lambda_j : \intt(\supp
\psi_{\lambda}) \cap \intt(\cQ_{j,p}) \cap \partial B \ne \emptyset\}.
\]
Further, for $j \ge 0$, $p \in \Z^3$, and $0 < \eps < 1$, we define $\Lambda_{j,p}(\eps)$ to be the index set of
shearlets $\psi_\lambda$, $\lambda \in {\Lambda_{j,p}}$, such that the magnitude of the corresponding shearlet
coefficient $\langle f,\psi_\lambda \rangle$ is larger than $\eps$ and the support of $\psi_\lambda$ intersects
$\cQ_{j,p}$ at the $j$th scale, \ie
\[ 
{\Lambda_{j,p}(\eps)} = \{\lambda \in {\Lambda_{j,p}} : |\langle
f,\psi_{\lambda}\rangle| > \eps\}.
\]
The collection of such shearlet indices across scales and translates will be denoted by {$\Lambda(\eps)$}, \ie
\[
\Lambda(\eps) = \bigcup_{j,p} \Lambda_{j,p}(\eps).
\]

As mentioned in Section~\ref{sec:general-orga-proofs}, we may assume that $j$ is sufficiently large. Suppose $\cQ_{j,p}
\in \cQ_j$ for some given scale $j \ge 0$ and position $p \in \Z^3$. Then the set
\[
\cS_{j,p}= \bigcup_{\lambda \in \Lambda_{j,p}} \supp \psi_\lambda
\]
is contained in a cube of size $C\cdot 2^{-j/2}$ by $C\cdot 2^{-j/2}$
by $C\cdot 2^{-j/2}$ and is, thereby, asymptotically of the same size as
$\cQ_{j,p}$.

We now restrict ourselves to considering $B \in \mathit{STAR}^\alpha(\cp)$; the piecewise case $B \in \mathit{STAR}^\alpha(\cp,L)$ will be
dealt with in Section~\ref{sec:proof-theorem-2}. By smoothness assumption on the discontinuity surface $\partial B$, the
discontinuity surface can locally be parametrized by either $(x_1,x_2,E(x_1,x_2))$, $(x_1,E(x_1,x_3),x_3)$, or
$(E(x_2,x_3),x_2,x_3)$ with $E \in C^\alpha$ in the interior of $\cS_{j,p}$ for sufficiently large $j$. In other words,
the part of the discontinuity surface $\partial B$ contained in $\cS_{j,p}$ can be described as the graph $x_3 =
E(x_1,x_2)$, $x_2=E(x_1,x_3)$, or $x_1=E(x_2,x_3)$ of a $C^\alpha$ function.

Thus, we are facing the following two
cases: 
\begin{description}
\item[Case 2a.] The discontinuity surface $\partial B$ can be parametrized by $(E(x_2,x_3),x_2,x_3)$ with $E \in C^\alpha$
  in the interior of $\cS_{j,p}$ such that, for any $\lambda \in \Lambda_{j,p}$, we have
     \begin{align*}
    |\pa^{(1,0)}E(\hat{x}_2,\hat{x}_3)| < + \infty \quad \text{and}
    \quad |\pa^{(0,1)}E(\hat{x}_2,\hat{x}_3)|
    < +\infty ,
  \end{align*}
for all $\hat{x} =
  (\hat{x}_1,\hat{x}_2, \hat{x}_3) \in \intt(\cQ_{j,p}) \cap
  \intt(\supp \psi_{\lambda} ) \cap \partial B$.
\item[Case 2b.] The discontinuity
 surface $\partial B$ can be
  parametrized by %either
 $(x_1,x_2,E(x_1,x_2))$ or
  $(x_1,E(x_1,x_3),x_3)$  with $E \in
  C^\alpha$ in the interior of $\cS_{j,p}$ such that, for any $\lambda
  \in \Lambda_{j,p}$, there exists some $\hat{x} =
  (\hat{x}_1,\hat{x}_2,\hat{x}_3) \in \intt(\cQ_{j,p}) \cap
  \intt(\supp \psi_{\lambda} ) \cap \partial B$ satisfying
     \begin{align*}
    \pa^{(1,0)}E(\hat{x}_1,\hat{x}_2) =0 \quad \text{or} \quad 
    \pa^{(1,0)}E(\hat{x}_1,\hat{x}_3) =0.
  \end{align*}
\end{description}

\subsection{Hyperplane discontinuity}
\label{sec:hyperpl-disc}

As described in Section~\ref{sec:general-orga-proofs}, the linearized estimates of the shearlet
coefficients will be one of the key estimates in proving Theorem~\ref{thm:3d-opt-sparse}. Linearized
estimates are used in the slightly simplified situation, where the discontinuity surface is linear.
Since such an estimate is interesting in it own right, we state and prove a linearized estimation
result below. Moreover, we will use the methods developed in the proof %of the following result
repeatedly in the remaining sections of the paper. In the proof, we will see that the shearing
operation is indeed very effective when analyzing hyperplane singularities.

\begin{theorem}\label{thm:decay-hyperplane}
  Let $\psi \in L^2(\R^3)$ be compactly supported, and assume that $\psi$ satisfies conditions~(i) and (ii) of
  Theorem~\ref{thm:3d-opt-sparse}. Further, let $\lambda \in \Lambda_{j,p}$ for $j \ge 0$ and $p \in \Z^3$. Suppose that
  $f \in \cE_\alpha^\beta(\R^3)$ for $1<\alpha\le \beta \le 2$ and that %and $\nu,\mu >0$.
  $\pa B$ is linear on the support of $\psi_\lambda$ in the sense that
\[ \supp \psi_\lambda \cap \pa B \subset H \]
for some affine hyperplane $ H$ of $\R^3$. Then,
\begin{romannum}
\item if $H$ has normal vector  $(-1,s_1,s_2)$ with $s_1 \le 3$
  and $s_2 \le 3$, 
\begin{equation}\label{eq:hp-estimate1}
|\langle f,\psi_{\lambda}\rangle| \leq C \cdot \min_{i=1,2} \left\{
\frac{2^{-j(\alpha/4+1/2)}}{|k_i+2^{j(\alpha-1)/2}s_i|^{3}}\right\}, 
\end{equation}
for some constant $C>0$.
\item if $H$ has normal vector  $(-1,s_1,s_2)$ with $s_1 \ge 3/2$
  or $s_2 \ge 3/2$, 
 \begin{equation}\label{eq:hp-estimate2}
|\langle f,\psi_{\lambda}\rangle| \leq C \cdot 2^{-j(\alpha/4+1/2+\alpha\beta/2)}
 \end{equation}
for some constant $C>0$.
\item if $H$ has normal vector  $(0,s_1,s_2)$ with $s_1,s_2 \in \R$,
then (\ref{eq:hp-estimate2}) holds.
\end{romannum}
\end{theorem}

\begin{proof}
  Let us fix $(j,k,m) \in \Lambda_{j,p}$ and $f \in \cE_\alpha^\beta(\R^3)$. We can without loss of
  generality assume that $f$ is only nonzero on $B$. We first consider the cases~(i) and (ii). The
  hyperplane can be written as
  \[ H = \setprop{x \in \R^3}{\innerprod{x-x_0}{(-1,s_1,s_2)}=0}\] for some $x_0 \in \R^3$. We shear the hyperplane by
  $S_{-s}$ for $s=(s_1,s_2)$ and obtain
\begin{align*}
  S_{-s}H &= \setprop{x \in \R^3}{\innerprod{S_s
      x-x_0}{(-1,s_1,s_2)}=0}  \\
&= \setprop{x \in \R^3}{\innerprod{
    x-S_{-s}x_0}{(S_s)^T(-1,s_1,s_2)}=0}  \\
&= \setprop{x \in \R^3}{\innerprod{ x-S_{-s}x_0}{(-1,0,0)}=0} \\
& =\setprop{x=(x_1,x_2,x_3) \in \R^3}{x_1=\hat x_1}, \quad \text{where $\hat x=S_{-s}x_0$,}  
\end{align*}
which is a hyperplane parallel to the $x_2 x_3$ plane. Here the power of shearlets comes into play
since it will allow us to only consider hyperplane singularities parallel to the $x_2 x_3$ plane.
Of course, this requires that we also modify the shear parameter of the shearlet, that is, we will
consider the right hand side of
\[{\innerprod{f}{\psi_{j,k,m}}}=
{\innerprods{f(S_s\cdot)}{\psi_{j,\hat{k},m}}} 
\] 
with the new shear parameter $\hat k$ defined by $\hat{k}_1= k_1 + 2^{j(\alpha-1)/2}s_1$ and
$\hat{k}_2= k_2 + 2^{j(\alpha-1)/2}s_2$. The integrand in
$\innerprods{f(S_s\cdot)}{\psi_{j,\hat{k},m}}$ has the singularity plane exactly located on $x_1 =
\hat x_1$, \ie on $S_{-s}H$.

To simplify the expression for the integration bounds, we will fix a new origin on $S_{-s}H$, that
is, on $x_1=\hat x_1$; the $x_2$ and $x_3$ coordinate of the new origin will be fixed in the next
paragraph. Since $f$ is assumed to be only nonzero on $B$, the function $f$ will be equal to zero on
one side of $S_{-s}H$, say, $x_1 < \hat x_1$. It therefore suffices to estimate
\begin{align*} %\label{eq:hp-integration-claim2}
    {\innerprods{f_0(S_s\cdot)\chi_{\Omega}}{\psi_{j,\hat{k},m}}}
\end{align*}
for $f_0 \in C^\beta(\R^3)$ and $\Omega = \R_+ \times \R^2$. We first consider the case
$\abssmall{\hat k_1}\le \abssmall{\hat k_2}$. We further assume that $\hat k_1 <0$ and $\hat k_2
<0$. The other cases can be handled similarly.

Since $\psi$ is compactly supported, there exists some $L > 0$ such that $\supp{\psi} \subset
\itvcc{-L}{L}^3$. By a rescaling argument, we can assume $L=1$. Let
  \begin{align} \label{eq:cPjk}
    \cP_{j,k} := \setprop{ x \in \R^3 }{ \abssmall{2^{j\alpha/2}x_1 + 2^{j/2}\hat k_1 x_2 +
        2^{j/2}\hat k_2 x_3} \le 1,\abs{x_2},\abs{x_3} \leq 2^{-j/2} },
  \end{align}
  With this notation, we have $\supp \psi_{j,k,0} \subset \cP_{j,k}$. We say that the shearlet normal
  direction of the shearlet box $\cP_{j,0}$ is $(1,0,0)$, thus the shearlet normal of a sheared
  element $\psi_{j,k,m}$ associated with $\cP_{j,k}$ is
  $(1,k_1/2^{j(\alpha-1)/2},k_2/2^{j(\alpha-1)/2})$. Now, we fix our origin so that, relative to
  this new origin, it holds that
  \begin{align*}
    \supp(\psi_{j,\hat k,m})  \subset \cP_{j,\hat k} +
    (2^{-j\alpha/2},0,0) =: \tilde{\cP}_{j,\hat k}.
    %\label{eq:hp-supp}
  \end{align*}
Then one face of $\tilde{\cP}_{j,\hat k}$ intersects the origin. 

For a fixed $\abs{\hat x_3} \le 2^{-j/2}$, we consider the cross section of the parallelepiped
$\tilde{\cP}_{j,\hat k}$ on the hyperplane $x_3 =\hat x_3$. This cross section will be a parallelogram
with sides $x_2 = \pm 2^{-j/2}$,
  \begin{align*}
    2^{j\alpha/2}x_1 + 2^{j/2}\hat k_1 x_2 + 2^{j/2}\hat k_2 x_3 = 0, \;
    \text{ and } \;
    2^{j\alpha/2}x_1 + 2^{j/2}\hat k_1 x_2 + 2^{j/2}\hat k_2 x_3 = 2.
  \end{align*}
  As it is only a matter of scaling we replace the right hand side of the last equation with $1$ for
  simplicity. Solving the two last equalities for $x_2$ gives the following lines on the hyperplane
  $x_3=\hat x_3$:
   \begin{align*}
    L_1: \; x_2=- \frac{2^{j(\alpha-1)/2}}{\hat k_1} x_1 -
    \frac{\hat k_2}{\hat k_1} x_3,
    \;\; \text{ and } \;\; 
    L_2: \; x_2=- \frac{2^{j(\alpha-1)/2}}{\hat k_1} x_1 -
    \frac{\hat k_2}{\hat k_1} x_3 + \frac{2^{-j/2}}{\hat k_1}.
  \end{align*}
  We therefore have
  \begin{align} \label{eq:hp-claim2-int-domain}
    \abs{\innerprod{f_0(S_s\cdot)\chi_{\Omega}}{\psi_{j,\hat{k},m}}}
    \lesssim \abs{\int_{-2^{-j/2}}^{2^{-j/2}} \int_{0}^{K_1}
      \int_{L_2}^{L_1} f_0(S_s x) \psi_{j,\hat k, m}(x) \, \mathrm{d}x_2 \mathrm{d}x_1
      \mathrm{d}x_3 },
  \end{align}
  where the upper integration bound for $x_1$ is $K_1 = 2^{-j(\alpha/2)}- 2^{-j\alpha/2}\hat k_1 -
  2^{j(\alpha-1)/2}\hat k_2 x_3$ which follows from solving $L_2$ for $x_1$ and using that
  $\abs{x_2} \le 2^{-j/2}$. We remark that the inner integration over $x_2$ is along lines parallel
  to the singularity plane $\pa \Omega = \{0\} \times \R^2$; as mentioned, this allows us to better
  handle the singularity and will be used several times throughout this paper.

  For a fixed $\abs{x_3}\le 2^{-j/2}$, we consider the one-dimensional Taylor expansion for $f_0(S_s
  \cdot)$ at each point $x = (x_1,x_2,x_3) \in L_2$ in the $x_2$-direction:
  \begin{align*}
    f_0(S_s x) &=
    a(x_1,x_3)+b(x_1,x_3)\left(x_2+\frac{2^{j(\alpha-1)/2}}{\hat{k}_1}
      \ x_1 + \frac{\hat k_2}{\hat k_1}x_3 - \frac{2^{-j/2}}{\hat
        k_1}\right)\\ &\phantom{=}+c(x_1,x_2,x_3)
    \left(x_2+\frac{2^{j(\alpha-1)/2}}{\hat{k}_1} \ x_1 + \frac{\hat
        k_2}{\hat k_1}x_3 - \frac{2^{-j/2}}{\hat k_1}\right)^\beta,
  \end{align*}
  where $a(x_1,x_3),b(x_1,x_3)$ and $c(x_1,x_2,x_3)$ are all bounded in absolute value by $C(1+\abs{s_1})^\beta$. Using
  this Taylor expansion in (\ref{eq:hp-claim2-int-domain}) yields
  \begin{align} \label{eq:hp-split-into-I_l}
    \abs{\innerprod{f_0(S_s\cdot)\chi_{\Omega}}{\psi_{j,\hat{k},m}}}
    \lesssim (1+\abs{s_1})^\beta \abs{\int_{-2^{-j/2}}^{2^{-j/2}}
      \int_{0}^{K_1 } \sum_{l=1}^{3}I_{l}(x_1,x_3) \, \mathrm{d}x_1 \mathrm{d}x_3 },
  \end{align}
  where
  \begin{align*}
    I_1(x_1,x_3) &= \abs{\int_{L_1}^{L_2} \psi_{j,\hat k,m}(x) \mathrm{d}x_2}, \\
    I_2(x_1,x_3) &= \abs{\int_{L_1}^{L_2} (x_2+K_2) \, \psi_{j,\hat k,m}(x)
      \mathrm{d}x_2}, \\
    I_3(x_1,x_3) &= \abs{\int_{0}^{-2^{-j/2}/\hat k_1} (x_2)^\beta\, \psi_{j,\hat k,m}(x_1,x_2-K_2,x_3) \mathrm{d}x_2},
  \end{align*}
  and \[K_2 = \frac{2^{j(\alpha-1)/2}}{\hat{k}_1} \ x_1 + \frac{\hat k_2}{\hat k_1}x_3 - \frac{2^{-j/2}}{\hat k_1}.\]
  
  We next estimate each of the integrals $I_1$, $I_2$, and $I_3$ separately. We start with
  estimating $I_1(x_1,x_3)$. The Fourier Slice Theorem~(\ref{eq:fourier-slice-thm}) yields directly
  that
  \[
  I_1(x_1,x_3) = \absBig{\int_{\R} \psi_{j,\hat k,m}(x) \mathrm{d}x_2} =
  \absBig{\int_{\R^2} \hat \psi_{j,\hat k,m}(\xi_1,0,\xi_3)\, \expo{2\pi
      i \innerprod{(x_1,x_3)}{(\xi_1,\xi_3)}} \mathrm{d}\xi_1\mathrm{d}\xi_3}.
  \]
  By assumptions (i) and (ii) from Theorem~\ref{thm:3d-opt-sparse}, we
  have, for all $\xi=(\xi_1,\xi_2,\xi_3)\in \R^3$,
  \begin{align*}
    \absbig{\hat \psi_{j,\hat k,m}(\xi)} \lesssim 2^{-j\frac{2+\alpha}{4}}
    \absbig{h(2^{-j\alpha/2}\xi_1)}
    \left(1+\absBig{\tfrac{2^{-j/2}\xi_2}{2^{-j\alpha/2}\xi_1} + \hat
        k_1}\right)^{-\gamma}
    \left(1+\absBig{\tfrac{2^{-j/2}\xi_3}{2^{-j\alpha/2}\xi_1} + \hat
        k_2}\right)^{-\gamma}
  \end{align*}
  for some $h \in L^1(\R)$. Hence, we can continue our estimate of
  $I_1$:
  \begin{align*}
    I_1(x_1,x_3) \lesssim \int_{\R^2} 2^{-j\frac{2+\alpha}{4}}
    \absbig{h(2^{-j\alpha/2}\xi_1)} (1 + \abssmall{\hat k_1})^{-\gamma}
    \left(1+\absBig{\tfrac{2^{-j/2}\xi_3}{2^{-j\alpha/2}\xi_1} + \hat
        k_2}\right)^{-\gamma} \mathrm{d}\xi_1\mathrm{d}\xi_3 ,
  \end{align*}
  and further, by a change of variables,
  \begin{align*}
    I_1(x_1,x_3) &\lesssim \int_{\R^2} 2^{j\alpha/4} \abs{h(\xi_1)} (1
    + \abssmall{\hat k_1})^{-\gamma} \left(1+\abs{\frac{\xi_3}{\xi_1}
        + \hat
        k_2}\right)^{-\gamma}  \mathrm{d}\xi_1\mathrm{d}\xi_3  \\
    &\lesssim 2^{j\alpha/4} (1 + \abssmall{\hat k_1})^{-\gamma},
  \end{align*}
since $h\in L^1(\R)$ and
$(1+\abssmall{\xi_3/\xi_1 + \hat k_2})^{-\gamma} =
     O(1)$  as $\abs{\xi_1}\to \infty$ for fixed $\xi_3$.

We estimate $I_2(x_1,x_3)$ by
    \begin{align*}
  I_2(x_1,x_3) \le \abs{\int_{\R} x_2\, \psi_{j,\hat k,m}(x) \mathrm{d}x_2} + 
\abs{K_2} \abs{\int_{\R} \psi_{j,\hat k,m}(x) \mathrm{d}x_2} =: S_1 + S_2 
\end{align*}
Applying the Fourier Slice Theorem again and then utilizing the decay assumptions on $\hat \psi$
yields
\begin{align*}
  S_1 &= \abs{\int_{\R} x_2 \psi_{j,\hat k,m}(x) \mathrm{d}x_2} \\
&\le   \abs{\int_{\R^2} \left(\frac{\pa}{\pa \xi_2} \hat \psi_{j,\hat
      k,m}\right)(\xi_1,0,\xi_3)\, \expo{2\pi
      i \innerprod{(x_1,x_3)}{(\xi_1,\xi_3)}} \mathrm{d}\xi_1\mathrm{d}\xi_3} 
\lesssim \frac{2^{j(\alpha/4-1/2)}}{(1+ \abssmall{\hat k_1})^{\beta+1}}.
\end{align*}
Since $\abs{x_1} \le -\hat{k}_1/2^j$ and $\abs{\xi_3} \le 2^{-j/2}$, we
have that 
 \[K_2 \le \abs{\frac{2^{j(\alpha-1)/2}}{\hat{k}_1} \
     \frac{\hat k_1}{2^j} + 2^{-j/2} - \frac{2^{-j/2}}{\hat
       k_1}},\]
The following estimate of $S_2$ then follows directly from the
estimate of $I_1$:
\begin{align*}
  S_2 \lesssim \abs{K_2} 2^{j\alpha/4}\, (1 + \abssmall{\hat
    k_1})^{-\gamma}
\lesssim 2^{j(\alpha/4-1/2)}  (1 + \abssmall{\hat
    k_1})^{-\gamma}.
%\frac{2^{j(\alpha/4-1/2)}}{  (1 + \abssmall{\hat k_1})^{-\gamma}}.
\end{align*}
From the two last estimate, we conclude that $ I_2(x_1,x_3s) \lesssim \frac{2^{j(\alpha/4-1/2)}}{(1+ \abssmall{\hat k_1})^{\beta+1}}$.

Finally, we estimate $I_3(x_1,x_3)$ by 
\begin{align*}
  I_3(x_1,x_3) &\le \abs{\int_{0}^{-2^{-j/2}/\hat k_1} (x_2)^\beta\,
    \normsmall[L^\infty]{\psi_{j,\hat k,m}} \, \mathrm{d}x_2} \\ 
&\lesssim 2^{j(\alpha/4+1/2)} \abs{\int_{0}^{-2^{-j/2}/\hat k_1} (x_2)^\beta\,
     \mathrm{d}x_2} = \frac{2^{j(\alpha/4-\beta/2)}}{\abssmall{\hat k_1}^{\beta+1}}.
\end{align*}

Having estimated $I_1$, $I_2$ and $I_3$, we continue with
(\ref{eq:hp-split-into-I_l}) and obtain 
 \[\abs{\innerprod{f_0(S_s\cdot)\chi_{\Omega}}{\psi_{j,\hat{k},m}}}
  \lesssim  (1+\abs{s_1})^\beta \left(
    \frac{2^{-j(\alpha/4+1/2)}}{(1+\abssmall{\hat k_1})^{\gamma-1}} + \frac{2^{-j(\alpha/4+1/2+\beta/2)}}{\abssmall{\hat
k_1}^{\beta}}\right). 
\]
By performing a similar analysis for the case $\abssmall{\hat k_2} \le
\abssmall{\hat k_1} $, we arrive at % therefore have 
\begin{align} \label{eq:hp-claim2-est-with-shear}
  \absbig{\innerprod{f_0(S_s\cdot)\chi_{\Omega}}{\psi_{j,\hat{k},m}}}
  \lesssim  \min_{i=1,2}\left\{(1+\abs{s_i})^\beta \left(
    \frac{2^{-j(\alpha/4+1/2)}}{(1+\abssmall{\hat k_i})^{\gamma-1}} +
    \frac{2^{-j(\alpha/4+1/2+\beta/2)}}{\abssmall{\hat
        k_i}^{\beta}}\right)  \right\}
\end{align}

 Suppose that $s_1 \le 3$ and $s_2 \le 3$. Then
 (\ref{eq:hp-claim2-est-with-shear}) reduces to 
\begin{align*}
  \abs{\innerprod{f}{\psi_{j,k,m}}}
  &\lesssim \min_{i=1,2}\left\{ 
      \frac{2^{-j(\alpha/4+1/2)}}{(1+\abssmall{\hat k_i})^{\gamma-1}} +
      \frac{2^{-j(\alpha/4+\beta/2+1/2)}}{\abssmall{\hat
          k_i}^{\beta}} \right\}  \\ %\label{eq:hp-est-bounded-shear}  \\
&\lesssim
\min_{i=1,2}\left\{
  \frac{2^{-j(\alpha/4+1/2)}}{|k_i+2^{j(\alpha-1)/2}s_i|^{3}}\right\}, %\nonumber
\end{align*}
since $\gamma \ge 4$ and $\beta\ge \alpha$.
On the other hand, if $s_1 \ge 3/2$ or $s_1 \ge 3/2$, then
\begin{align*} %\label{eq:hp-est-unbounded-shear}
  \absbig{\innerprod{f}{\psi_{j,k,m}}}
  \lesssim 2^{-j (\alpha/2+1/4)\alpha}.
\end{align*}
To see this, note that 
\begin{multline*}
  \min_{i=1,2}\left\{(1+\abs{s_i})^{\beta}
    \frac{2^{-j(\alpha/4+\beta/2+1/2)}}{(1+\abssmall{\hat k_i})^{\beta}}
  \right\} =   \min_{i=1,2}\left\{ \tfrac{(1+\abs{s_i})^{\beta}}{\abs{s_i}^{\beta}}
    \frac{2^{-j(\alpha/4+\beta/2+1/2)}}{(\abssmall{(1+k_i)/s_i +
        2^{j(\alpha-1)/2}})^{\beta}} 
  \right\}\\ \lesssim
  \frac{2^{-j(\alpha/4+\beta/2+1/2)}}{2^{j(\alpha-1)\beta/2}} = 2^{-j(\alpha/4+1/2+\alpha\beta/2)}. 
\end{multline*}
This completes the proof of the estimates (\ref{eq:hp-estimate1}) and
(\ref{eq:hp-estimate2}) in (i) and (ii), respectively.

\smallskip

Finally, we need to consider the case~(iii) in which the normal vector of the hyperplane $H$ is of the
form $(0,s_1,s_2)$ for $s_1,s_2 \in \R$. Let $\tilde{\Omega}=\setprop{x \in \R^3}{s_1x_2 \ge
  -s_2x_3}$. As in the first part of the proof, it suffices to consider
$\innerprod{\chi_{\tilde{\Omega}}f_0}{\psi_{j,k,m}}$, where $\supp \psi_{j,k,m} \subset
\cP_{j,k}-(2^{-j\alpha/2},0,0) = \tilde\cP_{j,k}$ with respect to some new origin. As before the
boundary of $\tilde\cP_{j,k}$ intersects the origin. By assumptions~(i) and (ii) from
Theorem~\ref{thm:3d-opt-sparse}, we have that 
\[ \left( \frac{\pa}{\pa \xi_1}\right)^\ell \hat \psi(0,\xi_2,\xi_3)=0 \quad \text{for }
\ell =0,1,\]
which implies that
\[ \int_{\R} x_1^\ell \psi(x) \mathrm{d}x_1 = 0 \quad \text{for all } x_2,x_3
\in \R \text{ and } \ell =0,1.\]
Therefore, we have
\begin{align}\label{eq:hp-shear-preserves-vm}
  \int_{\R} x_1^\ell \psi(S_k x) \mathrm{d}x_1 = 0 \quad \text{for all }
  x_2,x_3 \in \R, k=(k_1,k_2)\in \R^2, \text{ and } \ell =0,1,
\end{align}
since shearing operations $S_k$ preserve vanishing moments along the $x_1$ axis. Since the $x_1$
axis is in a direction parallel to the singularity plane $\partial \tilde{\Omega}$, we employ Taylor
expansion of $f_0$ in this direction. By (\ref{eq:hp-shear-preserves-vm}) everything but the last
term in the Taylor expansion disappears, and we obtain
\begin{align*}
  \abs{\innerprod{\chi_{\tilde{\Omega}}f_0}{\psi_{j,k,m}}} &\lesssim
2^{j(\alpha/4+1/2)} \int_{-2^{-j/2}}^{2^{-j/2}}
\int_{-2^{-j/2}}^{2^{-j/2}} \int_{-2^{-j\alpha/2}}^{2^{-j\alpha/2}} (x_1)^\beta \,\mathrm{d}x_1
\mathrm{d}x_2 \mathrm{d}x_3 \\
& \lesssim 2^{j(\alpha/4+1/2)}\, 2^{-j}\, 2^{-j(\beta+1)\alpha/2} = 2^{-j(\alpha/4+1/2+\alpha\beta/2)},
\end{align*}
which proves claim (iii). % (\ref{eq:hp-estimate-new}).
% is faster than (\ref{eq:hp-total-est-unbounded-shear}) and
% therefore not contribute to the sought estimates. 
\end{proof}

% Notice that in Case 2a, condition (i) or (ii) can occur, whereas in Case 2b, all three conditions can occur.

\subsection{General $C^\alpha$-smooth discontinuity}
\label{sec:general-case}

We now extend the result from the previous section, Theorem~\ref{thm:decay-hyperplane}, from a
linear discontinuity surface to a general, non-linear $C^\alpha$-smooth discontinuity surface. To
achieve this, we will mainly focus on the truncation arguments since the linearized estimates can be
handled by the machinery developed in the previous subsection.

\begin{theorem}\label{thm:decay-ca-smooth}
  Let $\psi \in L^2(\R^3)$ be compactly supported, and assume that
  $\psi$ satisfies conditions (i) and (ii) of
  Theorem~\ref{thm:3d-opt-sparse}. Further, let $j \ge 0$ and $p \in
  \Z^3$, and let $\lambda \in \Lambda_{j,p}$. Suppose $f \in
  \cE_\alpha^\beta(\R^3)$ for $1<\alpha\le \beta \le 2$ and $\nu, \mu
  >0$. For fixed $\hat{x}=(\hat{x}_1,\hat{x}_2,\hat{x}_3) \in
  \intt(\cQ_{j,p}) \cap \intt(\supp \psi_{\lambda} ) \cap \partial B$,
  let $H$ be the tangent plane to the discontinuity surface
  $\partial B$ at $(\hat{x}_1,\hat{x}_2, \hat{x}_3)$. Then, 
\begin{romannum}
\item if $H$ has normal vector  $(-1,s_1,s_2)$ with $s_1 \le 3$
  and $s_2 \le 3$, 
\begin{equation}\label{eq:estimate1}
|\langle f,\psi_{\lambda}\rangle| \leq C \cdot \min_{i=1,2} \left\{
\frac{2^{-j(\alpha/4+1/2)}}{|k_i+2^{j(\alpha-1)/2}s_i|^{\alpha+1}}\right\}, 
\end{equation}
for some constant $C>0$.
\item if $H$ has normal vector  $(-1,s_1,s_2)$ with $s_1 \ge 3/2$
  or $s_2 \ge 3/2$, 
 \begin{equation}\label{eq:estimate2}
|\langle f,\psi_{\lambda}\rangle| \leq C \cdot 2^{-j(\alpha/2+1/4)\alpha},
\end{equation}
for some constant $C>0$.
\item if $H$ has normal vector  $(0,s_1,s_2)$ with $s_1,s_2 \in \R$,
 then \eqref{eq:estimate2} holds.
\end{romannum}
\end{theorem}

\begin{proof}
  Let $(j,k,m) \in \Lambda_{j,p}$, and fix $\hat{x}=(\hat{x}_1,\hat{x}_2,\hat{x}_3) \in
  \intt(\cQ_{j,p}) \cap \intt(\supp \psi_{\lambda}) \cap \partial B$. We first consider the case~(i)
  and (ii). Let $(-1,s_1,s_2)$ be the normal vector to the discontinuity surface $\partial B$ at
  $(\hat{x}_1,\hat{x}_2, \hat{x}_3)$. Let $\pa B$ be parametrized by $(E(x_2,x_3),x_2,x_3)$ with $E
  \in C^\alpha$ in the interior of $\cS_{j,p}$. We then have $s_1 =
  \pa^{(1,0)}E(\hat{x}_2,\hat{x}_3)$ and $s_2 = \pa^{(0,1)}E(\hat{x}_2,\hat{x}_3)$.

  By translation symmetry, we can assume that the discontinuity surface satisfies $E(0,0)=0$ with
  $(\hat{x}_1,\hat{x}_2,\hat{x}_3) = (0,0,0)$. Further, since the conditions (i) and (ii) in
  Theorem~\ref{thm:3d-opt-sparse} are independent on the translation parameter $m$, it does not play
  a role in our analysis. Hence, we simply choose $m = (0,0,0)$. Also, since $\psi$ is compactly
  supported, there exists some $L > 0$ such that $\supp{\psi} \subset \itvcc{-1}{1}^3$. By a
  rescaling argument, we can assume $L=1$. Therefore, we have that
  \begin{align*}
    \supp \psi_{j,k,0} \subset \cP_{j,k}. %\label{eq:supp}
  \end{align*}
 where $\cP_{j,k}$ was introduced in \eqref{eq:cPjk}.

 Fix $f \in \cE_\alpha^\beta(\R^3)$. We can without loss of generality assume that $f$ is only
 nonzero on $B$. We let $\cP$ be the smallest parallelepiped which contains the discontinuity
 surface parametrized by $(E(x_2,x_3),x_2,x_3)$ in the interior of $\supp \psi_{j,k,0}$. Moreover,
 we choose $\cP$ such that two sides are parallel to the tangent plane with normal vector
 $(-1,s_1,s_2)$. Using the trivial identity $f = \chi_{\cP} f+ \chi_{\cP^c}f$, we see that
  \begin{align}
    \innerprod{f}{\psi_{j,k,0}} =
    \innerprod{\chi_{\cP}f}{\psi_{j,k,0}} +
    \innerprod{\chi_{\cP^c}f}{\psi_{j,k,0}}.\label{eq:splitting}
  \end{align}
  We will estimate $\abs{\innerprod{f}{\psi_{j,k,0}}}$ by estimating the two terms on the right hand
  side of (\ref{eq:splitting}) separately. In the second term
  $\innerprod{\chi_{\cP^c}f}{\psi_{j,k,0}}$ the shearlet only interacts with a discontinuity plane,
  and not a general $C^\alpha$ surface, hence this term corresponds to a linearized estimate (see
  Section~\ref{sec:general-orga-proofs}). Accordingly, the first term is a truncation term.

  Let us start by estimating the first term $\innerprod{\chi_{\cP}f}{\psi_{j,k,0}}$ in
  (\ref{eq:splitting}). Using the notation $\hat{k}_1= k_1 + 2^{j(\alpha-1)/2}s_1$ and $\hat{k}_2=
  k_2 + 2^{j(\alpha-1)/2}s_2$, we claim that
  \begin{align}
    \abs{\innerprod{\chi_{\cP}f}{\psi_{j,k,0}}} \lesssim \min_{i=1,2}
    \left( (1+s_i^2)^{\frac{\alpha+1}{2}}
      \frac{2^{-j(\alpha/4+1/2)}}{(1+\abssmall{\hat k_i})^{\alpha+1}}
    \right).
    \label{eq:claim1}
  \end{align}
  We will prove this claim in the following paragraphs. 

  We can assume that $\hat k_1 < 0$ and $\hat k_2<0$ since the other cases can be handled similarly.
  We fix $|\hat x_3| \le 2^{-j/2}$ and perform first a 2D analysis on the plane $x_3=\hat x_3$.
  After a possible translation (depending on $\hat x_3$) we can assume that the tangent line of $\pa
  B$ on the hyperplane is of the form
  \[ x_1 = s_1(\hat x_3) x_2 + \hat x_3. \] Still on the hyperplane, the shearlet
  normal direction is $(1,k_1/2^{j/2})$. Let $d=d(\hat x_3)$ denote the
  distance between the two points, where the tangent line intersects
  the boundary of the shearlet box $\cP_{j,k}$. It follows that
  \begin{align*}
    d(\hat x_3) \lesssim (1+s_1(\hat x_3))^{1/2}
    \frac{2^{-j/2}}{\abs{1+k_1+2^{j(\alpha-1)/2}s_1(\hat x_3)}}
  \end{align*}
  as in the proof of Proposition~2.2 in \cite{KL10}. We can replace
  $s_1(\hat x_3)$ by $s_1=s_1(0)$ in the above estimate. To see this note
  that $E \in C^\alpha$ implies
  \[ s_1(\hat x_3)-s_1(0) \lesssim \abs{\hat x_3}^{\alpha-1} \le
  2^{-j(\alpha-1)/2},\] and thereby,
  \[ \frac{2^{-j/2}}{\abs{1+k_1+2^{j(\alpha-1)/2}s_1(\hat x_3)}} \lesssim
  \frac{2^{-j/2}}{\abs{1+\tilde{k}_1+2^{j(\alpha-1)/2}s_1(0)}}, \]
  where $\tilde{k}_1 = k_1 + 2^{j(\alpha-1)/2}(s_1(\hat x_3)-s_1(0))$.
  Since
  \[ \abssmall{\tilde{k}_1}-C \le \abs{k_1} \le \abssmall{\tilde{k}_1}
  + C\] for some constant $C$, there is no need to distinguish between
  $k_1$ and $\tilde{k}_1$, and we arrive at
  \begin{align}\label{eq:universal-s}
    d(\hat x_3) \lesssim (1+s_1^2)^{1/2}
    \frac{2^{-j/2}}{1+\abs{k_1+2^{j(\alpha-1)/2}s_1}} =:d
  \end{align}
  for any $\abs{\hat x_3} \le 2^{-j/2}$.

  The cross section of our parallelepiped $\cP$ on the hyperplane will be a parallelogram with side
  length $d$ and height $d^\alpha$ (up to some constants). Since $\abs{x_3} \le 2^{-j/2}$ for
  $(x_1,x_2,x_3) \in \cP_{j,k}$, the volume of $\cP$ is therefore bounded by:
  \[ \vol{\cP} \lesssim 2^{-j/2} d^{1+\alpha} =
  (1+s_1^2)^{\frac{\alpha+1}{2}}
  \frac{2^{-j(\alpha/2+1)}}{(1+\abs{k_1+2^{j(\alpha-1)/2}s_1})^{\alpha+1}}.\]
  In the same way we can obtain an estimate based on $k_2$ and $s_2$
  with $k_1$ and $s_1$ replaced by $k_2$ and $s_2$, thus
  \[ \vol{\cP} \lesssim \min_{i=1,2} \left\{
    (1+s_i^2)^{\frac{\alpha+1}{2}}
    \frac{2^{-j(\alpha/2+1)}}{(1+\abs{k_i+2^{j(\alpha-1)/2}s_i})^{\alpha+1}}\right\}.\]
  Finally, using $\abs{\innerprod{\chi_{\cP}f}{\psi_{j,k,0}}} \le
  \norm[L^\infty]{\psi_{j,k,0}} \vol{\cP} = 2^{j(\alpha/4+1/2)}
  \vol{\cP}$, we arrive at our claim (\ref{eq:claim1}).

  We turn to estimating the linearized term in (\ref{eq:splitting}).
  This case can be handled as the proof of
  Theorem~\ref{thm:decay-hyperplane}, hence we therefore have
\begin{align} \label{eq:claim2-est-with-shear}
  \absbig{\innerprod{f_0(S_s\cdot)\chi_{\Omega}}{\psi_{j,\hat{k},0}}}
  \lesssim  \min_{i=1,2}\left\{(1+\abs{s_i})^\beta \left(
    \frac{2^{-j(\alpha/4+1/2)}}{(1+\abssmall{\hat k_i})^{\gamma-1}} +
    \frac{2^{-j(\alpha/4+\beta/2+1/2)}}{\abssmall{\hat
        k_i}^{\beta}}\right)  \right\}.
\end{align}
By summarizing from estimate (\ref{eq:claim1}) and
(\ref{eq:claim2-est-with-shear}), we conclude that
\begin{multline} \label{eq:total-est-with-shear}
  \absbig{\innerprod{f}{\psi_{j,k,0}}}
  \lesssim \min_{i=1,2}\biggl\{(1+\abs{s_i})^\beta \left(
      \frac{2^{-j(\alpha/4+1/2)}}{(1+\abssmall{\hat k_i})^{\gamma-1}} +
      \frac{2^{-j(\alpha/4+\beta/2+1/2)}}{\abssmall{\hat
          k_i}^{\beta}}\right) \\ +  (1+s_i^2)^{\frac{\alpha+1}{2}}
 \frac{2^{-j(\alpha/4+1/2)}}{(1+\abssmall{\hat
    k_i})^{\alpha+1}} \biggr\}.
\end{multline}
If $s_1 \le 3$ and $s_2 \le 3$, this reduces to 
\begin{align*}
  \absbig{\innerprod{f}{\psi_{j,k,0}}}
  &\lesssim \min_{i=1,2}\left\{ 
      \frac{2^{-j(\alpha/4+1/2)}}{(1+\abssmall{\hat k_i})^{\gamma-1}} +
      \frac{2^{-j(\alpha/4+\beta/2+1/2)}}{\abssmall{\hat
          k_i}^{\beta}} +   \frac{2^{-j(\alpha/4+1/2)}}{(1+\abssmall{\hat
    k_i})^{\alpha+1}} \right\} \\
&\lesssim
\min_{i=1,2}\left\{ \frac{2^{-j(\alpha/4+1/2)}}{|k_i+2^{j(\alpha-1)/2}s_i|^{\alpha+1}}\right\},
%\label{eq:total-est-bounded-shear}
\end{align*}
since $\gamma \ge 4$ and $\beta\ge \alpha$. On the other hand, if $s_1 \ge 3/2$ or $s_1 \ge 3/2$,
then
\begin{align*} %\label{eq:total-est-unbounded-shear}
  \absbig{\innerprod{f}{\psi_{j,k,0}}}
  \lesssim 2^{-j (\alpha/2+1/4)\alpha},
\end{align*}
which is due to the last term in (\ref{eq:total-est-with-shear}). To see this, note that
\begin{multline*}
  \min_{i=1,2}\left\{(1+s_i^2)^{\frac{\alpha+1}{2}}
    \frac{2^{-j(\alpha/4+1/2)}}{(1+\abssmall{\hat k_i})^{\alpha+1}}
  \right\} =   \min_{i=1,2}\left\{ \frac{(1+s_i^2)^{\frac{\alpha+1}{2}}}{\abs{s_i}^{\alpha+1}}
    \frac{2^{-j(\alpha/4+1/2)}}{(\abssmall{k_i/s_i +
        2^{j(\alpha-1)/2}})^{\alpha+1}} 
  \right\}\\ \lesssim
  \frac{2^{-j(\alpha/4+1/2)}}{2^{j(\alpha-1)(\alpha+1)/2}} = 2^{-j(\alpha/2+1/4)\alpha} 
\end{multline*}
This completes the proof of the estimates (\ref{eq:estimate1}) and (\ref{eq:estimate2}) in (i) and
(ii), respectively.

\smallskip

Finally, we need to consider the case~(iii), where the normal vector of the tangent plane $H$ is of
the form $(0,s_1,s_2)$ for $s_1,s_2 \in \R$. The truncation term can be handled as above, and the
linearization term as the proof of Theorem~\ref{thm:decay-hyperplane}.
\end{proof}

\section{Proof of Theorem~\ref{thm:3d-opt-sparse}}
\label{sec:proof-theorem-1}

Let $f \in \cE_\alpha^\beta(\R^3)$. By
Proposition~\ref{prop:main-smooth}, for $\alpha \le \beta$, we see that
shearlet coefficients associated with Case~1 meet the desired decay
rate~(\ref{eq:3d-approx-rate}). We therefore only need to consider
shearlet coefficients from Case 2, and, in particular, their decay
rate. For this, let $j \ge 0$ be sufficiently large and let $p \in
\Z^3$ be such that the associated cube satisfies $\cQ_{j,p} \in
\cQ_j$, hence $\intt(\cQ_{j,p}) \cap  \partial B \neq \emptyset$.

Let $\eps > 0$. Our goal will now be to estimate first
$\card{\Lambda_{j,p}(\eps)}$ and then $\card{\Lambda(\eps)}$. By assumptions on
$\psi$, there exists a $C>0$ so that $\norm[L^1]{\psi} \le C$. 
This implies that
\[
|\langle f,\psi_{\lambda}\rangle| \le \norm[L^\infty]{f}
\norm[L^1]{\psi_\lambda} \le \mu\, C\, 2^{-j(\alpha+2)/4}.
\]
Assume for simplicity $\mu\, C = 1$. Hence, for estimating $\card{\Lambda_{j,p}(\eps)}$, it is
sufficient to restrict our attention to scales
\begin{equation*}
% \label{eq:upperbdonj} 
0 \le j \le
j_0 := \frac{4}{\alpha+2}\log_2(\eps^{-1}).
\end{equation*}

 {\em Case 2a}. It suffices to consider one fixed $\hat{x}=(\hat{x}_1,\hat{x}_2,\hat{x}_3) \in
  \intt(\cQ_{j,p}) \cap \intt(\supp \psi_{\lambda} )$ $ \cap \partial B$
  associated with one fixed normal $(-1,s_1,s_2)$ in each $\cQ_{j,p}$;
  the proof of this fact is similar to the estimation of the term
$\innerprod{\chi_{\cP}f}{\psi_{j,k,0}}$ in (\ref{eq:splitting}) in the
proof of Theorem~\ref{thm:decay-ca-smooth}.
% and we therefore leave it as an
%exercise for the reader to check this. 

 We claim that the following \emph{counting estimate} hold:
\begin{equation}
\label{eq:counting}
\card{M_{j,k,Q_{j,p}}}  \lesssim
\abssmall{k_1+2^{j(\alpha-1)/2}s_1}+\abssmall{k_2+2^{j(\alpha-1)/2}s_2}+1,
\end{equation}
for each $k=(k_1,k_2)$ with $\abs{k_1},\abs{k_2} \le \ceil{2^{j(\alpha-1)/2}}$, where
\[ 
M_{j,k,Q_{j,p}}:=\setprop{m \in \Z^3 }{ |\supp{\psi_{j,k,m}} \cap \partial B \cap \cQ| \neq 0}
 \]

 Let us prove this claim. Without of generality, we can assume $\cQ:=\cQ_{j,p} =
 [-2^{-j/2} 2^{-j/2}]^3$ and that $H$ is a tangent plane to $\partial B$ at $(0,0,0)$.
For fixed shear parameter $k$, let $\cP_{j, k}$ be given as in
 (\ref{eq:cPjk}).  Note that $\supp{\psi_{j,k,0}} \subset \cP_{j, k}$ and that
\begin{align*}
\cardsmall{M_{j, k,\cQ}} %&= \card{ \{ m \in \Z^3 : \supp{\psi_{j,k,m}} \cap \partial B \cap \cQ \} } \\
 \leq C \,\cdot\, \cardsmall{ \{ m_1 \in \Z : \bigl(\cP_{j,k}+(2^{-\alpha j/2}m_1,0,0)\bigr) \cap H \cap \cQ\}}
\end{align*}
Consider the cross section $\cP_0$ of $\cP_{j,\hat k}$:
\[
\cP_0 = \{ x \in \R^3 : x_1+\frac{{k}_1}{2^{j(\alpha-1)/2}}x_2+\frac{{k}_2}{2^{j(\alpha-1)/2}}x_3 = 0, |x_2|,|x_3| \leq 2^{-j/2} \}.
\]
Then we have  
\[
\cardsmall{M_{j,k,\cQ}} \leq C \, \cdot \, \card{ \{ m_1 \in \Z : \absbig{\bigl(\cP_{0}+(2^{-\alpha
      j/2}m_1,0,0)\bigr) \cap H \cap \cQ} \neq 0\}}
\]
Note that for $|x_2|,|x_3| \leq 2^{-j/2}$, 
\begin{align*}
H&: x_1-s_1x_2-s_2x_3 = 0, \qquad \text{and} \\
\cP_0+(2^{-\alpha j/2}m_1,0,0)&: x_1-2^{-\alpha j/2}m_1+\frac{{k}_1}{2^{j/2(\alpha-1)}}x_2+\frac{{k}_2}{2^{j/2(\alpha-1)}}x_3 = 0.
\end{align*}
Solving 
\[
s_1x_2+s_2x_3 = 2^{-\alpha j/2}m_1-\frac{{k}_1}{2^{j/2(\alpha-1)}}x_2-\frac{{k}_2}{2^{j/2(\alpha-1)}}x_3,
\]
we obtain 
\[
m_1 = 2^{j/2}\bigl( (k_1+2^{j/2(\alpha-1)}s_1) x_2 + (k_2+2^{j/2(\alpha-1)}s_2) x_3  \bigr).
\]
Since $|x_2|,|x_3| \leq 2^{-j/2}$, 
\[
|m_1| \leq |k_1+2^{j/2(\alpha-1)}s_1|+|k_2+2^{j/2(\alpha-1)}s_2|.
\]
This gives our desired estimate. 

Estimate \eqref{eq:estimate1} from Theorem~\ref{thm:decay-ca-smooth} reads
$\frac{2^{-j(\alpha/4+1/2)}}{|k_i+2^{j(\alpha-1)/2}s_i|^{\alpha+1}}
\gtrsim |\langle f,\psi_{\lambda}\rangle|>\eps$ which implies that
\begin{equation}\label{eq:number2}
|k_i+2^{j(\alpha-1)/2}s_i| \le C\cdot \eps^{-1/(\alpha+1)}\, 2^{-j\left(\frac{\alpha/4+1/2}{\alpha+1}\right)}
\end{equation}
for $i=1,2$.
From \eqref{eq:counting} and \eqref{eq:number2}, we then see that 
\begin{align*}
\card{\Lambda_{j,p}(\eps)}&\le C \sum_{(k_1,k_2) \in
  K_j(\eps)} \card{M_{j,k,Q_{j,p}}(\eps)}
 \\ %\label{eq:case1} \\ 
 &\le C \sum_{(k_1,k_2) \in
  K_j(\eps)}(\abssmall{k_1+2^{j(\alpha-1)/2}s_1}+\abssmall{k_2+2^{j(\alpha-1)/2}s_2}+1)
 \\ %\label{eq:case1} \\ 
&\le C \cdot \eps^{-3/(\alpha+1)}\,
2^{-j\left(\frac{3\alpha/4+3/2}{\alpha+1}\right)},
%\nonumber
\end{align*}
where $M_{j,k,Q_{j,p}}(\eps)=\setprop{m \in
  M_{j,k,Q_{j,p}}}{\abs{\innerprod{f}{\psi_{j,k,m}}} > \eps}$ and
$K_j(\eps) = \{k\in\Z^2:\abssmall{k_i+2^{j(\alpha-1)/2}s_i} \le C
\cdot \eps^{-1/(\alpha+1)}\,
2^{-j\left(\frac{\alpha/4+1/2}{\alpha+1}\right)}\}$.

\medskip

{\em Case 2b}. By similar arguments as given in Case 2a, it also suffices to consider just one fixed $\hat x \in
\intt(\cQ_{j,p}) \cap
\intt(\supp(\psi_{\lambda})) \cap \partial B$. Again, our goal is now to estimate $\card{\Lambda_{j,p}(\eps)}$.

By estimate \eqref{eq:estimate2} from
Theorem~\ref{thm:decay-ca-smooth}, $|\langle
f,\psi_{\lambda}\rangle|\ge \eps$ implies
\[
C \cdot 2^{-j(\alpha/2+1/4)\alpha} \ge \eps,
\]
hence we only need to consider scales
\begin{equation*} %\label{eq:case2-thm}
0 \le  j \le j_1 +C, \quad \text{where } j_1:=\frac{4}{(1+2\alpha)\alpha} \log_2{(\eps^{-1})}.
\end{equation*}
Since $\cQ_{j,p}$  is a cube with side lengths of size $2^{-j/2}$, we
have, counting the number of translates and shearing, the estimate 
\[
\card{\Lambda_{j,p}} \le C \cdot 2^{j3(\alpha-1)/2},
\]
for some $C$. It then obviously follows that
\begin{equation*}
%\label{eq:estimateLambda_jp}
\card{\Lambda_{j,p}(\eps)} \le C \cdot 2^{j3(\alpha-1)/2}.
\end{equation*}
Notice that this last estimate is exceptionally crude, but it will be
sufficient for the sought estimate.

\medskip

We now combine the estimates for $\card{\Lambda_{j,p}(\eps)}$ derived in
Case 2a and Case 2b. We first consider $\alpha<2$. Since
\[
\card{\cQ_j} \leq C\cdot 2^{j}.
\]
 we have, %for $\alpha \in \itvoo{1}{2}$.
\begin{align}
  \card{\Lambda(\eps)} &\lesssim
  \sum_{j=0}^{\frac{2}{3\alpha-1}j_0}
  2^j \, 2^{j3(\alpha-1)/2} +
   \sum_{j=\frac{2}{3\alpha-1}j_0 }^{j_0}
   2^j \eps^{-3/(\alpha+1)}\, 2^{j\frac{3\alpha/4+3/2}{\alpha+1}} +
 %  \nonumber \\ &\phantom{\,=\,}+
   \sum_{j=0}^{j_1} 2^j\,
   2^{j3(\alpha-1)/2} \nonumber \\ 
  &\lesssim
  \sum_{j=0}^{\frac{2}{3\alpha-1}j_0}
  2^{j(3\alpha-1)/2} +
   \eps^{-3/(\alpha+1)}\!\!\!\!\! \sum_{j=\frac{2}{3\alpha-1}j_0 }^\infty
    2^{-j\left(\frac{2-\alpha}{4(\alpha+1)}\right)} +
    \sum_{j=0}^{j_1}
   2^{j(3\alpha-1)/2} \nonumber \\
&\lesssim \eps^{\frac{4}{\alpha+2}}
   + \eps^{-3/(\alpha+1)}
   \eps^{\frac{2(2-\alpha)}{(\alpha+1)(\alpha+2)(3\alpha-1)}} +\eps^{-\frac{2(3\alpha-1)}{2(2\alpha+1)}}
\lesssim  \eps^{-\frac{9\alpha^2+17\alpha-10}{(\alpha+1)(\alpha+2)(3\alpha-1)}} 
 \label{eq:totalnumber}.
\end{align}

Having estimated $\card{\Lambda(\eps)}$, we are now ready to prove our main claim. For this, set $N
= \card{\Lambda(\eps)}$, \ie $N$ is the total number of shearlets $\psi_{\lambda}$ such that the
magnitude of the corresponding shearlet coefficient $\langle f,\psi_{\lambda} \rangle$ is larger
than $\eps$. By \eqref{eq:totalnumber}, it follows that
\[
\eps \lesssim 
N^{-\frac{(\alpha+1)(\alpha+2)(3\alpha-1)}{9\alpha^2+17\alpha-10}}  .
\]
This implies that
\[
\norm[L^2]{f-f_N}^2 \lesssim \sum_{n > N}|c(f)^*_{n}|^2 \lesssim
N^{-\frac{2(\alpha+1)(\alpha+2)(3\alpha-1)}{9\alpha^2+17\alpha-10}+1} 
= 
N^{-\frac{6\alpha^3+7\alpha^2-11\alpha+6}{9\alpha^2+17\alpha-10}}, 
\]
which, in turn, implies
\[
|c(f)_N^*| \le C \cdot N^{-\frac{(\alpha+1)(\alpha+2)(3\alpha-1)}{9\alpha^2+17\alpha-10}}.
\]
Summarising, we have proven (\ref{eq:3d-approx-rate}) and (\ref{eq:3d-coeff-decay-rate}) for $\alpha \in
\itvoo{1}{2}$. The case $\alpha=2$ follows similarly. This completes the proof of
Theorem~\ref{thm:3d-opt-sparse}.

\section{Proof of Theorem~\ref{thm:3d-opt-sparse-piecewise}}
\label{sec:proof-theorem-2}

We now allow the discontinuity surface $\partial B$ to be piecewise $C^\alpha$-smooth, that is, $B
\in \mathit{STAR}^\alpha(\nu,L)$. In this case $B$ is a bounded subset of $\itvcc{0}{1}^3$ whose
boundary $\partial B$ is a union of finitely many pieces $\pa B_1, \dots, \pa B_L$ which do not
overlap except at their boundaries. If two patches $\pa B_i$ and $\pa B_j$ overlap, we will denote
their comment boundary $\pa \Gamma_{i,j}$ or simply $\pa \Gamma$. We need to consider four new
subcases of Case 2:

\begin{description}
\item[Case 2c.] The support of $\psi_{\lambda}$ intersects two $C^\alpha$ discontinuity surfaces $\partial B_1$ and
  $\partial B_2$, but stays away from the 1D edge curve $\pa \Gamma_{1,2}$, where the two patches
  $\partial B_1$, $\partial B_2$ meet.
\item[Case 2d.] The support of $\psi_{\lambda}$ intersects two $C^\alpha$ discontinuity surfaces
  $\partial B_1$, $\partial B_2$ and the 1D edge curve $\pa \Gamma_{1,2}$, where the two patches
  $\partial B_1$, $\partial B_2$ meet.
\item[Case 2e.] The support of $\psi_{\lambda}$ intersects finitely many (more than two)
  $C^\alpha$ discontinuity surfaces $\partial B_1,\dots,\partial B_L$, but stays away from a point where
  all of the surfaces $\partial B_1,\dots,\partial B_L$ meet.
\item[Case 2f.] The support of $\psi_{\lambda}$ intersects finitely many (more than two)
  $C^\alpha$ discontinuity surfaces $\partial B_1,\dots,\partial B_L$ and a point where all of the
  surfaces $\partial B_1,\dots,\partial B_L$ meet.
\end{description}

In the following we prove that these new subcases will not destroy the optimal sparse approximation
rate by estimating $\card{\Lambda(\eps)}$ for each of the cases. Here, we assume that each patch
$\partial B_i$ is parametrized by $C^{\alpha}$ function $E_i$ so that
\[
\partial B_i = \{(x_1,x_2,x_3) \in \R^3 : x_1 = E_i(x_2,x_3) \}
\]
and $\norm[C^1]{E_i}\le C$.
The other cases are proved similarly. Also, for each case, we let $\cQ_{j,p}$ be the collection of
the dyadic boxes containing the relevant surfaces $\partial B_i$ and may assume $p = (0,0,0)$
without loss of generality. Finally, we assume $\supp{\psi} \subset [0,1]^3$ for simplicity and the
same proof with rescaling can be applied to cover the general case. We now estimate
$\card{\Lambda(\eps)}$ to show the optimal sparse approximation rate in each case. For this, we
compute the number of all relevant shearlets $\psi_{j,k,m}$ in each of the dyadic boxes $\cQ_{j,p}$
applying a counting argument as in Section~\ref{sec:proof-theorem-1} and estimate the decay rate of
the shearlets coefficients $\langle f,\psi_{j,k,m}\rangle$.

\smallskip

\begin{figure}[ht]
\centering
\includegraphics{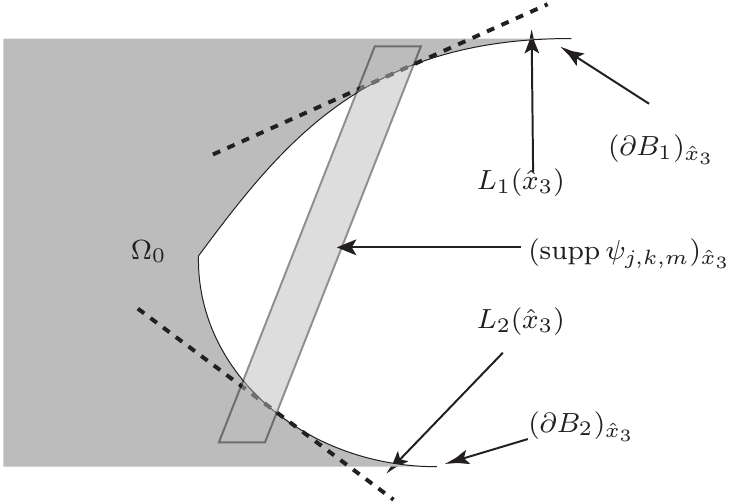}
\caption{\emph{Case 2c}. A 2D cross sections of $\supp \psi_{\lambda}$ and the two discontinuity
  surfaces $\pa B_1$ and $\pa B_2$.}
\label{fig:shear_2c}
\end{figure}
\paragraph{Case 2c} Without loss of generality, we may assume that $(\hat x_1, \hat x_2, 0)$ and
$(\hat x'_1, \hat x'_2, 0)$ belong to $\partial B_1 \cap \supp{\psi_{j,k,m}} \cap \cQ_{j,p}$ and
$\partial B_2 \cap \supp{\psi_{j,k,m}} \cap \cQ_{j,p}$ respectively for some $\hat x_1, \hat x_2,
\hat x'_1, \hat x'_2 \in \R$. Note that for a shear index $k = (k_1,k_2)$ and scale $j\ge 0$
fixed, we have by a simple counting argument that
\begin{multline}\label{eq:count1}
 \card{\bigcap_{i=1}^{2} \{m \in \Z^3 : \intt(\supp{\psi_{j,k,m}}) \cap \partial B_i \cap
   \cQ_{j,p} \neq \emptyset \} } \\
\leq C\min_{i=1,2}\set{|k_i+2^{j(\alpha-1)/2}s_i|+1}
\end{multline}
where $s_1 = \partial^{(1,0)} E_1 (\hat x_2, 0)$ and $s_2 = \partial^{(0,1)} E_2 (\hat x'_2, 0)$.
For each $\hat x_3 \in [0,2^{-j/2}]$, we define the 2D slice of $\supp{\psi_{j,k,m}}$  by
\[
(\supp{\psi_{j,k,m}})_{\hat x_3} = \{(x_1,x_2,\hat x_3) : (x_1,x_2,\hat x_3) \in \supp{\psi_{j,k,m}}\}. 
\]
We will now estimate the following 2D integral over $(\supp{\psi_{j,k,m}})_{\hat x_3}$
\begin{equation}\label{eq:int}
I_{j,k,m}(\hat x_3) = \int_{(\supp{\psi_{j,k,m}})_{\hat x_3}} f(x_1,x_2,\hat x_3)\psi_{j,k,m}(x_1,x_2,\hat x_3) \mathrm{d}x_1\mathrm{d}x_2.
\end{equation} 
This integral above gives us the worst decay rate when the 2D support
$(\supp{\psi_{j,k,m}})_{\hat x_3}$ meets both edge curves, see Figure~\ref{fig:shear_2c}. Therefore, we
may assume that for each $\hat x_3$ fixed, the set $(\supp{\psi_{j,k,m}})_{\hat x_3}$ intersects two
edge curves
\[
(\partial B_i)_{\hat x_3} = \{(x_1,x_2,\hat x_3) : (x_1,x_2,\hat x_3) \in \partial B_i \cap \cQ_{j,p}\} \quad \text{for} \,\, i = 1,2. 
\]
By a similar argument as in Section~\ref{sec:general-case}, one can
linearize the two curves $(\partial B_1)_{\hat x_3}$ and $(\partial B_2)_{\hat x_3}$ within
$(\supp{\psi_{j,k,m}})_{\hat x_3}$. In other words, we now replace the discontinuity curves
$(\partial B_1)_{\hat x_3}$ and $(\partial B_2)_{\hat x_3}$ by
\[
L_i(\hat x_3) = \{ (s_i(\hat x_3)(x_2-\hat x_2)+\hat x_1, x_2, \hat x_3) \in \cQ_{j,p}\cap (\supp{\psi_{j,k,m}})_{\hat x_3} : x_2 \in \R \} 
\]
where 
\[
s_i(\hat x_3) = \frac{\partial E_i(\hat x_2, \hat x_3)}{\partial x_2} \qquad \text{for some} \,\, (\hat x_1, \hat x_2, \hat x_3) \in (\partial B_i)_{\hat x_3} \,\, \text{and}\,\,  i = 1,2.
\]
Further, we may assume that the tangent lines $L_i(\hat x_3)$ on $(\supp{\psi_{j,k,m}})_{\hat x_3}$
do not intersect each other. In particular, one can take secant lines instead of the tangent lines
if necessary. The truncation error for the linearization with the secant line instead of
linearization with the tangent line would not change our estimates for $\card{\Lambda(\eps)}$.
Now, on each 2D support $(\supp{\psi_{j,k,m}})_{\hat x_3}$, we have a 2D piecewise smooth function
\[
f(x_1,x_2,\hat x_3) = f_0(x_1,x_2,\hat x_3) \chi_{\Omega_0} + f_1(x_1,x_2,\hat x_3) \chi_{\Omega_1}
\]
where $f_0,f_1 \in C^{\beta}$ and $\Omega_0,\Omega_1$ are disjoint subsets of $[0, 2^{-j/2}]^2$ as
in Figure~\ref{fig:shear_2c}. Observe that
\[f = f_0 \chi_{\Omega_0} + f_1 \chi_{\Omega_1} = (f_0-f_1)\chi_{\Omega_0}+f_1\] on $\cQ_{j,p} \cap
(\supp{\psi_{j,k,m}})_{\hat x_3}$. By Proposition~\ref{cor:main-smooth}, the optimal rate of sparse approximations can
be achieved for the smooth part $f_1$. Thus, it is sufficient to consider the first term $(f_0-f_1)
\chi_{\Omega_0}$ in the equation above. Therefore, we may assume that $f = g_0 \chi_{\Omega_0}$ with a
2D function $g_0 \in C^{\beta}$ on $\cQ_{j,p} \cap (\supp{\psi_{j,k,m}})_{\hat x_3}$.
Note that the discontinuities of the function $f$ lie on the two edge curves $L_i(\hat x_3)$ for $i=1,2$
on $\cQ_{j,p} \cap (\supp{\psi_{j,k,m}})_{\hat x_3}$. Applying the same linearized estimates as in Section~\ref{sec:hyperpl-disc}  for each of edge curves $L_i(\hat x_3)$, we obtain
\[
|I_{j,k,m}(\hat x_3)| \lesssim \max_{i=1,2} \set{\frac{2^{-j\alpha/4}}{(1+|k_1+2^{j(\alpha-1)/2}s_i(\hat x_3)|)^{\alpha+1}}}. 
\]      
By similar arguments as in (\ref{eq:universal-s}), we can replace $s_i(\hat x_3)$ by a universal
choice $s_i$ for $i=1,2$ independent of $\hat x_3$. 
Since $\hat x_3 \in [0,2^{-j/2}]$, this yields 
\begin{equation}\label{eq:decay1}
|\langle \psi_{j,k,m},f \rangle| \lesssim \max_{i=1,2} \biggl\{\frac{2^{-j\frac{\alpha+2}{4}}}{(1+|\hat k_i|)^{\alpha+1}}\biggr\},
\end{equation}
where $\hat k_i = k_1 + 2^{j(\alpha-1)/2}s_i$ for $i=1,2$ as usual. Also, we note that the number of
dyadic boxes $\cQ_{j,p}$ containing two distinct discontinuity surfaces is bounded above by
$2^{j/2}$ times a constant independent of scale $j$. Moreover, there are a total of $\lceil
2^{j\frac{\alpha-1}{2}}\rceil+1$ shear indices with respect to the parameter $k_2$. Let us define
\[
K_j(\eps) = \setprop{k_1 \in \Z}{\max_{i=1,2} \setbig{(1+|\hat k_i|)^{-(\alpha+1)}2^{-j\frac{\alpha+2}{4}}} > \eps }.  
\]
By \eqref{eq:count1} and \eqref{eq:decay1}, we have 
\[
\card{\Lambda(\eps)} \lesssim \sum_{j=0}^{\frac{4}{\alpha+2}\log{(\eps^{-1})}} 2^{j/2}2^{j\frac{\alpha-1}{2}}
\sum_{k_1 \in K_j(\eps)} \min_{i=1,2} \setbig{1+|\hat k_i|}. 
\]
Without loss of generality, we may assume $|\hat k_1| \leq |\hat k_2|$. Then 
\begin{multline*}
  \card{\Lambda(\eps)} \lesssim \sum_{j=0}^{\frac{4}{\alpha+2}\log{(\eps^{-1})}}
  2^{j/2}2^{j\frac{\alpha-1}{2}} \sum_{k_1 \in K_j(\eps)} (1+|\hat k_1|) \lesssim
  \eps^{-\frac{2}{\alpha+2}} \sum_{j=0}^{\frac{4}{\alpha+2}}
  2^{j\frac{\alpha^2-2}{2(\alpha+1)}}\lesssim \eps^{-\frac{4}{\alpha+2}}.
\end{multline*}
Letting $N = \card{\Lambda(\eps)}$, we therefore have that $\eps \lesssim N^{-\frac{\alpha+2}{4}}$. This implies that
\[
\norm[L^2]{f-f_N} \lesssim \sum_{n>N} \abs{c(f)^\ast_n}^2 \lesssim N^{-\alpha/2},
\]
and this completes the proof.  

\smallskip

\paragraph{Case 2d} Let $\partial \Gamma$ be the edge curve in which two discontinuity surfaces
$\partial B_1$ and $\partial B_2$ meet inside $\intt(\supp{\psi_{j,k,m}})$. Let us assume that the
edge curve $\partial \Gamma$ is given by $(E_1(x_2,\rho(x_2)),x_2,\rho(x_2))$ with some smooth
function $\rho \in C^{\alpha}(\R)$. The other case, $(E_1(\rho(x_3),x_3),\rho(x_3),x_3)$ can be
handled in similar way. Without loss of generality, we may assume that the edge curve $\partial
\Gamma$ passes through the origin and that $(0,0,0) \in \supp{\psi_{j,k,m}}$. Let $\kappa =
\rho'(0)$, and we now consider the case $|\kappa| \leq 1$.
\begin{figure}[ht]
\centering 
%\hspace*{2cm}
\includegraphics{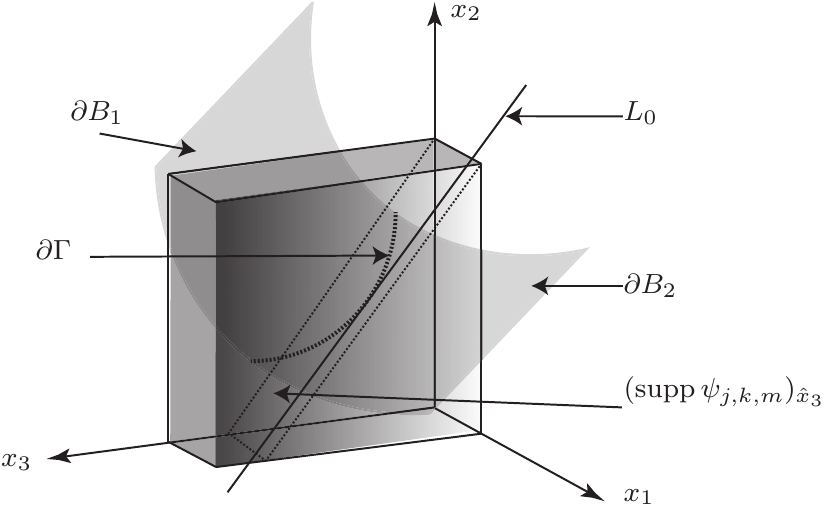}
\caption{\emph{Case 2d}. The support of $\psi_{\lambda}$ intersecting the two $C^\alpha$ discontinuity
  surfaces $\partial B_1$, $\partial B_2$ and the 1D edge curve $\pa \Gamma$, where the two patches
  $\partial B_1$ and $\partial B_2$ meet. The 2D cross section $(\supp{\psi_{j,k,m}})_{\hat{x}_3}$
  is indicated; it is seen as a tangent plane to $\pa \Gamma$.}
\label{fig:shear_2d}
\end{figure}
The other case, $\abs{\kappa} > 1$, can be handled by
switching the role of variables $x_2$ and $x_3$. Let us consider the tangent line $L_0$ to $\partial
\Gamma$ at the origin. We have
\[
L_0 : \frac{x_1}{(s_1+\kappa s_2)} = x_2 = \frac{x_3}{\kappa}, \,\,\, \text{where} \,\,
s_1 = \tfrac{\partial E_1(0,0)}{\partial x_2} \,\, \text{and} \,\, s_2 = \tfrac{\partial E_1(0,0)}{\partial x_3}.
\]
For each $\hat x_3 \in [0,2^{-j/2}]$ fixed, define 
\[
(\supp{\psi_{j,k,m}})_{\hat x_3} = \{(x_1,x_2,\kappa x_2+\hat x_3) \in \supp{\psi_{j,k,m}}:x_1,x_2 \in \R\}.
\]
Also, let 
\[
s^1_1(\hat x_3) = \tfrac{\partial E_1(\hat x_2, \hat x_3)}{\partial x_2}, \; s^1_2(\hat x_3) =
\tfrac{\partial E_1(\hat x_2, \hat x_3)}{\partial x_3}, \;
% \]
% and 
% \[
s^2_1(\hat x_3) = \tfrac{\partial E_2(\hat x'_2, \hat x_3)}{\partial x_2}, \; s^2_1(\hat x_3) = \tfrac{\partial E_2(\hat x'_2, \hat x_3)}{\partial x_2}
\]
for some $\hat x_2, \hat x'_2 \in \R$ such that 
\begin{align}\label{eq:incl-cond1}
(E_1(\hat x_2, \hat x_3), \hat x_2, \hat x_3) &\in \partial B_1 \cap (\supp{\psi_{j,k,m}})_{\hat x_3}
\\ 
\intertext{and}\label{eq:incl-cond2}
(E_1(\hat x'_2, \hat x_3), \hat x'_2, \hat x_3) &\in \partial B_2 \cap (\supp{\psi_{j,k,m}})_{\hat x_3}. 
\end{align}
%Especially, we let $s^i_{i'} = s^i_{i'}(0)$ for $i,i' = 1,2$. 
If such a point $\hat x_2$ (or $\hat x'_2$) does not exist, there will be no discontinuity curve on
$(\supp{\psi_{j,k,m}})_{\hat x_3}$ which leads to a better decay of the 2D surface integrals of the
form \eqref{eq:int}. Therefore, we may assume conditions \eqref{eq:incl-cond1} and
\eqref{eq:incl-cond1} holds for any $\hat x_3 \in [0,2^{-j/2}]$. For $k_2$ fixed, let $\hat k_1 =
(k_1+\kappa k_2)+2^{j\frac{\alpha-1}{2}}(s_1 + \kappa s_2)$. Applying a similar counting argument as
in Section~\ref{sec:proof-theorem-1}, for the shear index $k=(\hat k_1,k_2)$ fixed, we obtain an
upper bound for the number of shearlets $\psi_{j,k,m}$ intersecting $\partial \Gamma$ inside
$\cQ_{j,p}$ as follows:
\begin{equation}\label{eq:count2}
\card{\{(j,k,m) : \intt(\supp{\psi_{j,k,m}})\cap \cQ_{j,p} \cap \partial \Gamma  \neq \emptyset \}} \leq C (|\hat k_1|+1).
\end{equation} 
Notice that there exists a region $\cP$ such that the following assertions hold:
\begin{romannum}
\item $\cP$ contains $\partial \Gamma$ inside $\supp{\psi_{j,k,m}} \cap \cQ_{j,p}$. 
\item $\cP \subset \{(x_1,x_2,\kappa x_2 + t) \in \supp{\psi_{j,k,m}}: 0 \leq t \leq b\} \cap \supp{\psi_{j,k,m}}$ for some $b \ge 0$. 
\end{romannum}
Here, we choose the smallest $b$ so that (ii) holds. For each $\hat x_3 \in [0,2^{-j/2}]$ fixed, let
$H_{\hat x_3} = \{(x_1,x_2,\kappa x_2+\hat x_3) : x_1,x_2 \in \R\}$. Applying a similar argument as in
the proof of Theorem~\ref{thm:decay-hyperplane} to each of the 2D cross sections $\cP \cap H_{\hat x_3}$ of $\cP$, we
obtain
\begin{equation}\label{eq:pararell}
\vol{\cP} \lesssim 2^{-j\frac{\alpha}{2}} \Bigl(\frac{1}{|\hat k_1| 2^{j/2}}\Bigr)^{\alpha+1}.
\end{equation}
Figure~\ref{fig:shear_2d} shows the 2D cross section of $\cP$. Let us now estimate the decay rate of
shearlet coefficients $\langle f,\psi_{j,k,m}\rangle$. Using \eqref{eq:pararell},
\begin{eqnarray}
\Bigl|\int_{\R^3} f(x) \psi_{j,k,m}(x) \mathrm{d}x\Bigr| &\leq& \Bigl|\int_{\cP} f(x) \psi_{j,k,m}(x) \mathrm{d}x\Bigr| + \Bigl|\int_{\cP^{c}} f(x) \psi_{j,k,m}(x) \mathrm{d}x\Bigr| \nonumber
\\ &\leq& C\frac{2^{-j(\frac{3\alpha}{4})}}{(1+|\hat k_1|)^{\alpha+1}} + \Bigl|\int_{\cP^{c}} f(x)
\psi_{j,k,m}(x)\mathrm{d}x\Bigr| \label{eq:twoint}
\end{eqnarray}
Next, we compute the second integral $\int_{\cP^{c}} f(x) \psi_{j,k,m}(x)\mathrm{d}x$ in
\eqref{eq:twoint}. For each $\hat x_3 \in [0,2^{-j/2}]$, define
\[
(\supp{\psi_{j,k,m}})_{\hat x_3} = H_{\hat x_3} \cap \supp{\psi_{j,k,m}} \cap \cP^{c}.
\] 
Again, we assume that on each 2D cross section $(\supp{\psi_{j,k,m}})_{\hat x_3}$ there are two edge
curves $\partial B_1 \cap H_{\hat x_3}$ and $\partial B_2 \cap H_{\hat x_3}$ since we otherwise
could obtain a better decay rate of $\langle f,\psi_{j,k,m}\rangle$. As we did in the previous case,
we compute the 2D surface integral $I_{j,k,m}(\hat x_3)$ over the cross section
$(\supp{\psi_{j,k,m}})_{\hat x_3}$ defined as in \eqref{eq:int}. Applying a similar linearization
argument as in Section~\ref{sec:general-case}, we can now replace the two edge curves $\partial B_i
\cap H_{\hat x_3}$ for $i=1,2$ by two tangent lines as follows:
\[
L_1(\hat x_3) = \setpropsmall{((s^1_{1}(\hat x_3)+\kappa s^1_2(\hat x_3))x_2+\hat x_1, x_2 + \hat x_2, \kappa x_2 + \hat x_3) \in \R^3}{x_2 \in \R}
\]
and
\[
L_2(\hat x_3) = \setpropsmall{((s^2_{1}(\hat x_3)+\kappa s^2_2(\hat x_3))x_2+\hat x'_1, x_2 + \hat
  x'_2, \kappa x_2 + \hat x_3) \in \R^3}{x_2 \in \R}.
\]
Here, the points $\hat x_1, \hat x_2, \hat x'_1$, and $\hat x'_2$ are defined as
in~(\ref{eq:incl-cond1}) and (\ref{eq:incl-cond2}), and we
may assume that the two lines $L_1(\hat x_3)$ and $L_2(\hat x_3)$ do not intersect each other within
$(\supp{\psi_{j,k,m}})_{\hat x_3}$; otherwise, we can take secant lines instead as argued in the previous
case. Let $\cQ_{\hat x_3}$ be the projection of $(\supp{\psi_{j,k,m}})_{\hat x_3}$ onto the $x_1x_2$
plane. By the assumptions on $\psi$, we have
\begin{eqnarray*}
I_{j,k,m}(\hat x_3) &=& \sqrt{1+{\kappa^2}}\int_{{\cQ}_{\hat x_3}} f(x_1,x_2,\kappa x_2 + \hat
x_3)\psi_{j,k,m}(x_1,x_2,\kappa x_2 + \hat x_3) \mathrm{d}x_2 \mathrm{d}x_1 
\\ &=& 2^{j\frac{\alpha+2}{4}}\sqrt{1+{\kappa^2}}\int_{{\cQ}_{\hat x_3}}f(x_1,x_2,\kappa x_2 + \hat x_3) 
\\ &&
\phantom{\sqrt{1+{\kappa^2}}} g^0_{\kappa,2^{j/2}\hat x_3}
\Bigl(2^{j\alpha/2} x_1+ 2^{j/2}(k_1+k_2\kappa)x_2+ 2^{j/2} k_2 \hat
x_3, 2^{j/2}x_2 \Bigr) \mathrm{d}x_2\mathrm{d}x_1
\end{eqnarray*}
The integral above is of the same type as in \eqref{eq:hp-split-into-I_l} except for the $\hat x_3 $
translation parameter. The function $f(x_1,x_2,\kappa x_2 + \hat x_3)$ has singularities lying on
the projection of the lines $L_1(\hat x_3)$ and $L_2(\hat x_3)$ onto the $x_1x_2$ plane which do not
intersect inside $\intt(\cQ_{\hat x_3})$. Therefore, we can apply the linearized estimate as in the proof
of Theorem~\ref{thm:decay-hyperplane} and obtain
\begin{equation*}
|I_{j,k,m}(\hat x_3)| \leq C \max_{i=1,2} \set{2^{-j\frac{\alpha}{4}}{\Bigl(1+|(k_1+\kappa k_2)+2^{j\frac{\alpha-1}{2}}(s^i_1(\hat x_3)+\kappa s^i_2(\hat x_3))|\Bigr)^{-\alpha-1}}}.
\end{equation*}
By a similar argument as in (\ref{eq:universal-s}), we can now replace $s^i_{i'}(\hat x_3)$ by  universal choices
$s_i$ for $i,i'=1,2$ respectively, in the equation above. This implies
\begin{equation}\label{eq:2ndint}
\Bigl| \int_{\cP^c} f(x) \psi_{j,k,m}(x) \mathrm{d}x \Bigr| \leq C \frac{2^{-j\frac{\alpha+2}{4}}}{(1+|\hat k_1|)^{\alpha+1}}. 
\end{equation}
Therefore, from \eqref{eq:twoint}, \eqref{eq:2ndint}, we obtain 
\begin{equation}\label{eq:decay2}
|\langle f,\psi_{j,k,m}\rangle| \leq C \frac{2^{-j\frac{\alpha+2}{4}}}{(1+|\hat k_1|)^{\alpha+1}}. 
\end{equation}
In this case, the number of all dyadic boxes $\cQ_{j,p}$ containing two distinct discontinuity
surfaces is bounded above by $2^{j/2}$ up to a constant independent of scale $j$, and 
there are shear indices $\lceil 2^{j\frac{\alpha-1}{2}}\rceil+1$ with respect to $k_2$. 
Let us define 
\[
K_j(\eps) = \setprop{k_1 \in \Z }{ (1+|\hat k_1|)^{-(\alpha+1)}2^{-j\frac{\alpha+2}{4}} > \eps }.  
\]
Finally, we now estimate $\card{\Lambda(\eps)}$ using \eqref{eq:count2} and \eqref{eq:decay2}. 
\[
\card{\Lambda(\eps)} \leq C \sum_{j=0}^{\frac{4}{\alpha+2}\log{(\eps^{-1})}} 2^{j\frac{\alpha-1}{2}}
\, 2^{j/2} \sum_{k_1 \in K_j(\eps)} (1+|\hat k_1|) \leq C \eps^{-\frac{4}{\alpha+2}}
\]
which provides the sought approximation rate. \\

\paragraph{Case 2e} In this case, we assume that $f = f_0 \chi_{\Omega_0} + f_1 \chi_{\Omega_1}$
with $f_0,f_1 \in C^{\beta}$, and that there are $L$ discontinuity surfaces
$\partial B_1, \dots \partial B_L$ inside $\intt(\supp{\psi_{j,k,m}})$ so that each of the
discontinuity surfaces is parametrized by $x_1 = E_{i}(x_2,x_3)$ with $E_i \in C^{\alpha}$ for
$i=1,\dots,L$. For each $\hat x_3 \in \itvcc{0}{2^{-j/2}}$, let us consider the 2D support
\[(\supp{\psi_{j,k,m}})_{\hat x_3} = \{(x_1,x_2,\hat x_3) \in \supp{\psi_{j,k,m}} : x_1,x_2 \in \R\}.\] 
On each 2D slice $(\supp{\psi_{j,k,m}})_{\hat x_3}$, let 
\[\partial \Gamma^{i}_{\hat x_3} = (\supp{\psi_{j,k,m}})_{\hat x_3} \cap \partial B_i \qquad \text{for} \,\, i = 1,\dots, L.\] 
Observe that there are at most two distinct curves $\partial \Gamma^i_{\hat x_3}$ and $\partial
\Gamma^{i'}_{\hat x_3}$ on $(\supp{\psi_{j,k,m}})_{\hat x_3}$ for some $i,i' = 1,\dots, L$. We can
 assume that there are such two edge curves $\partial \Gamma^1_{\hat x_3}$ and $\partial
\Gamma^2_{\hat x_3}$ for each $\hat x_3 \in [0,2^{-j/2}]$ since we otherwise could obtain better decay
rate of the shearlet coefficients $|\langle f,\psi_{j,k,m}\rangle|$. From this, we may assume that for
each $\hat x_3$, there exist $(\hat x_1, \hat x_2, \hat x_3)$ and $(\hat x'_1,\hat x'_2,\hat x_3)
\in \intt(\supp{\psi_{j,k,m}})$ such that $(\hat x_1, \hat x_2, \hat x_3) \in \partial
\Gamma^1(\hat x_3)$ and $(\hat x'_1, \hat x'_2, \hat x_3) \in \partial \Gamma^{2}(\hat x_3)$. We
then set:
\[
s^1_1(\hat x_3) = \tfrac{\partial E_{1}(\hat x_1, \hat x_2)}{\partial x_2}
\qquad \text{and} \qquad 
s^{2}_1(\hat x_3) = \tfrac{\partial E_{2}(\hat x_1, \hat x'_2)}{\partial x_2}.
\]
Applying a similar linearization argument as in Section~\ref{sec:general-case}, we can replace the two edge curves by two
tangent lines (or secant lines) as follows:
\[
L^1(\hat x_3) = \setprop{(s^1_1(\hat x_3)x_2+\hat x_1,x_2+\hat x_2,\hat x_3)}{ x_2 \in \R}
\]
and
\[
L^{2}(\hat x_3) = \setprop{(s^{2}_1(\hat x_3)x_2+\hat x'_1,x_2+\hat x'_2,\hat x_3)}{ x_2 \in \R}.
\]
Here, we may assume that the two tangent lines $L^1(\hat x_3)$ and $L^{2}(\hat x_3)$ do not intersect
inside $(\supp{\psi_{j,k,m}})_{\hat x_3} \cap \cQ_{j,p}$ for each $\hat x_3$. In fact, the number of
shearlet supports $\psi_{j,k,m}$ intersecting $\cQ_{j,p}\cap \partial B_1 \cap \dots \cap \partial B_L$, so
that there are two tangent lines $L^1(\hat x_3)$ and $L^{2}(\hat x_3)$ meeting  each other inside
$(\supp{\psi_{j,k,m}})_{\hat x_3}$ for some $\hat x_3$, is bounded by some constant $C$ independent
of scale $j$. Those shearlets $\psi_{j,k,m}$ are covered by Case~2f, and  we may therefore simply
ignore those shearlets in this case. Using a similar argument as in the estimate of
(\ref{eq:hp-split-into-I_l}), one can then estimate $I_{j,k,m}(\hat x_3)$ defined as in \eqref{eq:int} as follows:
\[
I_{j,k,m}(\hat x_3) \leq C \min_{i = 1,2} \set{\frac{2^{-j\frac{\alpha}{4}}}{(1+|k_1+2^{j\frac{\alpha-1}{2}}s^i_1(\hat x_3)|)^{\alpha+1}}}.
\]
Again, applying similar arguments as in~(\ref{eq:universal-s}), we may replace the slopes $s^i_1(\hat x_3)$ and $s^{i'}_1(\hat x_3)$ by universal choices $s^i_1(0)$ and $s^{i'}_1(0)$, respectively. 
This gives 
\begin{equation}\label{eq:decay3}
|\langle f,\psi_{j,k,m}\rangle| \leq C \max_{i=1,\dots,L} \set{ \frac{2^{-j\frac{\alpha+2}{4}}}{(1+|\hat k^{i}_1|)^{\alpha+1}}},
\end{equation}
where $\hat k^i_1 = s^i_1(0)2^{j\frac{\alpha-1}{2}}+k_1$ for $i = 1,\dots,L$.
Further, applying a similar counting argument as in Section~\ref{sec:proof-theorem-1}, for $k =
(k_1,k_2)$ and $j\ge 0$ fixed, we have 
\begin{multline}\label{eq:count3}
\card{\set{(j,k,m)}{\intt(\supp{\psi_{j,k,m}})\cap\partial B_1 \cap \dots \cap \partial B_L \cap
    \cQ_{j,p} \neq \emptyset }} \\ \leq C \min_{i=1,\dots,L}\setbig{1+|\hat k^i_1|}.
\end{multline}
In this case, the number of all dyadic boxes $\cQ_{j,p}$ containing more than two distinct discontinuity surfaces is bounded by some constant independent of scale $j$, and 
there are $\lceil 2^{j\frac{\alpha-1}{2}}\rceil+1$ shear indices with respect to $k_2$. 
Let us define 
\[
K_j(\eps) = \setprop{k_1 \in \Z}{ \max_{i=1,\dots,L}\setbig{(1+|\hat k^i_1|)^{-(\alpha+1)}2^{-j\frac{\alpha+2}{4}}} > \eps }.  
\]
Finally, using \eqref{eq:decay3} and \eqref{eq:count3}, we see that
\[
|\Lambda(\eps)| \leq C \sum_{j=0}^{\frac{4}{\alpha+2}\log{(\eps^{-1})}} 2^{j\frac{\alpha-1}{2}} \sum_{k_1 \in K_j(\eps)} \min_{i=1,\dots,L}\setbig{1+|\hat k^i_1|} \leq C \eps^{-\frac{2}{\alpha+4}}.
\]
This proves Case 2e.

\paragraph{Case 2f} In this case, since the total number of shear parameters $k = (k_1,k_2)$ is
bounded by a constant times $2^j$ for each $j\ge 0$, it follows that
\[
\card{\Lambda_{j,p}(\eps)} \leq C \cdot 2^{j}.
\]
Since there are only finitely many corner points with its number not depending on scale $j \ge 0$, we have
\[
\card{\Lambda(\eps)} \leq C \cdot \sum_{j=0}^{\frac{4}{\alpha+2}\log_2{(\eps^{-1})}} 2^{j} \leq C \cdot \eps^{-\frac{4}{\alpha+2}},
\]
which, in turn, implies the optimal sparse approximation rate for Case~2f. This completes the proof of
Theorem~\ref{thm:3d-opt-sparse-piecewise}.

\section{Extensions}
\label{sec:extensions}

\subsection{Smoothness parameters $\alpha$ and $\beta$} %Limitations}
Our 3D image model class $\cE_\alpha^\beta(\R^3)$ depends primarily of the two parameters $\alpha$
and $\beta$. 
The particular choice of scaling matrix is essential for the nearly optimal approximation results in
Section~\ref{sec:optimal-sparsity-3d}, but any choice of scaling matrix basically only allows us to
handle one parameter. This of course poses a problem if one seeks optimality results for all
$\alpha, \beta \in \itvoc{1}{2}$. We remark that our choice of scaling matrix exactly ``fits'' the
smoothness parameter of the discontinuity surface $\alpha$, which exactly is the crucial parameter when
$\beta \ge \alpha$ as assumed in our optimal sparsity results. It is unclear whether one can
circumvent the problem of having ``too'' many parameters,
and thereby prove sparse approximation results as in Section~\ref{sec:optimal-sparsity-3d} for the
case $\beta<\alpha \le 2$.

For $\alpha > 2$ we can, however, not expect shearlet systems
$\SH{(\phi,\psi,\tilde\psi,\breve\psi)}$ to deliver optimal sparse approximations. The heuristic
argument is as follows. For simplicity let us only consider shearlet elements associated with the
pyramid pair $\cP$. Suppose that  the discontinuity surface is $C^2$. Locally we can assume the surface
will be of the form $x_1 = E(x_2,x_3)$ with $E \in C^2$. 
Consider a Taylor expansion of $E$ at $(x_2',x_3')$: 
 \begin{multline}\label{eq:taylor-exp-dB}
  E(x_2,x_3) = E(x_1',x_2') + \begin{pmatrix} \pa^{(1,0)}E(x_1',x_2') &
    \pa^{(0,1)} E(x_1',x_2')
  \end{pmatrix}
  \begin{pmatrix}
    x_2 \\ x_3
  \end{pmatrix}\\
  + \begin{pmatrix}x_2 & x_3
  \end{pmatrix}
  \begin{pmatrix}
    \pa^{(2,0)} E(\xi_1,\xi_2) & \pa^{(1,1)} E(\xi_1,\xi_2)\\
    \pa^{(1,1)} E(\xi_1,\xi_2) & \pa^{(0,2)} E(\xi_1,\xi_2)
  \end{pmatrix}
  \begin{pmatrix}
    x_2 \\ x_3
  \end{pmatrix}.
\end{multline}
Intuitively, we need our shearlet elements $\psi_{j,k,m}$ to capture the geometry of $\pa B$. For
the term $E(x_1',x_2')$ we use the translation parameter $m \in \Z^3$ to locate the shearlet element
near the expansion point $p:=(E(x_1',x_2'),x_2',x_3')$. Next, we ``rotate'' the element
$\psi_{j,k,m}$ using the sharing parameter $k \in \Z^2$ to align the shearlet normal with the normal
of the tangent plane of $\pa B$ in $p$; the direction of the tangent is of course governed by
$\pa^{(1,0)}E(x_1',x_2')$ and $\pa^{(0,1)} E(x_1',x_2')$. Since the last parameter $j \in \N_0$ is a
multi-scale parameter, we do not have more parameters available to capture the geometry of $\pa B$.
Note that the scaling matrix $A_{2^j}$ can, for $\alpha =2$, be written as
\[ A_{2^j} =
\begin{pmatrix}
  2^{j} & 0 & 0 \\
 0 & 2^{j/2} & 0 \\
 0 & 0 & 2^{j/2} 
\end{pmatrix} = \begin{pmatrix}
  2 & 0 & 0 \\
 0 & 2^{1/2} & 0 \\
 0 & 0 & 2^{1/2} 
\end{pmatrix}^j.
\] 
The shearlet element will therefore have support in a parallelopiped
with side lengths $2^{-j}$ , $2^{-j/2}$ and $2^{-j/2}$ in directions of the
$x_1$, $x_2$, and $x_3$ axis, respectively. Since
\[ \abs{x_2x_3} \le 2^{-j}, x_2^2 \le 2^{-j}, \text{ and } x_3^2 \le 2^{-j}, \] for
$\abs{x_2},\abs{x_3} \le 2^{-j/2}$, we see that the paraboliodal scaling gives shearlet elements of
a size that exactly fits the Hermitian term in (\ref{eq:taylor-exp-dB}). If $\pa B\in C^\alpha$ for $1
< \alpha \le 2$, that is, $E\in C^\alpha$ for $1 < \alpha \le 2$, we in a similar way see that our
choice of scaling matrix exactly fits the last term in the corresponding Taylor expansion. Now, if
the discontinuity surface is smoother than $C^2$, that is, $\pa B \in C^\alpha$ for $\alpha > 2$,
say $\pa B \in C^3$, we could include one more term in the Taylor expansion
(\ref{eq:taylor-exp-dB}), but we do not have any more free parameters to adapt to this increased
information. Therefore, we will arrive at the same (and now non-optimal) approximation rate as for
$\pa B\in C^2$. We conclude that for $\alpha >2$ we will need representation systems with not only a
directional characteristic, but also some type of curvature characteristic.

For $\alpha <1$, we do not have proper directional information about the anisotropic discontinuity,
in particular, we do not have a tangential plane at every point on the discontinuity surface. This
suggests that this kind of anisotropic phenomenon should not be investigated with \emph{directional}
representation systems. For the boarder-line case $\alpha =1$, our analysis shows that wavelet
systems should be used for sparse approximations.

\subsection{Needle-like shearlets}
\label{sec:needle-like-shearl}

In place of $A_{2^j}=\diag{(2^{\alpha j/2}, 2^{j/2}, 2^{j/2})}$, one could also
use the scaling matrix $A_{2^j}=\diag{(2^{j\scp/2}, 2^{j\scp/2}, 2^{j/2})}$ with similar changes for
$\tilde A_{2^j}$ and $ \breve A_{2^j}$. This would lead to needle-like shearlet elements instead of
the plate-like elements considered in this paper. As Theorem~\ref{thm:3d-opt-sparse-piecewise} in
Section~\ref{sec:sparse-appr-3d} showed, the plate-like shearlet systems are able to deliver almost
optimal sparse approximation even in the setting of cartoon-like images with certain types of 1D
singularities. This might suggest that needle-like shearlet systems are not necessary, at least not
for sparse approximation issues. Furthermore, the tiling of the frequency space becomes increasingly
complicated in the situation of needle-like shearlet systems which yields frames with less favorable
frame constants. However, in non-asymptotic analyses, \eg image separation, a combined needle-like
and plate-like shearlet system might be useful.

\subsection{Future work} %/Open problems [[other name?]]}
For $\alpha<2$, the obtained approximation error rate is only near-optimal since it differs by
$\tau(\alpha)$ from the true optimal rate. It is unclear whether one can get rid of the
$\tau(\alpha)$ exponent (perhaps replacing it with a poly-$\log$ factor) by using better estimates
in the proofs in Section~\ref{sec:discontinuity}. More general, it is also future work to determine
whether shearlet systems with $\alpha, \beta \in \itvoc{1}{2}$ provide nearly or truly optimal
sparse approximations of all $f \in \cE_\scp^\beta(\R^3)$. To answer this question, one would,
however, need to develop a completely new set of techniques. This would mean that the
approximation error would decay as $O(N^{-\min\{\scp/2,2\beta/3\}})$ as $N \to \infty$, perhaps with
additional poly-log factors or a small polynomial factor.

\bigskip 

\begin{footnotesize}
  \paragraph{\bf Acknowledgements} The first and third author acknowledge
  support from DFG Grant SPP-1324, KU 1446/13. The first author also
  acknowledges support from DFG Grant KU 1446/14.
\end{footnotesize}

\appendix

\section{Estimates}
\label{sec:estimates}
The following estimates are used repeatedly in Section~\ref{sec:constr-comp-supp} and
follows by direct verification. For $t = 2^{-m}$, \ie $-\log_2 t=m$, $m \in \N_0 := \N \cup \{0\}$,
we have
\begin{align*}
  \sum_{\setprop{j\in\N_0}{ 2^{-j} \ge t}} (2^{-j})^{-\iota} = \sum_{j=0}^{{-\log_2 t}}
  (2^\iota)^j = \frac{t^{-\iota}-2^{-\iota}}{1-2^{-\iota}} \qquad \text{for $\iota \neq 0$,}\\
  \sum_{\setprop{j\in\N_0}{ 2^{-j} \le t}} (2^{-j})^{\iota} = \sum_{j={-\log_2 t}}^\infty
  (2^{-\iota})^j = \frac{t^{\iota}}{1-2^{-\iota}} \qquad \text{for $\iota > 0$,}
\end{align*}
For $t \in \itvoc{0}{1}$, we have $\ceil{-\log_2 t} \in \N_0$ and therefore
\begin{align}
  \sum_{\setprop{j\in\N_0}{ 2^{-j} \ge t}} (2^{-j})^{-\iota} = \sum_{j=0}^{\floor{-\log_2 t}}
  (2^\iota)^j \le \frac{t^{-\iota}-2^{-\iota}}{1-2^{-\iota}} \qquad
  \text{for $\iota > 0$,} \label{eq:quotientsum-finite}\\
  \sum_{\setprop{j\in\N_0}{ 2^{-j} \le t}} (2^{-j})^{\iota} = \sum_{j=\ceil{-\log_2 t}}^\infty
  (2^{-\iota})^j \le \frac{t^{\iota}}{1-2^{-\iota}} \qquad \text{for
    $\iota > 0$,} \label{eq:quotientsum-inf}
\end{align}
where we have used that $2^{\floor{-\log_2 t}}\le t^{-1}$ and
$2^{-\ceil{-\log_2 t}} = 2^{\floor{\log_2 t}} \le t$. For $t>1$ we
finally have that
\begin{equation}
  \sum_{\setprop{j\in\N_0}{ 2^{-j} \ge t}} (2^{-j})^{-\iota} = 0
\quad \text{and} \quad \sum_{\setprop{j\in\N_0}{ 2^{-j} \le t}} (2^{-j})^{\iota} =
\sum_{j=0}^\infty (2^{-\iota})^{j} = \frac{1}{1-2^{-\iota}}. \label{eq:quotientsum-t-big}
\end{equation}

\section{Proof of Proposition~\ref{prop:estimateR}}
\label{sec:proof-proposition-Rc}
We start by estimating $\Gamma(2\omega)$, and will use this later to derive the claimed upper
estimate for $R(c)$. For brevity we will use
$K_j:=\itvcc{-\ceilsmall{2^{j(\scp-1)/2}}}{\ceilsmall{2^{j(\scp-1)/2}}}$ and $k \in K_j$ to mean
$k_1, k_2 \in K_j$. By definition it then follows that
\begin{multline*}
  \Gamma(2\omega_{1},2\omega_{2},2\omega_3) \\
  \le \esssup_{\xi \in {\R^{3}}} \; \sum_{j \ge 0} \sum_{k \in
    K_j} \abs{\hat{\psi}\left(2^{-j\scp/2}
      \xi_{1}, k_12^{-j\scp/2}\xi_{1}+2^{-j/2}\xi_{2},
      k_22^{-j\scp/2}\xi_{1}+2^{-j/2}\xi_{3}\right)} \\
  \cdot  \abs{\hat{\psi}\left(2^{-j\scp/2}\xi_{1} + 2\omega_{1},
      k_12^{-j\scp/2}\xi_{1}+2^{-j/2}\xi_{2}+2\omega_{2},k_22^{-j\scp/2}\xi_{1}+2^{-j/2}\xi_{3}+2\omega_3
    \right)}.
\end{multline*}
For each $(\omega_{1}, \omega_{2}, \omega_3) \in {\R^{3}}\setminus \set{0}$, we first split the sum
over the index set $\N_0$ into index sets $J_1=\setprop{j\ge0}{ \abssmall{2^{-j\scp/2}\xi_{1}} \le
  \norm[\infty]{\omega}}$ and $J_2=\setprop{j\ge0}{ \abssmall{2^{-j\scp/2}\xi_{1}} >
  \norm[\infty]{\omega}}$. We denote these sums by $I_1$ and $I_2$, respectively. In other words, we
have that
\begin{align}
  \Gamma(2\omega_{1},2\omega_{2},2\omega_3) \le \esssup\limits_{\xi
    \in {\mathbb{R}^{3}}} (I_{1}+I_{2}),
\label{eq:Gamma-esssupI}
\end{align}
where
\begin{multline*}
  I_1 =\sum_{j\in J_1} \sum_{k \in K_j}
  \absBig{\hat{\psi}\bigl(2^{-j\scp/2} \xi_{1},
      k_12^{-j\scp/2}\xi_{1}+2^{-j/2}\xi_{2},
      k_22^{-j\scp/2}\xi_{1}+2^{-j/2}\xi_{3}\bigr)} \\
  \cdot   \absBig{\hat{\psi}\bigl(2^{-j\scp/2}\xi_{1} + 2\omega_{1},
      k_12^{-j\scp/2}\xi_{1}+2^{-j/2}\xi_{2}+2\omega_{2},k_22^{-j\scp/2}\xi_{1}+2^{-j/2}\xi_{3}+2\omega_3
    \bigr)}
\end{multline*}
and 
\begin{multline*}
  I_2 =\sum_{j\in J_2} \sum_{k \in K_j}
  \abs{\hat{\psi}\left(2^{-j\scp/2} \xi_{1},
      k_12^{-j\scp/2}\xi_{1}+2^{-j/2}\xi_{2},
      k_22^{-j\scp/2}\xi_{1}+2^{-j/2}\xi_{3}\right)} \\
   \cdot \abs{\hat{\psi}\left(2^{-j\scp/2}\xi_{1} + 2\omega_{1},
      k_12^{-j\scp/2}\xi_{1}+2^{-j/2}\xi_{2}+2\omega_{2},k_22^{-j\scp/2}\xi_{1}+2^{-j/2}\xi_{3}+2\omega_3
    \right)}.
\end{multline*}

The next step consists of estimating $I_1$ and $I_2$, but we first
introduce some useful inequalities which will be needed later. Recall
that $\vap>2\dep > 6$, and $q,q',r,s$ are positive constants
satisfying $q',r,s \in \itvoo{0}{q}$. Further, let $\dep'' =
\dep-\dep'$ for an arbitrarily fixed $\dep'$ satisfying $1 < \dep' <
\dep-2$. Let $\iota>\dep>3$. Then we have the following inequalities
for $x,y,z \in \R$.
\begin{equation}
  \min\{1,|qx|^{\iota}\}%\min\{1,|q'x|^{-\dep}\}
  \min\{1,|ry|^{-\dep}\} 
  \leq
  \min\{1,|qx|^{\iota-\dep}\}%\min\{1,|q'x|^{-\dep}\}
  \min\left\{1,|        (qx)^{-1}ry|^{-\dep}\right\},  \label{eq:ineq1}
\end{equation}
\begin{equation}\label{eq:ineq2}
  \min\{1,|x|^{-\dep}\}\min\left\{1,\left| \frac{1+z}{x+y}\right|^{\dep}\right\} \le
2^{\dep''}|y|^{-\dep''}\min\{1,|x|^{-\dep'}\}
  \max\{1,|1+z|^{\dep''}\},
\end{equation}
\begin{equation}\label{eq:ineq3}
  \min\{1,|qx|^{\iota-\dep}\}\min\{1,|q'x|^{-\dep}\}|x|^{\dep''} \le (q')^{-\dep''},
\end{equation}
and
\begin{equation}\label{eq:ineq4}
  \min\{1,|qx|^{\iota-\dep}\}\min\{1,|q'x|^{-\dep}\}|x|^{\dep''} \le (q')^{-\dep''}
  \min\{1,|qx|^{\iota-\dep+\dep''}\}\min\{1,|q'x|^{-\dep'}\}.
\end{equation}
We fix $\xi \in \R^3$ and start with $I_1$. By the decay assumptions
\eqref{eq:psi} on $\hat\psi$, it follows directly that
\begin{multline*} I_1 \le \sum_{j\in J_1}
  \min\set{|q2^{-j\scp/2}\xi_{1}|^{\vap},1}\min
  \set{|q'2^{-j\scp/2}\xi_{1}|^{-\dep},1} \\ \cdot \min
  \set{\abs{q(2^{-j\scp/2}\xi_{1} + 2\omega_{1})}^{\vap},1} \min\set{
    \abs{q'(2^{-j\scp/2}\xi_{1} + 2\omega_{1})}^{-\dep},1} \\
  \sum_{k_1 \in K_j}
  \min\setbig{\absbig{r(k_12^{-j\scp/2}\xi_1+ 2^{-j/2}\xi_2)}^{-\dep}}
  \min\setbig{\absbig{r(k_12^{-j\scp/2}\xi_1+
      2^{-j/2}\xi_2+2\omega_2)}^{-\dep}}  \\ \sum_{k_2 \in
    K_j} 
\min\setbig{\absbig{s(k_22^{-j\scp/2}\xi_1\!+
      2^{-j/2}\xi_3)}^{-\dep}} \min\setbig{\absbig{s(k_22^{-j\scp/2}\xi_1\!+
      2^{-j/2}\xi_3+2\omega_3)}^{-\dep}}.
\end{multline*}
Further, using inequality~\eqref{eq:ineq1} with $\iota=\vap$ and
$\iota=2\vap$ twice,
\begin{multline} I_1 \le \sum_{j\in J_1}
  \min\set{|q2^{-j\scp/2}\xi_{1}|^{\vap-2\dep},1}\min
  \set{|q'2^{-j\scp/2}\xi_{1}|^{-\dep},1} \\ \cdot \min
  \set{\abs{q(2^{-j\scp/2}\xi_{1} + 2\omega_{1})}^{\vap-2\dep},1}
  \min\set{ \abs{q'(2^{-j\scp/2}\xi_{1} + 2\omega_{1})}^{-\dep},1} \\
  \sum_{k_1 \in \Z} \min\set{\abs{\frac rq
      \left(k_1+2^{j\frac{\scp-1}{2}}\tfrac{\xi_{2}}{\xi_{1}}\right)}^{-{\dep}},1} \\
  \min\set{ \abs{\frac rq
      \left[\left(\frac{2\omega_{2}}{2^{-j\scp/2}\xi_{1}}\right)
        +\left(k_1 + 2^{j\frac{\scp-1}{2}}\tfrac{\xi_{2}}{
            \xi_{1}}\right)\right]}^{{-\dep}}
    \abs{1+\frac{2\omega_{1}}{2^{-j\scp/2}\xi_{1}}}^{\dep},1 } \\
  \sum_{k_2 \in \Z} \min\set{\abs{\frac sq
      \left(k_2+2^{j\frac{\scp-1}{2}}\tfrac{\xi_{3}}{\xi_{1}}\right)}^{-{\dep}},1}\\
  \min\set{ \abs{\frac rq
      \left[\left(\frac{2\omega_{3}}{2^{-j\scp/2}\xi_{1}}\right)
        +\left(k_2 + 2^{j\frac{\scp-1}{2}}\tfrac{\xi_{3}}{
            \xi_{1}}\right)\right]}^{{-\dep}}
    \abs{1+\frac{2\omega_{1}}{2^{-j\scp/2}\xi_{1}}}^{\dep},1 },
\label{eq:pf_estimateR2}
\end{multline}
where we, \eg in the sum over $k_1$, have used paraphrases as
 \[
  \frac{r(k_12^{-j\scp/2}\xi_1+2^{-j/2} \xi_2)}{q2^{-j\scp/2}\xi_1} =\frac rq
  \left(k_1+2^{j\frac{\scp-1}{2}}\tfrac{\xi_{2}}{\xi_{1}}\right)
\] 
and
\begin{multline*}
  \frac{r(k_12^{-j\scp/2}\xi_1+2^{-j/2}\xi_2+2\omega_2)}{q(2^{-j\scp/2}\xi_1+2\omega_1)} = \frac rq
  \Biggl[\left(\frac{2\omega_{2}}{2^{-j\scp/2}\xi_{1}}\right) \\ +\left(k_1 +
      2^{j\frac{\scp-1}{2}}\tfrac{\xi_{2}}{ \xi_{1}}\right)\Biggr]
  \left(1+\frac{2\omega_{1}}{2^{-j\scp/2}\xi_{1}} \right)^{-1}.
\end{multline*}

\smallskip

We now consider the following three cases: $\norm[\infty]{\omega} =
\abs{\omega_{1}} \ge \abs{2^{-j\scp/2}\xi_{1}}$,
$\norm[\infty]{\omega} = |\omega_{2}| \ge \abs{2^{-j\scp/2}\xi_{1}}$,
and $\norm[\infty]{\omega} = |\omega_{3}| \ge
\abs{2^{-j\scp/2}\xi_{1}}$. Notice that these three cases indeed do
include all possible relations between $\omega$ and $\xi_1$.

%\smallskip

\paragraph{Case I} We assume that
$\norm[\infty]{\omega} = \abs{\omega_{1}} \ge
\abs{2^{-j\scp/2}\xi_{1}}$, hence
$\abs{2^{-j\scp/2}\xi_{1}+2\omega_{1}} \ge \abs{\omega_{1}}$. Using
the trivial estimates $\min \setsmall{\abs{q(2^{-j\scp/2}\xi_{1} +
    2\omega_{1})}^{\vap-2\dep},1}\le 1$, \[ \min\set{ \abs{\frac rq
    \left[\left(\frac{2\omega_{2}}{2^{-j\scp/2}\xi_{1}}\right)
      +\left(k_1 + 2^{j\frac{\scp-1}{2}}\tfrac{\xi_{2}}{
          \xi_{1}}\right)\right]}^{{-\dep}}
  \abs{1+\frac{2\omega_{1}}{2^{-j\scp/2}\xi_{1}}}^{\dep},1 } \le 1, \]
and analogue estimates for the sum over $k_2$, we can continue
\eqref{eq:pf_estimateR2},
\begin{multline*} I_1 \le \sum_{j\in J_1}
  \min\set{|q2^{-j\scp/2}\xi_{1}|^{\vap-2\dep},1}\min
  \set{|q'2^{-j\scp/2}\xi_{1}|^{-\dep},1} \abs{q'(2^{-j\scp/2}\xi_{1}
    + 2\omega_{1})}^{-\dep} \\ \sum_{k_1 \in \Z} \min\set{\abs{\frac
      rq
      \left(k_1+2^{j\frac{\scp-1}{2}}\tfrac{\xi_{2}}{\xi_{1}}\right)}^{-{\dep}},1}
  \sum_{k_2 \in \Z} \min\set{\abs{\frac sq
      \left(k_2+2^{j\frac{\scp-1}{2}}\tfrac{\xi_{3}}{\xi_{1}}\right)}^{-{\dep}},1}.
\end{multline*}
Our assumption $\norm[\infty]{\omega} = \abs{\omega_{1}}$ implies
$\abs{q'(2^{-j\scp/2}\xi_{1} + 2\omega_{1})}^{-\dep} \le
\norm[\infty]{q'\omega}^{-\dep} $. Therefore,
\begin{multline*} I_1 \le \norm[\infty]{q'\omega}^{-\dep} \sum_{j\in
    J_1} \min\set{\abs{q2^{-j\scp/2}\xi_{1}}^{\vap-2\dep},1}\min
  \set{\abs{q'2^{-j\scp/2}\xi_{1}}^{-\dep},1} \\ \frac qr \sum_{k_1
    \in \Z} \frac rq \min\set{\abs{\frac rq
      \left(k_1+2^{j\frac{\scp-1}{2}}\tfrac{\xi_{2}}{\xi_{1}}\right)}^{-{\dep}},1}
  \cdot \frac qs \sum_{k_2 \in \Z} \frac sq \min\set{\abs{\frac sq
      \left(k_2+2^{j\frac{\scp-1}{2}}\tfrac{\xi_{3}}{\xi_{1}}\right)}^{-{\dep}},1}.
\end{multline*}
By the estimate (\ref{eq:sk2}) with $y=r/q\le 1$ (and $y=s/q \le 1$) as
constant, we can bound the sum over $k_1$ (and $k_2$), leading to
\begin{equation*} 
  I_1 \le \norm[\infty]{q'\omega}^{-\dep} \sum_{j\in
    J_1} \min\set{\abs{q2^{-j\scp/2}\xi_{1}}^{\vap-2\dep},1}\min
  \set{\abs{q'2^{-j\scp/2}\xi_{1}}^{-\dep},1} \, \frac qr C(\dep) \,
  \frac qs C(\dep).
\end{equation*}
Taking the supremum over $\xi_1 = \eta_1/q \in \R$ and using equations
(\ref{eq:quotientsum-finite}) and (\ref{eq:quotientsum-inf}) as in the
proof of Proposition~\ref{prop:Lsupfinite} yields
\begin{align} 
  I_1 &\le \frac{q^2}{rs} C(\dep)^2 \norm[\infty]{q'\omega}^{-\dep}
  \sup_{\eta_1\in\R} \sum_{j\in J_1}
  \min\set{\abs{2^{-j\scp/2}\eta_{1}}^{\vap-2\dep},1}\min
  \set{\abs{q'q^{-1}2^{-j\scp/2}\eta_{1}}^{-\dep},1} \nonumber \\
  &\le \frac{q^2}{rs} C(\dep)^2 \norm[\infty]{q'\omega}^{-\dep} \Bigl(
  \ceil{\frac{2}{\alpha}\log_{2}\Bigl(\frac{q}{q'}\Bigr)}
  +\frac{1}{1-2^{-\vap+2\dep}}+1\Bigr). \label{eq:I1-case1}
\end{align}

\paragraph{Case II} We now assume that
$\norm[\infty]{\omega} = \abs{\omega_{2}} \ge
\abs{2^{-j\scp/2}\xi_{1}}$. For $\dep=\dep'+\dep''$, $\dep>\dep'+2>3$,
$\dep'>1$, $\dep''>2$ by \eqref{eq:ineq2} 
\begin{multline*}
  \min\set{\abs{\frac rq
      \left(k_1+2^{j\frac{\scp-1}{2}}\tfrac{\xi_{2}}{\xi_{1}}\right)}^{-{\dep}},1}
  \min\set{ \abs{\frac{1+\frac{2\omega_{1}}{2^{-j\scp/2}\xi_{1}}}{
        \frac rq \left(\frac{2\omega_{2}}{2^{-j\scp/2}\xi_{1}} +k_1 +
          2^{j\frac{\scp-1}{2}}\frac{\xi_{2}}{ \xi_{1}}\right) }}^\dep,1} \\
  \le 2^{\dep''}
  \abs{\frac{r}{q}\frac{2\omega_2}{2^{-j\scp/2}\xi_1}}^{-\dep''}
  \min\set{\abs{\frac rq
      \left(k_1+2^{j\frac{\scp-1}{2}}\tfrac{\xi_{2}}{\xi_{1}}\right)}^{-{\dep'}},1}
  \max{\set{\abs{1+\frac{2\omega_1}{2^{-j\scp/2}\xi_1}}^{\dep''},1}}
\end{multline*}
Applied to \eqref{eq:pf_estimateR2} this yields
\begin{multline} I_1 \le \sum_{j\in J_1}
  \min\set{|q2^{-j\scp/2}\xi_{1}|^{\vap-2\dep},1}\min
  \set{|q'2^{-j\scp/2}\xi_{1}|^{-\dep},1} \\ \cdot \min
  \set{\abs{q(2^{-j\scp/2}\xi_{1} + 2\omega_{1})}^{\vap-2\dep},1}
  \min\set{ \abs{q'(2^{-j\scp/2}\xi_{1} + 2\omega_{1})}^{-\dep},1} \\
  \sum_{k_1 \in \Z} 2^{\dep''}
  \abs{\frac{r}{q}\frac{2\omega_2}{2^{-j\scp/2}\xi_1}}^{-\dep''}
  \min\set{\abs{\frac rq
      \left(k_1+2^{j\frac{\scp-1}{2}}\tfrac{\xi_{2}}{\xi_{1}}\right)}^{-{\dep'}},1}
  \max{\set{\abs{1+\frac{2\omega_1}{2^{-j\scp/2}\xi_1}}^{\dep''},1}}
  \\ \sum_{k_2 \in \Z} \min\set{\abs{\frac sq
      \left(k_2+2^{j\frac{\scp-1}{2}}\tfrac{\xi_{3}}{\xi_{1}}\right)}^{-{\dep}},1}.
\end{multline}
Hence, by estimate~(\ref{eq:sk2}), 
\begin{multline} 
  I_1 \le 2^{\dep''} \frac{q^2}{rs}C(\dep)C(\dep')
  \normsmall[\infty]{2\tfrac{r}{q}w}^{-\dep''} \sum_{j\in J_1}
  \min\set{|q2^{-j\scp/2}\xi_{1}|^{\vap-2\dep},1}\min
  \set{|q'2^{-j\scp/2}\xi_{1}|^{-\dep},1} \\ \cdot \min
  \set{\abs{q(2^{-j\scp/2}\xi_{1} + 2\omega_{1})}^{\vap-2\dep},1}
  \min\set{ \abs{q'(2^{-j\scp/2}\xi_{1} + 2\omega_{1})}^{-\dep},1} \\
  \abs{{2^{-j\scp/2}\xi_1}}^{\dep''}
  \max{\set{\abs{1+\frac{2\omega_1}{2^{-j\scp/2}\xi_1}}^{\dep''},1}}.
  \label{eq:case2}
\end{multline}

We further split Case~II into the following two subcases: $1 \le
|1+\frac{2\omega_{1}}{2^{-j\scp/2} \xi_{1}}|$ and $1 >
|1+\frac{2\omega_{1}}{2^{-j\scp/2} \xi_{1}}|$. Now, in case $1 \le
|1+\frac{2\omega_{1}}{2^{-j\scp/2} \xi_{1}}|$, then obviously
\[
\left|{2^{-j\scp/2}\xi_1} \right|^{\dep''}\max\left\{
  1,\left|1+\frac{2\omega_1}{2^{-j\scp/2}\xi_1}
  \right|^{\dep''}\right\} \le \left|2^{-j\scp/2}\xi_1+2\omega_1
\right|^{\dep''},
\]
which used in \eqref{eq:case2} yields
\begin{multline*}
  I_1 \le \frac{q^2}{rs}C(\dep)C(\dep')
  \normsmall[\infty]{\tfrac{r}{q}w}^{-\dep''} \sum_{j\in J_1}
  \min\set{|q2^{-j\scp/2}\xi_{1}|^{\vap-2\dep},1}\min
  \set{|q'2^{-j\scp/2}\xi_{1}|^{-\dep},1} \\ \cdot \min
  \set{\abs{q(2^{-j\scp/2}\xi_{1} + 2\omega_{1})}^{\vap-2\dep},1}
  \min\set{ \abs{q'(2^{-j\scp/2}\xi_{1} + 2\omega_{1})}^{-\dep},1}
  \left|2^{-j\scp/2}\xi_1+2\omega_1 \right|^{\dep''},
\end{multline*}
Hence, by inequality \eqref{eq:ineq3} with $\iota=\vap-\dep$, \ie
\[ \min \set{\abs{q(2^{-j\scp/2}\xi_{1} +
    2\omega_{1})}^{\vap-2\dep},1} \min\set{
  \abs{q'(2^{-j\scp/2}\xi_{1} + 2\omega_{1})}^{-\dep},1}
\left|2^{-j\scp/2}\xi_1+2\omega_1 \right|^{\dep''} \le
(q')^{-\dep''},\] 
we arrive at
\begin{align}
  I_1 &\le \frac{q^2}{rs}C(\dep)C(\dep')
  \normsmall[\infty]{\tfrac{q'r}{q}w}^{-\dep''} \sum_{j\in J_1}
  \min\set{|q2^{-j\scp/2}\xi_{1}|^{\vap-2\dep},1}\min
  \set{|q'2^{-j\scp/2}\xi_{1}|^{-\dep},1}\nonumber \\
  &\le \frac{q^2}{rs}C(\dep)C(\dep')
  \normsmall[\infty]{\tfrac{q'r}{q}w}^{-\dep''} \Bigl(
  \ceil{\frac{2}{\alpha} \log_{2}\Bigl(\frac{q}{q'}\Bigr)}
  +\frac{1}{1-2^{-\vap+2\dep}}+1\Bigr).
  \label{eq:I1-case2a}
\end{align}

On the other hand, if $1 \ge
|1+\frac{2\omega_{1}}{2^{-j\scp/2}\xi_{1}}|$, then,
\[
\min \set{ \absbig{q(2^{-j\tfrac{\scp}{2}}\xi_{1} + 2\omega_{1})}^{{\vap -
    \dep}},1} \min\set{ \absbig{q'(2^{-j\tfrac{\scp}{2}}\xi_{1} +
  2\omega_{1})}^{-\dep},1} \max\set{
  \absbig{1+\tfrac{2\omega_1}{a_j\xi_1}}^{\dep''},1} \le 1,
\]
for all $j \ge 0$. Hence from \eqref{eq:case2}, by employing
inequality \eqref{eq:ineq4}, we arrive at
\begin{align} 
  I_1 &\le \frac{q^2}{rs}C(\dep)C(\dep')
  \normsmall[\infty]{\tfrac{r}{q}w}^{-\dep''} \sum_{j\in J_1}
  \min\set{|q2^{-j\tfrac{\scp}{2}}\xi_{1}|^{\vap-2\dep},1}\min
  \set{|q'2^{-j\tfrac{\scp}{2}}\xi_{1}|^{-\dep},1}
  \abs{{2^{-j\tfrac{\scp}{2}}\xi_1}}^{\dep''}\nonumber \\
  &\le\frac{q^2}{rs}C(\dep)C(\dep')
  \normsmall[\infty]{\tfrac{r}{q}w}^{-\dep''} \sum_{j\in J_1}
  (q')^{-\dep''}\min\set{|q2^{-j\tfrac{\scp}{2}}\xi_{1}|^{\vap-2\dep+\dep''},1}\min
  \set{|q'2^{-j\tfrac{\scp}{2}}\xi_{1}|^{-\dep'},1}\nonumber \\
  &\le \frac{q^2}{rs}C(\dep)C(\dep')
  \normsmall[\infty]{\tfrac{q'r}{q}w}^{-\dep''} \Bigl(
  \ceil{\frac{2}{\alpha} \log_{2}\Bigl(\frac{q}{q'}\Bigr)}
  +\frac{1}{1-2^{-\vap+2\dep-\dep''}}+1\Bigr).\label{eq:I1-case2b}
\end{align}

\paragraph{Case III}  This case is similar to Case~II
and the estimates from Case~II hold with the obvious modifications. We
therefore skip the proof.

\medskip

We next estimate $I_{2}$. First, notice that the inequality
\eqref{eq:pf_estimateR2} still holds for $I_2$ with the index set
$J_1$ replaced by $J_2$. Therefore, we obviously have
\begin{multline*} I_2 \le \sum_{j\in J_2}
  \min\set{|q2^{-j\scp/2}\xi_{1}|^{\vap-2\dep},1}\min
  \set{|q'2^{-j\scp/2}\xi_{1}|^{-\dep},1} \\
  \sum_{k_1 \in \Z} \min\set{\abs{\frac rq
      \left(k_1+2^{j\frac{\scp-1}{2}}\tfrac{\xi_{2}}{\xi_{1}}\right)}^{-{\dep}},1}
  \sum_{k_2 \in \Z} \min\set{\abs{\frac sq
      \left(k_2+2^{j\frac{\scp-1}{2}}\tfrac{\xi_{3}}{\xi_{1}}\right)}^{-{\dep}},1},
\end{multline*}
by \eqref{eq:sk2},
\begin{align} I_1 &\le \frac{q^2}{rs} C(\dep)^2 \sum_{j\in J_2}
  \min\set{|q2^{-j\scp/2}\xi_{1}|^{\vap-2\dep},1}\min
  \set{|q'2^{-j\scp/2}\xi_{1}|^{-\dep},1}  \nonumber  \\
  &\le \frac{q^2}{rs} C(\dep)^2 \norm[\infty]{q'
    \omega}^{-\dep}. \label{eq:I2}
\end{align}

\smallskip

Summarising, using \eqref{eq:Gamma-esssupI}, \eqref{eq:I1-case1}, and \eqref{eq:I2}, we have that
 \begin{align*} \Gamma(2\omega)
\le  \frac{q^2}{rs} \frac{C(\dep)^2}{\norm[\infty]{q'\omega}^{\dep}}
\Bigl( \ceil{\log_{2}\Bigl(\frac{q}{q'}\Bigr)}
+\frac{1}{1-2^{-\vap+2\dep}}+\frac{1}{1-2^{-\dep}}\Bigr) + \frac{q^2}{rs} \frac{C(\dep)^2}{\norm[\infty]{q'
       \omega}^{\dep}} \frac{1}{1-2^{-\dep}}, 
 \end{align*}
 whenever $\norm[\infty]{\omega}=\abs{\omega_1}$, and by
 \eqref{eq:I1-case2a}, \eqref{eq:I1-case2b}, and \eqref{eq:I2},
 \begin{align*}
   \Gamma(2\omega) &\le \frac{q^2}{rs}\frac{C(\dep)C(\dep')}{
     \normsmall[\infty]{\tfrac{q'\min\{r,s\}}{q}w}^{\dep''}} \Bigl(
   \frac{4}{\alpha}\ceil{ \log_{2}\Bigl(\frac{q}{q'}\Bigr)}
   +\frac{1}{1-2^{-\vap+2\dep}}+\frac{1}{1-2^{-\vap+2\dep-\dep''}}+2\Bigr)\\
   &\phantom{\le} + \frac{q^2}{rs} \frac{C(\dep)^2}{\norm[\infty]{q'
       \omega}^{\dep}},
 \end{align*}
 otherwise. We are now ready to prove the claimed estimate for $R(c)$.
 Define
 \[ \mathcal{Q} = \setprop{m \in \Z^3 }{ |m_1| > |m_2| \text{ and }
   \abs{m_1}>\abs{m_3}}, \] 
and
\[ \tilde{\mathcal{Q}} = \setprop{m \in \Z^3 }{ c_1^{-1}|m_1| >
  c_2^{-1}|m_2| \text{ and } c_1^{-1}|m_1| > c_2^{-1}|m_3| }.
\]
If $m \in \tilde{\mathcal{Q}}$, that is, if $c_1^{-1}|m_1| >
c_2^{-1}|m_2|$ and $c_1^{-1}|m_1| > c_2^{-1}|m_2|$, then
\begin{multline*}
  \Gamma(\pm M_c^{-1}m) \le \frac{q^2}{rs}
  \frac{C(\dep)^2}{\norm[\infty]{m}^{\dep}} \left(\frac{2c_1}{q'}
  \right)^\dep \Bigl( \ceil{\log_{2}\Bigl(\frac{q}{q'}\Bigr)}
  +\frac{1}{1-2^{-\vap+2\dep}}+\frac{2}{1-2^{-\dep}}\Bigr) \\= (T_1+T_3)
  \norm[\infty]{m}^{-\dep}
\end{multline*}
If on the other hand $m\in\tilde{\mathcal{Q}}^c\setminus \{0\}$, that
is, if $c_1^{-1}|m_1| \le c_2^{-1}|m_2|$ or $c_1^{-1}|m_1| \le
c_2^{-1}|m_3|$ with $m \neq 0$, then
\begin{multline*}
  \Gamma(\pm M_c^{-1}m) \le \frac{q^2}{rs}\frac{C(\dep)C(\dep')}{
    \normsmall[\infty]{m}^{\dep''}} \left(\frac{2qc_2}{q'r}
  \right)^{\dep''} \Bigl( 2\ceil{ \log_{2}\Bigl(\frac{q}{q'}\Bigr)}
  +\frac{1}{1-2^{-\vap+2\dep}}+\frac{1}{1-2^{-\dep}}+\\\frac{1}{1-2^{-\vap+2\dep-\dep''}}+\frac{1}{1-2^{-\dep'}}\Bigr)
  + \frac{q^2}{rs} \frac{C(\dep)^2}{\norm[\infty]{ m}^{\dep}}
  \left(\frac{2c_1}{q'} \right)^{\dep} \frac{1}{1-2^{-\dep}}=(T_2+T_3)
  \,\normsmall[\infty]{m}^{\dep''},
\end{multline*}
Therefore, we obtain
\begin{align} \nonumber
  R(c) &= \sum_{m \in \Z^3\setminus\{0\}}\left(
    \Gamma(M_c^{-1}m)\,\Gamma(-M_c^{-1}m)\right)^{1/2}\\ \label{eq:Rc}
  &\le \left(\sum_{m \in \tilde{{\mathcal Q}}}
    T_1\|m\|^{-\dep}_{\infty}+T_3\|m\|^{-\dep}_{\infty}\right) +
  \left(\sum_{m \in \tilde{{\mathcal Q}}^c\backslash\{0\}}
    T_2\|m\|^{-\dep''}_{\infty}+T_3\|m\|^{-\dep}_{\infty}\right)
\end{align}
Notice that, since $\tilde{{\mathcal Q}} \subset {\mathcal Q}$,
\[
\sum_{m \in \tilde{{\mathcal Q}}} \|m\|^{-\dep} \le \sum_{m \in
  {{\mathcal Q}}} \|m\|^{-\dep}.
\]
Also, we have
\[
\sum_{m \in \tilde{{\mathcal Q}}^c\backslash\{0\}} \|m\|^{-\dep''} \le
3\min{\set{\ceil{\frac{c_1}{c_2}},2}} \sum_{m \in {{\mathcal
      Q}^c\backslash\{0\}}} \|m\|^{-\dep''}.
\]
Therefore, \eqref{eq:Rc} can be continued by
\begin{equation*}
  R(c) \le T_3 \sum_{m \in \Z^3\setminus\{0\}}\|m\|_{\infty}^{-\dep} + T_1 \sum_{m \in {\mathcal Q}}
\|m\|_{\infty}^{-\dep}
  +3\min{\set{\ceil{\frac{c_1}{c_2}},2}} T_2 \sum_{m \in {{\mathcal Q}^c\backslash\{0\}}} \|m\|^{-\dep''}.
\end{equation*}
To provide an explicit estimate for the upper bound of $R(c)$, we
compute $\sum_{m \in {\mathcal Q}} \|m\|_{\infty}^{-\dep}$ and
$\sum_{m \in {\mathcal Q}^c\backslash} \|m\|_{\infty}^{-\dep}$ as
follows:
\begin{align*}
  \sum_{m \in \Z^3\setminus\{0\}} \norm[\infty]{m}^{-\dep} &=
  \sum_{d=1}^{\infty}(24d^2+2) d^{-\dep} = 24 \zeta(\dep-2) + 2
  \zeta(\dep)
\end{align*}
where $(2d+1)^3-(2d+1)^3=24d^2+2$ is the number of lattice points in
$\Z^3$ at distance $d$ (in max-norm) from origo. Further,
\begin{align*}
  \sum_{m \in {\mathcal Q}} \norm[\infty]{m}^{-\dep} &=
  2\sum_{m_1=1}^{\infty}(2m_1-1)^2 m_1^{-\dep} = \sum_{m_1=1}^{\infty}
  (8m_1^{2-\dep} -8m_1^{1-\dep}+2 m_1^{-\dep}) \\ &= 8 \zeta(\dep-2) -
  4 \zeta(\dep-1) + 2 \zeta(\dep)
\end{align*}
and
\begin{align*}
  \sum_{m \in {\mathcal Q}^c\setminus \{0\}} \norm[\infty]{m}^{-\dep}
  &= 24 \zeta(\dep-2) + 2 \zeta(\dep) - \bigl(8 \zeta(\dep-2) - 4
  \zeta(\dep-1) + 2 \zeta(\dep) \bigr) \\ &= 16 \zeta(\dep-2) - 4
  \zeta(\dep-1),
\end{align*}
which completes the proof.

\end{document}